\documentclass[10pt,reqno]{amsart}

\usepackage{amssymb}
\usepackage{amsmath}
\usepackage{amsfonts, dsfont}

\usepackage{amsthm} 
\usepackage[utf8]{inputenc} 
\usepackage{enumitem} 

\usepackage{graphicx}
\usepackage{mathrsfs}
\usepackage[utf8]{inputenc}
\usepackage{anyfontsize}
\usepackage{tikz}
\usepackage{color}
\usepackage{comment}
\usetikzlibrary{
	3d,decorations.markings,angles,quotes,
	calc,intersections,arrows.meta,patterns,
	decorations.pathmorphing,
	decorations.fractals
}

\usepackage[colorlinks=true, pdfstartview=FitV, linkcolor=blue, citecolor=red, urlcolor=blue]{hyperref} 

\DeclareGraphicsExtensions{.pdf,.png,.jpg}

\theoremstyle{theorem} 
\newtheorem{thm}{Theorem}[section]

\newtheorem{lemma}[thm]{Lemma}
\newtheorem{prop}[thm]{Proposition}
\newtheorem{assumption}[thm]{Assumption}

\theoremstyle{definition}   
\newtheorem{defn}[thm]{Definition}

\theoremstyle{remark}  
\newtheorem{remark}[thm]{Remark}
\newtheorem{example}[thm]{Example}

\newcommand\bR{\mathbb{R}}

\newcommand\bS{\mathbb{S}}

\newcommand\bE{\mathbb{E}}
\newcommand\bN{\mathbb{N}}

\newcommand\bP{\mathbb{P}}

\newcommand\cA{\mathcal{A}}
\newcommand\cB{\mathcal{B}}

\newcommand\cD{\mathcal{D}}

\newcommand\cL{\mathcal{L}}
\newcommand\cP{\mathcal{P}}

\newcommand\cM{\mathcal{M}}

\newcommand\cO{\mathcal{O}}

\newcommand\cV{\mathcal{V}}

\newcommand{\mysection}[1]{\section{#1}
\setcounter{equation}{0}}

\begin{document}

\title[]
{Green's function estimates for time measurable parabolic operators on polyhedrons and  polyhedral cones}
 
\thanks{The first author was supported by the National Research Foundation of Korea(NRF) grant
funded by the Korea government(MSIT) (RS-2025-00556160).
The third author has been supported by a KIAS Individual Grant (MG095801) at Korea Institute for Advanced Study, and by the National Research Foundation of Korea (NRF-2019R1A5A1028324).}

\author{Kyeong-Hun Kim}
\address{Kyeong-Hun Kim, Department of Mathematics, Korea University, Anam-ro 145, Sungbuk-gu, Seoul, 02841, Republic of Korea}
\email{kyeonghun@korea.ac.kr}

\author{Kijung Lee}
\address{Kijung Lee, Department of Mathematics, Ajou University, Worldcup-ro 206, Yeongtong-gu, Suwon, 16499, Republic of Korea}
\email{kijung@ajou.ac.kr}

\author{jinsol Seo}
\address{Jinsol Seo, School of Mathematics, Korea Institute for Advanced Study, Hoegi-ro 85, Dongdaemun-gu, Seoul, 02455, Republic of Korea}
\email{seo9401@kias.re.kr}

\subjclass[2020]{34B27, 35K08, 35K20, 60H15}

\keywords{Green's function estimate, polyhedron, polyhedral cone, parabolic operator, mixed weight}

\begin{abstract}

We provide  Green's function estimates for parabolic operators on polyhedrons and  polyhedral cones in $\bR^3$.
 These estimates incorporate mixed weights, which include appropriate powers of the distances to the vertices, the edges, and the boundary of the domains. The allowable ranges for the weight parameters are explicitly determined by the geometry of the domains.
\end{abstract}

\maketitle

\mysection{Introduction}\label{sec:Introduction}

In this article we provide  Green's function  estimates for  the parabolic operator 
\begin{equation*}
\cL=\frac{\partial}{\partial t}-\sum^{3}_{i,j=1}a_{ij}(t)D_{ij}
\end{equation*}
subject to the zero Dirichlet boundary condition on polyhedrons  and infinite polyhedral cones  in $\bR^3$.  Here, the coefficients $a_{ij}$s are assumed to be  measurable in $t$,  and satisfy the  uniform parabolicity condition \eqref{uniform parabolicity}.

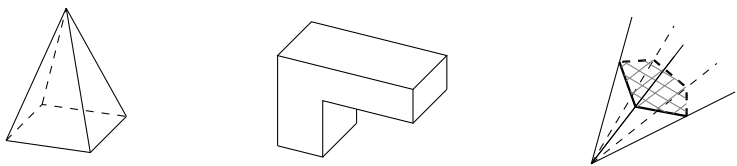
\begin{figure}[h]
\begin{tikzpicture}

\begin{scope}[shift={(-1.5,-0.1)},scale=0.4]
	
	\coordinate (A1) at (0,0);
	\coordinate (B1) at (2,-0.4);
	\coordinate (C1) at (1.2,-2);
	\coordinate (D1) at (-0.8,-1.6);
	
	\coordinate (A2) at (-0.2,-2);
	\coordinate (B2) at (2.6,-2.4);
	\coordinate (C2) at (1.4,-3.6);
	\coordinate (D2) at (-1.4,-3.2);
	
	\coordinate (P) at (0.6,1.2);
	
	\draw[dashed] (A2)--(P);
	\draw (B2)--(P);
	\draw (C2)--(P);
	\draw (D2)--(P);
	\draw (D2) -- (C2) -- (B2);
	\draw[dashed] (D2) -- (A2) -- (B2);
\end{scope}
	
	\begin{scope}[shift={(2,0.2)}, scale=0.6]
	\coordinate (A1) at (0,0);
	\coordinate (B1) at (3,-0.75);
	\coordinate (C1) at (2.25,-1.5);
	\coordinate (D1) at (-0.75,-0.75);
	\coordinate (E1) at (1,-0.25);
	\coordinate (F1) at (0.25,-1);
	
	\coordinate (A2) at (0,-0.75);
	\coordinate (B2) at (3,-1.5);
	\coordinate (C2) at (2.25,-2.25);
	\coordinate (D2) at (-0.75,-1.5);
	\coordinate (E2) at (1,-1.94);
	\coordinate (F2) at (0.25,-1.75);
	
	\coordinate (A3) at (0,-2);
	\coordinate (B3) at (3,-2.75);
	\coordinate (C3) at (2.25,-3.5);
	\coordinate (D3) at (-0.75,-2.75);
	\coordinate (E3) at (1,-2.25);
	\coordinate (F3) at (0.25,-3);
	
	\draw (A1) -- (B1) -- (C1) -- (D1) -- cycle;
	\draw (B1) -- (B2);
	\draw (C1) -- (C2);
	\draw (D1) -- (D3);
	\draw (E2) -- (E3);
	\draw (D3) -- (F3) -- (F2) -- (C2) -- (B2);	
	\draw (F3) -- (E3);
	
		\end{scope}
	
			\begin{scope}[shift={(6.5,-.85)}, scale=0.6]
		\clip (-2.2,-2)--(2.4,-2)--(2.4,2.4)--(-2.4,2.4);		
		\begin{scope}[shift={(-1.3,-1.4)}, scale=1.2]
			
			\begin{scope}[scale=0.9]
				\clip (0,0) circle (3.3) ;
				\clip (-3,-3) -- (3,-3) -- (3,3) -- (-3,3);
				\draw (0,0) -- (4.4,2.2); 
				\draw[dashed] (0,0) -- (5,4);
				\draw[dashed] (0,0) -- (3,5);
				\draw (0,0) -- (1.3,4.7);
				\draw (0,0) -- (1.70,2.20);
			\end{scope}
			
			
			\begin{scope}[scale=1.6]
				\draw[dashed,line width=0.3mm] (1.1,0.55) -- (1.1, 0.88) -- (0.72, 1.2) -- (0.325, 1.175);
			\end{scope}
			
			\begin{scope}[scale=1.6]
				\draw[line width=0.3mm] (0.325, 1.175) -- (0.51, 0.66) -- (1.1,0.55);
			\end{scope}

			\begin{scope}[scale=1.6]
				\clip (1.1,0.55)-- (1.1, 0.88) -- (0.72, 1.2) -- (0.325, 1.175) -- (0.51, 0.66) -- (1.1,0.55);
				
										\foreach \i in {-4,-3.85,...,2}
										{\draw[gray] (-2.8,\i+3.3)--(2.8,\i);}
										
										\foreach \i in {-2,-1.8,...,4}
										{\draw[gray] (-2.8,\i-4.3)--(2.8,\i);}
				
			\end{scope}

			\draw (0,0) -- (1.70,2.20);
		\end{scope}

	\end{scope}
	
\end{tikzpicture}

\end{figure}

A polyhedron, which we denote as $\cP$, is a bounded three dimensional object whose boundary consists of a finite number of polygonal faces, straight edges, and vertices.
A polyhedral cone $\cD$, on the other hand, is an infinite three-dimensional object whose boundary, $\partial \cD$, consists of a vertex at 
${\bf{0}}$, infinite straight edges, and flat faces.  See Assumptions \ref{ass 1} and \ref{ass 2} for  mathematical definitions.

   Green's function is characterized by three factors;  the operator, the domain, and the zero Dirichlet boundary condition. Green's function $G=G(t,s,x,y)$ for a parabolic operator $L$ given on   domain $\cO\subset \bR^d$   is supposed to be \emph{fundamental} for the operator $L$, meaning that for any fixed $(s,y)\in \bR \times \cO$ it satisfies 
\begin{equation}\label{Green.def.}
L G(\cdot,s,\cdot,y)=\delta(\cdot-s)\delta(\cdot-y)  ; \quad  G|_{s>t}=0, \,\, G|_{x\in \partial \cO}=0
\end{equation}
in the sense of distributions in $\bR\times\cO$ (see Lemma \ref{241226452}).  In a probabilistic point of view, Green's function is the transition probability density of a certain  stochastic process which is killed once it reaches the boundary of the domain. For detail, see Remark \ref{probability}.

Pointwise estimate of Green's function has been an intriguing topic in mathematics since Nash's work \cite{Nash}. See e.g.  \cite{Cho,KimXu2022} for a brief history of the study of Green's functions for both elliptic and parabolic operators.  We refer to \cite{Aronson1967, Aronson1968, Davies87, Davies, Fabes, KimXu2022, Semenov} for Gaussian bounds of Green's functions for parabolic operators on general domains and we  refer to \cite{Cho,Davies} for the Green's function estimates on $C^{1,\alpha}$-domains in $\bR^d$, $\alpha\in (0,1)$, of the type
\begin{equation*}
 G(t,s,x,y) \leq N 1_{t>s} \left(1 \wedge \frac{d(x)}{\sqrt{t-s}}\right)\left(1 \wedge \frac{d(y)}{\sqrt{t-s}}\right) \Gamma_{\sigma}(t-s,x-y),
\end{equation*}
where $\sigma,N>0$ are constants, $\Gamma_{\sigma}(t,x):=|t|^{-d/2} e^{\frac{-\sigma|x|^2}{t}}$, and $d(x)$ denotes the distance from $x$ to the boundary of a domain. We also refer to \cite{Green, Kozlov_Nazarov_2014, Koz1991} for  estimates of the Green's function on smooth cones  with  a vertex at ${\bf{0}}$. In these works  the estimates of the Green's function account for the distance to the boundary and the distance to the vertex ${\bf{0}}$.

Our motivation for the  Green's function estimates on polyhedral cones and  polyhedrons is to develop a nice  Sobolev regularity theory for partial differential equations (PDEs) and stochastic partial differential equations (SPDEs) on these types of non-smooth domains. 
The regularity of the solutions to PDEs and SPDEs  can be significantly influenced by singularities of
the boundary, such as corners and edges. This influence is already known for a
long time from the analysis of deterministic equations, see, e.g., \cite{Dau1988, Gri1985,Gri1992, KozMazRos1997,MaRo2010,Naz2001}.
The singularities of the boundary  are  more critical for SPDEs due to the interaction with random noises near the boundary (cf. \cite{CKL2019+,Fla1990,Kim2004}).
Actually even on $C^{\infty}$ domains, the second and higher derivatives of solutions to SPDEs blow up near the boundary unless certain compatibility condition is fulfilled (cf. \cite{Fla1990,Krylov1994}).
To measure the blow up of derivatives of the solutions to PDEs and SPDEs, it is necessary to introduce suitable weighted Sobolev spaces.
The Green's function appears in the integral representation of solutions and the behavior of solutions  near the boundary can be understood by studying the Green's function (cf. \cite{CKL2019+,ConicSPDE}). In particular, our Green's function estimates must involve the distances to the edges, vertices, and boundary in order for us to capture the singularities of the solutions.

Now we introduce our Green's function estimates for the polyhedral cone $\cD \subset \bR^3$.  Let $V=\bf{0}$ denote the vertex,  $E_i$, $ i=1,2,\cdots, N_0$,  denote the edges of $\cD$, and $\hat{E}:=\bigcup_{i=1}^{N_0}E_i$. For any $r>0, x \in \cD$, and $\Lambda^{+}=(\lambda^{+}_0, \lambda^{+}_1,\cdots,\lambda^{+}_{N_0})$, $\Lambda^{-}=(\lambda^{-}_0, \lambda^{-}_1,\cdots,\lambda^{-}_{N_0})\in (\bR_+)^{N_0+1}$, we define    
$\mathbf{I}(x,r;\Lambda^{+})$ and $\mathbf{I}(x,r;\Lambda^{-})$  as 
\begin{equation*}
\mathbf{I}(x,r;\Lambda^{\pm})=\left(\frac{d(x,V)}{r}\wedge 1\right)^{\lambda^{\pm}_0}\left(\prod_{i=1}^{N_0}\left(\frac{d\big(x,E_i\big)\wedge r}{d\big(x,V\big)\wedge r}\right)^{\lambda^{\pm}_{i}}\right) \frac{d(x,\partial\cD)\wedge r}{d(x,\hat{E})\wedge r},
\end{equation*}
where $d(x,V)$, $ d(x,E_i)$, $ d(x,\partial D)$, and $d(x, \hat{E})$ represent the distances from $x$ to the vertex, each edge $E_i$, the boundary $\partial \cD$, and the union of all edges $\hat{E}$, respectively. 
Note that, since $V\in E_i \subset \hat{E} \subset \partial \cD$,  for any $x \in \cD$ we have
$$
d(x,\partial \cD)\leq d(x, \hat{E})\leq d(x, E_i)\leq d(x, V).
$$
Our Green's function  estimate on the polyhedral cone $\cD$ has the following form:
\begin{eqnarray}
 \nonumber
G(t,s,x,y) &\leq& N\, 1_{t>s} \, \mathbf{ I}\big(x,\sqrt{t-s};\Lambda^+\big) \, \mathbf{I}\big(y,\sqrt{t-s};\Lambda^-\big) \\
&&\,\, \times \frac{1}{(t-s)^{3/2}}e^{-\sigma\frac{|x-y|^2}{t-s}}, \label{eqn main}
\end{eqnarray}
where $\sigma>0$ is a constant depending only on the uniform parabolicity of $\cL$, and $N$ is a constant that depends also on $\Lambda^{\pm}$.  The admissible ranges of $\lambda^{+}_i$ and $\lambda^{-}_i$, $i=0,1,\cdots,N_0$, are explicitly determined based on the operator $\cL$ and the geometry of the domain $\cD$.  Specifically, when $\cL=\partial_t-\Delta$, the permissible ranges of  $\lambda^{\pm}_i$ are given by
$$
\lambda^{\pm}_i \in \left(0, \frac{\pi}{\kappa_i}\right), \quad i=1,2,\cdots, N_0,
$$
where $\kappa_i$ is the inner angle between the two planes that form the edge $E_i$. On the other hand, in this case the permissible ranges of $\lambda^{\pm}_0$ are
$$
\lambda^{\pm}_0\in \left(0, -\frac{1}{2}+\sqrt{\mathcal{E}_0+\frac{1}{4}} \right),
$$
where $\mathcal{E}_0$ is the first eigenvalue  of  the Laplace Beltrami operator on $\cM:= \cD\cap \bS^2$ with Dirichlet boundary condition (see \eqref{250121350} and Remark \ref{250220324}).

The product $\mathbf{I}\big(x,\sqrt{t-s};\Lambda^+\big) \cdot \mathbf{I}\big(y,\sqrt{t-s};\Lambda^-\big)$ characterizes the behavior of the Green's function, particularly its decay rates, near the vertex, the edges, and the boundary; see Remarks \ref{remark decay} and \ref{remark decay 2} for details.

Our estimate for the Green's function on polyhedron $\cP$ also takes the form of \eqref{eqn main}. However, since there are multiple vertices, say $V_1,V_2,\cdots, V_{M_0}$, we use a different definition for $\mathbf{I}$: for 
	$\Lambda^{\pm}:=\big((\lambda_{v,i}^{\pm})_{1\leq i\leq M_0},\,(\lambda_{e,j}^{\pm})_{1\leq j\leq N_0}\big)$,
\begin{align*}
	&\quad\mathbf{I}(x,r;\Lambda^{\pm})\\
	&:=\left(\,\prod_{i=1}^{M_0}\left(\frac{d(x,V_i)}{r}\wedge 1\right)^{\lambda^{\pm}_{v,i}}\right)\cdot \left(\,\prod_{j=1}^{N_0}\left(\frac{d(x,E_j)\wedge r}{d\big(x,\hat{V}\cap E_j\big)\wedge r}\right)^{\lambda^{\pm}_{e,j}}\right)\cdot \frac{d(x,\partial\cP)\wedge r}{d(x,\hat{E})\wedge r}, \nonumber
\end{align*}
where $\hat{V}:=\{V_1,V_2,\ldots,V_{M_0}\}$ and  $d(x, \hat{V}\cap E_j)$ represents the distance from $x$ to the  vertices on the edge $E_j$. The admissible ranges of $\lambda^{\pm}_{v,i}$ and $\lambda^{\pm}_{e,j}$ are also explicitly determined by $\cP$ and $\cL$.

The proofs for our weighted estimates of Green's function are more involved than those for smooth cones  because polyhedral cones include both vertices and edges, and polyhedrons have multiple vertices in addition to edges. As the Green's function estimate for infinite smooth cones without edges in \cite{Green} was the key for us to develop a regularity theory in \cite{ConicSPDE}, in a subsequent work our Green's function estimates  will be essential for us to build a regularity theory of (stochastic) parabolic  equations on   polyhedral cones and polyhedrons.  

\vspace{2mm}

We finish the introduction with notations used in this article.

\begin{itemize}
\item  $A:=B$ (or $B=:A$) means that $A$ is defined by $B$.
\item  For a measure space $(A, \cA, \mu)$, a Banach space $B$ and $p\in[1,\infty)$, we write $L_p(A,\cA, \mu;B)$ for the collection of all $B$-valued $\bar{\cA}$-measurable functions $f$ such that
$$
\|f\|_{L_p(A,\cA,\mu;B)}:=\left(\int_{A} \lVert f\rVert^p_{B} \,d\mu\right)^{1/p}<\infty.
$$
Here, $\bar{\cA}$ is the completion of $\cA$ with respect to $\mu$.  The Borel $\sigma$-algebra on a topological space $E$ is denoted by $\cB(E)$. We will drop $\cA$ or $\mu$ or even $B$ in $L_p(A,\cA, \mu;B)$   when they   are obvious in the context.

\item $\bR^d$ stands for the $d$-dimensional Euclidean space of points $x=(x_1,\cdots, x_d)$.  $B_r(x):=\{y\in \bR^d: |x-y|<r\}$,
 $\bR^d_+:=\{x=(x_1,\cdots,x_d): x_d>0\}$,  $\bR_+=\bR^1_+$, ${\bf{0}}=(0,\cdots,0)$, and $\bS^{d-1}:=\{x\in \bR^d: |x|=1\}$.
 
 \item For $x$, $y$ in $\bR^d$,  $x\cdot y :=\sum^d_{i=1}x_iy_i$ denotes the standard inner product.
 
 \item  $\bN=\{1,2,3,\cdots\}$ and $\bN_0:=\{0\}\cup \bN$.

\item  For  function $u(t,x)$
$$
u_t=\frac{\partial u}{\partial t}, \quad u_{x_i}=D_i u:=\frac{\partial u}{\partial x_i},  \quad D_{ij}u=D_jD_i u,
$$
and  $\nabla u:=\nabla_x u=(D_1u,\cdots, D_d u)$.

 \item For  a domain $\cO \subset \bR^d$, $B^{\cO}_r(x):=B_r(x)\cap \cO$ and  $Q^{\cO}_r(t,x):=(t-r^2,t]\times B^{\cO}_r(x)$.

\item For a set $\mathcal{S}$ in $\bR^d$, $\partial \mathcal{S}$ denotes the boundary of $\mathcal{S}$ and $\overline{\mathcal{S}}$ denotes the closure of $\mathcal{S}$. Also, $d(x,\mathcal{S})$ stands for the distance between a point $x$ and  $\mathcal{S}$ for $\mathcal{S}\ne\emptyset$.

 \item For a function $\zeta$ defined on $\bR^d$ or $\bR\times \bR^d$, we define 
$$
\textrm{supp}(\zeta):=\overline{\{x\in \bR^d:\zeta(x)\ne 0\}}
 \quad\text{or}\quad
\overline{\{(t,x)\in \bR\times\bR^d:\zeta(t,x)\ne 0\}}
$$
respectively.

 \item For a set $\mathcal{S}$ in $\bR^d$ or $\bR\times \bR^d$, $C^{\infty}_c(\mathcal{S})$ denotes the space of infinitely differentiable functions with compact supports in $\mathcal{S}$. The space $C^n_c(\mathcal{S})$, $n\in\mathbb{N}_0$ denotes the space of $n$-times continuously differentiable functions with compact supports, and $C_c(\mathcal{S}):= C^0_c(\mathcal{S})$.

\item We will use the following spaces of functions defined in a bounded interval $(a,b)$ or $(a,b]$ and a (possibly unbounded) domain $\cO \subset \bR^d$, $d\geq 1$. 
\begin{itemize}
\item[-]
$W_2^{1,0}\big((a,b)\times \cO \big)$ : the Hilbert space with the norm
$$
\qquad \qquad \|u\|_{W_2^{1,0}((a,b)\times \cO)}^2:=\int_{(a,b)\times \cO}|u|^2dxdt+\int_{(a,b)\times \cO}|\nabla u|^2dxdt\,.
$$

\item[-]
$\mathcal{V}_2\big((a,b]\times \cO \big)$ : the Banach space consisting of  functions $u\in W_2^{1,0}((a,b)\times \cO)$ such that 
\begin{equation*}
\qquad \qquad \qquad \|u\|_{\mathcal{V}_2((a,b]\times \cO)}:=\underset{t\in (a,b]}{\mathrm{ess\,sup}}\,\|u(t,\cdot)\|_{L_2(\cO)} +\|\nabla u\|_{L_2((a,b)\times \cO;\bR^d)}<\infty\,.\nonumber
\end{equation*}

\item[-]
$\cV_2^{1,0}\big((a,b]\times \cO\big):=\cV_2((a,b]\times \cO)\cap C([a,b];L_2(\cO))$.

\item[-]
$\mathring{W}_2^{1,0}\big((a,b)\times \cO\big)$ : the closure of $C_c^{\infty}([a,b]\times \cO)$ in $W_2^{1,0}((a,b)\times \cO)$.

\item[-] 
$\mathring{\cV}_2\big((a,b]\times \cO\big):=\mathcal{V}_2((a,b]\times \cO)\cap \mathring{W}_2^{1,0}((a,b)\times \cO)$.

\item[-]
$\mathring{\cV}_2^{1,0}\big((a,b]\times \cO \big):=\cV_2^{1,0}((a,b]\times \cO)\cap \mathring{W}_2^{1,0}((a,b)\times \cO)$.
\end{itemize}

\item
Let $-\infty\leq a<+\infty$ and denote  $(a,\infty)_T:=(a,\infty)\cap (-T,T]$.

\begin{itemize}
\item[-]
$\mathcal{V}_2\big((a,\infty)\times \cO \big)$ : the set of all $u\in \bigcap_{T>0}\mathcal{V}_2\big((a,\infty)_T\times \cO \big)$ such that
$$
\|u\|_{\mathcal{V}_2((a,\infty)\times \cO)}:=\sup_{T>0}\|u\|_{\mathcal{V}_2((a,\infty)_T\times \cO)}<\infty.
$$

\item[-] $\mathcal{V}_2^{1,0}\big((a,\infty)\times \cO \big)$, $\mathring{\cV}_2\big((a,\infty)\times \cO \big)$, $\mathring{\cV}_2^{1,0}\big((a,\infty)\times \cO \big)$ are defined in the same manner.
\end{itemize}

\item For $\beta\in(0,1)$ and $A\subset \bR\times \bR^3$ we write $u\in H^{\beta/2,\beta}(A)$ if
\begin{align*}
 |u|_{H^{\beta/2,\beta}(A)}:=\,&
\sup_{A}|u|+\sup_{(\tau,\zeta)\ne (\tau,\eta)\in 
A}\frac{|u(\tau,\zeta)-u(\sigma,\eta)|}{|\zeta-\eta|^{\beta}+|\tau-\sigma|^{\beta/2}} < \infty.
\end{align*} 

\item  Throughout the article, the letter $N$ denotes a finite positive constant which may have different values along the argument  while the dependence  will be informed;  $N=N(a,b,\cdots)$, meaning that  $N$ depends only on the parameters inside the parentheses.
 
 \item  $A\simeq B$ (or   $A\simeq_N B$)  means that there is a constant $N$ independent of $A$ and $B$ such that  $A\leq N B$ and $B\leq N A$.

 \item $a \vee b =\max\{a,b\}$, $a \wedge b =\min\{a,b\}$. 
 
 \item $1_U$ the indicator function on $U$.

\end{itemize}

\mysection{Main Results} \label{sec:Cone}

In this section we present our Green's function estimates for the parabolic operator
\begin{equation}\label{our operator}
\cL=\frac{\partial}{\partial t}-\sum^{3}_{i,j=1}a_{ij}(t)D_{ij}.
\end{equation}
given with  zero Dirichlet boundary condition  on  polyhedral cones $\cD$ and  polyhedrons $\cP$ in $\bR^3$. Here, the coefficients $a_{ij}$, $i,j=1,2,3$,  are real valued measurable functions of $t$, $a^{ij}=a^{ji}$, and satisfy the  uniform parabolicity condition: there exist positive constants $\nu_1\le \nu_2$ such that 
\begin{align}\label{uniform parabolicity}
\nu_1 |\xi|^2\le \sum_{i,j=1}^3a_{ij}(t)\xi_i\xi_j\le \nu_2|\xi|^2, \quad \forall t\in \bR, \, \xi\in \bR^3.
\end{align}

Recall that the Green function $G=G(t,s,x,y)$ for the operator $\cL$ and the polyhedral cone $\cD$ or the polyhedron $\cP$ satisfies \eqref{Green.def.}. In Appendix {\bf{A}}  we collect some auxiliary results of the Green's function including the existence and the uniqueness. In addition, we  provide a probabilistic view on the Green function.

\subsection{Green's function estimate on  polyhedral cones}\label{sec:estimate}

We define  polyhedral cone $\cD\subset \bR^3$   by
\begin{equation*}
\cD=\Big\{x\in\mathbb{R}^3\setminus\{{\bf{0}}\}\,:\,\frac{x}{|x|}\in\cM\Big\},
\end{equation*}
where $\cM$ is a connected open subset of $\bS^2$ such that  $\partial \cM$ has $N_0 \geq 3$ number of vertices, and each edge of $\partial \cM$  is a part of a great circle of $\bS^2$.

The boundary of $\cD$, $\partial \cD$, consists of a vertex $\mathbf{0}$, $N_0$ number of edges, and $N_0$ number of flat faces. 
\begin{figure}[h]
	\begin{tikzpicture}
		\begin{scope}[shift={(10,-.3)}, scale=0.4]
			\begin{scope}[shift={(-1.3,-1.4)}, scale=1.2]
				
				\begin{scope}
					\clip (-3.6,-3) -- (3.6,-3) -- (3.6,3.6) -- (-3.6,3.6);
					\clip (0,0) circle (3.6) ;
					\draw[gray!50] (0,0) circle (2.5);
					\draw (-2.5,0) arc(-180:26.5:2.5);
					\draw (-2.5,0) arc(180:74.5:2.5);

					\draw (0,0) -- (4.4,2.2); 
					\draw[dashed] (0,0) -- (5,4);
					\draw[dashed] (0,0) -- (3,5);
					\draw (0,0) -- (1.3,4.7);
					\draw (0,0) -- (1.70,2.20);
					\draw (-2.5,0) arc(-180:0:2.5 and 0.353*2.5);
					\draw[dashed] (-2.5,0) arc(180:85:2.5 and 0.353*2.5);
					\draw[dashed] (2.5,0) arc(0:50:2.5 and 0.353*2.5);
				\end{scope}

				\begin{scope}[scale=1.6]
					\draw[line width=0.3mm] (0.51, 0.66) arc(27:54.5:0.575 and 1.42) arc(101:87:1.63 and 1.2) arc(60:40:1.42 and 1.39) arc(10.5:-10:1.115 and 0.9) arc(53:74:1.83 and 0.6);
				\end{scope}

				\begin{scope}[scale=1.6]
					\clip (0.51, 0.66) arc(27:54.5:0.575 and 1.42) arc(101:87:1.63 and 1.2) arc(60:40:1.42 and 1.39) arc(10.5:-10:1.115 and 0.9) arc(53:74:1.83 and 0.6);

					\fill[gray!25] (0.51, 0.66) arc(27:54.5:0.575 and 1.42) arc(101:87:1.63 and 1.2) arc(60:40:1.42 and 1.39) arc(10.5:-10:1.115 and 0.9) arc(53:74:1.83 and 0.6);
					
					\draw[gray] (0,1.7) arc(90:0:0.54 and 1.95);
					\draw[gray] (0,1.7) arc(90:0:0.795 and 1.92);
					\draw[gray] (0,1.7) arc(90:0:1.03 and 1.85);
					\draw[gray] (0,1.7) arc(90:0:1.23 and 1.787);
					\draw[gray] (0,1.7) arc(90:0:1.4 and 1.71);
					
					\draw[gray] (-0.15,1.14) arc(-90:0:1.7*0.565 and 0.6*0.565);
					\draw[gray] (-0.15,1.00) arc(-90:0:1.7*0.66 and 0.6*0.66);
					\draw[gray] (-0.15,0.84) arc(-90:0:1.7*0.745 and 0.6*0.745);
					\draw[gray] (-0.15,0.675) arc(-90:0:1.7*0.82 and 0.6*0.82);
					\draw[gray] (-0.15,0.5) arc(-90:0:1.7*0.89 and 0.6*0.89);
				\end{scope}

				\draw (0,0) -- (1.70,2.20);
			\end{scope}

		\end{scope}

	\end{tikzpicture}
\end{figure}

Let $V_{\cM}:=\{p_1,p_2,\ldots,p_{N_0}\}$ denote the vertices of $\cM$.  
We denote the edge of $\cD$ passing $p_i$ (a vertex of $\cM$) by the closed half line
\begin{equation*}
E_i:=\{tp_i\,:\,t\geq 0\}.
\end{equation*}

We will call $\mathcal{W}$ \textit{a wedged domain} (with inner angle $\kappa\in(0,2\pi)\setminus\{\pi\}$) if there is a rotation map $T:\bR^3\rightarrow \bR^3$ with $T({\bf{0}})={\bf{0}}$ such that $\mathcal{W}$ is the image of 
\begin{equation}\label{W}
	W_{\kappa}:=\big\{(\rho\cos\theta,\rho\sin\theta)\in\bR^2\,:\,\rho>0,\,\theta\in (0,\kappa)\big\}\times \bR
\end{equation}
under $T$.
We denote the edge  of $\mathcal{W}$ by $E_\mathcal{W}:=T\big( \{(0,0)\}\times \bR\big)$.

We pose natural  assumption on polyhedral cone $\cD$ is as follows. Actually, this is our mathematical definition of polyhedral cone.

\begin{assumption}\label{ass 1}\,
	
\begin{enumerate}[label=(\roman*)]
\item There is a constant  $0<r_0<1$ that allows us to have the following mappings:
for each $i=1,\,\ldots,\,N_0$, $\cD\cap B_{r_0}(p_i)$ and $E_i\cap B_{r_0}(p_i)$ are the translations by $p_i$ of parts of a wedged domain $\mathcal{W}_i$ and its edge $E_{\mathcal{W}_i}$, respectively.
More precisely, 
\begin{align}\label{241121211}
	\begin{aligned}
	&\;\;\qquad\cD\cap B_{r_0}(p_i)=p_i+ \big(\mathcal{W}_i\cap B_{r_0}({\bf{0}}) \big):=\{p_i+x\,:\,x\in\mathcal{W}_i\cap B_{r_0}(\bf{0})\}\,\,,\\
	&\;\;\qquad E_i\cap B_{r_0}(p_i)=p_i+ \big(E_{\mathcal{W}_i}\cap B_{r_0}(\bf{0}) \big),
	\end{aligned}
\end{align}
where $\mathcal{W}_i$ is a wedged domain with angle $\kappa_i\in(0,2\pi)\setminus\{\pi\}$.

\item There is a constant $0<r_1<1$ satisfying the following:
For any $p\in\partial\cM\setminus V_{\cM}$ we denote the distance from $p$ to its nearest edge(s) of $\cD$ by $d_p:=d(p,\hat{E})$, where $\hat{E}:=\;\bigcup^{N_0}_{i=1} E_i$. Then there exists an affine map $T_p$ from $\bR^3$ to $\bR^3$ which is a composition of a rotation and a transition, and implements
\begin{align*}
&T_p\big(B^{\cD}_{r_1d_p}(p)\big)=\bR_+^3\cap B_{r_1d_p}(\bf{0})\,\,,\\
&T_p\big(\partial\cD\cap B_{r_1d_p}(p)\big)=(\partial\bR_+^3)\cap  B_{r_1d_p}(\bf{0})\,\,,\\
&T_p(p)=\bf{0},
\end{align*}
where $\bR^3_+=\{(x_1,x_2,x_3)\,:\,x_3>0\}.$
\end{enumerate}

\end{assumption}

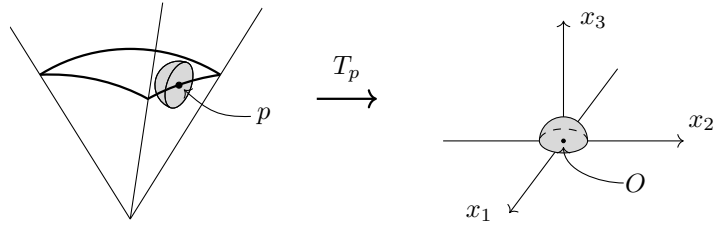
\begin{figure}[h]
	\begin{tikzpicture}[scale=0.8]
		
		\begin{scope}[shift={(-3.6,-2)}]
			\clip (-2.3,-0.5) -- (2.7,-0.5) -- (2.7,4) -- (-2.3,4);
			\clip (0.2,2) circle (2.5);
			
				\begin{scope}[rotate=70]
				\draw[fill=gray!30] (2.77,0) arc (0:180:0.4);
				\draw[fill=gray!30] (2.77,0) arc (0:360:0.4 and 0.2)	;
				\draw[fill=black] (2.37,0) circle (0.05);
				\draw (2.77,0) arc (0:180:0.4);		
				\draw (2.77,0) arc (0:360:0.4 and 0.2)	;
				\draw (2.37,-1.5) node {$p$};
				\draw[<-] (2.33,-0.1) .. controls +(-0.2,-0.5) and +(-0.2,0.5) .. (2.29, -1.3);
			\end{scope}

			\draw (0,0) -- (0.54,3.6);
			\draw (0,0) -- (2.5,4);
			\draw (0,0) -- (-2.5,4);
			\begin{scope}
				\clip (0,0) -- (0.45,3) -- (2.5,4) -- (0,0);
				\draw[line width=0.3mm] (1.5,2.4) arc (100:130:3 and 2.4);
			\end{scope}
			\begin{scope}
				\clip (0,0) -- (0.45,3) -- (-2.5,4) -- (0,-4);
				\draw[line width=0.3mm] (-1.5,2.4) arc (93.3:30:3 and 2.55);
			\end{scope}
		\begin{scope}
			\clip (0,0) --  (-2.5,4) -- (2.5,4) --(0,0);
			\draw[line width=0.3mm] (0,0) circle (2.83);
		\end{scope}

		\end{scope}

		\draw[line width=0.3mm, ->] (-0.5,0)--(0.5,0);
		\draw (0,0.5) node {$T_p$};

		\begin{scope}[shift={(3.6,-0.7)}]

			\draw[->] (-2,0) -- (2,0);
			\draw[<-] (-0.9,-1.2) -- (0.9,1.2);
			\draw[->] (0,0) -- (0,2);
			
				\begin{scope}
				\fill[gray!30] (0.4,0) arc (0:180:0.4);
				\fill[gray!30] (0.4,0) arc (0:360:0.4 and 0.2)	;
				\draw (0.4,0) arc (0:180:0.4);		
				\draw[dashed] (0.4,0) arc (0:180:0.4 and 0.2)	;
				\draw (0.4,0) arc (0:-180:0.4 and 0.2);
				\draw (-1.4,-1.2) node {$x_1$};
				\draw (2.3,0.3) node {$x_2$};
				\draw (0.5,2) node {$x_3$};
				\draw[fill=black] (0,0) circle (0.03);
				\draw (1.2,-0.7) node {$O$};
				\draw[<-] (0,-0.1) .. controls +(0,-0.4) and +(-0.4,0) .. (1, -0.7);
			\end{scope}
		\end{scope}
	\end{tikzpicture}
\end{figure}

\begin{remark}\label{distance to edges}
For any $\xi$ on an edge $E_i$, $\xi=|\xi|p_i$, and by the cone structure of $\cD$, $B^{\cD}_{r_0|\xi|}(\xi)$ is a spherical wedge with the inner angle $\kappa_i$ as $B^{\cD}_{r_0}(p_i)$ is.

Also, for any $\xi$ on a face of $\cD$ with the closest edge $E_i$, we have $d_{\xi}:=d(\xi,\hat{E})=|\xi|d(p,\hat{E})=:|\xi|d_p$, where $p:=\frac{\xi}{|\xi|}$. Indeed, if $d(p,\hat{E})=d\left(p,E_i\right)=d\left(p,tp_i\right)$ for some $t>0$, then $d(\xi,\hat{E})=d(\xi,E_i)=d(\xi,|\xi|tp_i)=|\xi|d\left(p,tp_i\right)=|\xi|d(p,\hat{E})$ by the cone structure of $\cD$. Hence, as $B^{\cD}_{r_1d_p}(p)$ is an open half ball, so is $B^{\cD}_{r_1d_{\xi}}(\xi)$.

\end{remark}

Recall  our operator $\mathcal{L}=\frac{\partial}{\partial t}-\sum_{i,j}^3a_{ij}(t)D_{ij}$  with the  condition \eqref{uniform parabolicity}. 
We will also consider the operator  
\begin{equation}\label{another operator}
\tilde{\mathcal{L}}:=\frac{\partial}{\partial t}-\sum_{i,j}^3a_{ij}(-t)D_{ij}.
\end{equation}

To state our  Green's function estimate we need the following definition. We emphasize that $a_{ij}$, $i,j=1,2,3$, in the operator $\cL$ and the manifold $\cM\subset \bS^2$ are fixed throughout the article. For other parabolic operators we use different notations for the coefficients and operators.

\begin{defn}\label{def.critical lambda.edges.}
\,

\begin{enumerate}[label=(\roman*)]
	\item (\emph{critical exponent for the vertex}) 
By $\hat{\lambda}^+_{o}(\cD,\cL)$, we denote the supremum of positive constants $\lambda$ each of which allows a constant $N=N(\cD, \cL,\lambda)>0$ fulfilling the following:
For any $t_0\in\bR$, $r>0$ and  $u\in\cV_2^{1,0}(Q_r^{\cD}(t_0,\bf{0}))$ satisfying 
$$
\cL u=0\quad\text{in}\quad Q_r^{\cD}(t_0,{\bf{0}})\quad;\quad u|_{(\bR\times\partial\cD)\cap Q_r(t_0,{\bf{0}})}=0
$$
in the sense of distributions (see Definition \ref{250219615}), 
the estimate
\begin{align}\label{23081054111}
|u(t,x)|\leq N \Big(\frac{d(x,V)}{r}\Big)^{\lambda}\sup_{Q^{\cD}_r(t_0,\bf{0})}|u|\quad\quad\forall \,\,\,
(t,x)\in Q^{\cD}_{r/2}(t_0,\bf{0})
\end{align}
holds. 
Another constant $\hat{\lambda}^-_{o}(\cD,\cL)$ is defined in the same manner using the operator $\tilde{\cL}$ instead of $\cL$.

	\item (\emph{critical exponent for the edges}) Let $\mathcal{W}$ be a wedged domain with the inner angle $\kappa$, \textit{i.e.}, a rotation of $W_{\kappa}$ in \eqref{W}.
	Then $\hat{\lambda}^+_{e}(\mathcal{W},\cL)$ is defined as  the supremum of  positive constants $\lambda$ each of which allows a constant $N=N(\mathcal{W},\cL,\lambda)>0$ fulfilling the following: 
	for any $t_0\in\bR$, $r>0$ and  $u\in\cV_2^{1,0}(Q_r^{\mathcal{W}}(t_0,\bf{0}))$ satisfying 
	$$
	\cL u=0\quad\text{in}\quad Q_r^{\mathcal{W}}(t_0,{\bf{0}})\quad;\quad u|_{(\bR\times\partial\mathcal{W})\cap Q_r(t_0,{\bf{0}})}=0\,,
	$$
	the estimate
	$$
	|u(t,x)|\leq N \bigg(\frac{d(x,E_\mathcal{W})}{r}\bigg)^{\lambda}\sup_{Q^{\mathcal{W}}_r(t_0,{\bf{0}})}|u|\quad\;\;\forall \,\,\,
	(t,x)\in Q^{\mathcal{W}}_{r/2}(t_0,{\bf{0}})
	$$
	holds. 
	Another constant $\hat{\lambda}^-_{e}(\mathcal{W},\cL)$ is defined as $\hat{\lambda}^+_{e}(\mathcal{W},\tilde{\cL})$.
	
Now, for each $i=1,2,\ldots,N_0$ let $\kappa_i\in(0,2\pi)\setminus\{\pi\}$ be the inner angle of two adjacent faces of $\cD$ meeting on the edge $E_i$.  Take the wedge domain $\mathcal{W}_i$  with the inner angle $\kappa_i$ such that any spherical wedge $\cD\cap B_r(p_i)$ for $0<r\le r_0$ becomes a part of $p_i+\mathcal{W}_i$ in the manner of \eqref{241121211}.
	Then we define $\hat{\lambda}^\pm_{e,i}(\cD,\cL):=\hat{\lambda}^\pm_{e}(\mathcal{W}_i,\cL)$.
	\end{enumerate}
\end{defn}

\begin{remark}\label{edge.def.2.}
The definition of $\hat{\lambda}^\pm_{e,i}(\cD,\cL)$ in
Definition \ref{def.critical lambda.edges.} (ii) is equivalent to the following more intuitive definition:

{\it{For each $i=1,2,\ldots,N_0$ we define $\hat{\lambda}^+_{e,i}(\cD,\cL)$ by the supremum of the positive constants $\lambda$ each of which allows us to have a constant $N=N(\kappa_i,\cD, \cL,\lambda)>0$ fulfilling the following: 
for any $t_0\in\bR$, $\xi_0\in E_i$, $r>0$ with which $B_r(\xi_0)$ does not intersect the vertex $V$ and all edges except $E_i$, if $u\in\mathcal{V}(Q_r^{\cD}(t_0,\xi_0))$ satisfies $\cL u=0$ in $Q_r^{\cD}(t_0,\xi_0)$ with $u|_{(\bR\times\partial\cD)\cap Q_r(t_0,\xi_0)}=0$, then
the estimate
$$
|u(t,x)|\leq N \Big(\frac{d(x,E_i)}{r}\Big)^{\lambda}\sup_{Q^{\cD}_r(t_0,\xi_0)}|u|\quad\;\;\forall \,\,\,
(t,x)\in Q^{\cD}_{r/2}(t_0,\xi_0)
$$
holds. Another constant $\hat{\lambda}^-_{e,i}(\cD,\cL)$ is similarly defined  using the operator $\tilde{\cL}$ in \eqref{another operator} instead of $\cL$.}}

This equivalency is based on the facts that (a) $\xi+\mathcal{W}=\mathcal{W}$ for any $\xi\in E_{\mathcal{W}}$ and hence $d(x,E_{\mathcal{W}})=d(x-\xi,E_{\mathcal{W}})$. (b)  For any $\xi_0\in E_{\mathcal{W}}$, $r>0$ and a function defined on $Q^{\mathcal{W}}_r(t_0,\xi_0)$ one can always find $\xi'_0\in E_{\mathcal{W}}$ such that $r\le r_0|\xi'_0|$ satisfying $Q^{\mathcal{W}}_r(t_0,\xi'_0)=Q^{\cD}_r(t_0,\xi'_0)$ and consider  the fuction $v(t,y)=u(t,x)$ with the relation $y=x+\xi'_0-\xi_0$ defined on $Q^{\cD}_r(t_0,\xi'_0)$.
\end{remark}

 The critical exponents $\hat{\lambda}_o^{\pm}$ are  closely related to the first eigenvalue   $\mathcal{E}_0$ of the Laplace Beltrami operator on $\cM\subset \bS^2$ with Dirichlet boundary condition.
 Here, $\mathcal{E}_0$ is defined as
 \begin{align}\label{250121350}
\mathcal{E}_0:=\inf_{f\in C_c^\infty(\cM),f\neq 0}\frac{\int_\cM|\nabla_{\bS} f|^2d\sigma}{\int_{\cM}|f|^2d\sigma}>0\,,
 \end{align}
where $\nabla_{\bS}$ is the spherical gradient (for the definition of $C_c^\infty(\cM)$ and $\nabla_\bS$, see \cite[Subsection 5.5]{Seo202411}).
Note that $\mathcal{E}_0>0$ due to our assumption on $\cM$ (see \cite[Proposition 5.24]{Seo202411}).
 Below are some relations between $\hat{\lambda}_o^{\pm}$ and $\mathcal{E}_0$.

\begin{prop}\label{2501181139}
With our operator $\cL$ in \eqref{our operator} conditioned by \eqref{uniform parabolicity} and polyhedral cone $\cD$, we have the followings.

\begin{enumerate}[label=(\roman*)]
\item $\hat{\lambda}^{\pm}_{o}(\cD,\cL)>0$ and
$$
\hat{\lambda}^{\pm}_{o}(\cD,\cL)\geq -\frac{3}{2}+\sqrt{\frac{\nu_1}{\nu_2}}\sqrt{\mathcal{E}_0+\frac{1}{4}}\,.
$$

\item If $\cL=\frac{\partial}{\partial t}-\Delta_x$, then
$$
\hat{\lambda}^+_{o}(\cD,\cL)
=\hat{\lambda}^-_{o}(\cD,\cL)
=-\frac{1}{2}+\sqrt{\mathcal{E}_0+\frac{1}{4}}.
$$
\end{enumerate}
\end{prop}
\begin{proof}
We first note that our polyhedral cones satisfy the following volume density condition, often referred to as the (A)-condition:
	$$
	\inf_{p\in\partial\cD,r>0}\frac{|\cD^c\cap B_r(p)|_d}{|B_r(p)|_d}>0\,,
	$$
	(see \cite[Proposition 5.24.(2)]{Seo202411}).
	It is straightforward to verify that the inequality above implies
	\begin{align}\label{250102420}
		\inf_{(t,p)\in\bR\times \partial\cD,r>0}\frac{|(\bR\times \cD)^c\cap Q_r(t,p)|_{d+1}}{|Q_r(t,p)|_{d+1}}>0
	\end{align}
; see Definition \ref{A} in Appendix \ref{sec:appendix A}  for a generalized version.

(i)   We can prove the claims in exactly the same way as outlined  in the proof of \cite[Theorems 2.4.1, 2.4.7]{Kozlov_Nazarov_2014}. 
Notably, the proof of \cite[Theorems 2.4.1, 2.4.7]{Kozlov_Nazarov_2014} relies on \cite[Sections III.8, III.10]{LSU_1968}, which does not require the smoothness of the domain $\cD$, but instead depends on \eqref{250102420}. 

For the reader's convenience, in Appendix \ref{sec:appendix A} we provide the detailed proof of (i) with the strategy that we will be brief on our proofs in  places where we simply need to repeat the arguments given in  \cite[Theorems 2.4.1, 2.4.7]{Kozlov_Nazarov_2014}; Theorem \ref{A2} establishes the first claim $\hat{\lambda}^{\pm}_{o}(\cD,\cL)>0$, and Theorem \ref{A6} together with Lemma \ref{A4} explains why the second claim holds.

(ii) 
This assertion is essentially proved in \cite{Seo202411}. Let us denote $\lambda_0:=-\frac{1}{2}+\sqrt{\mathcal{E}_0+\frac{1}{4}}$.  One can observe that our polyhedral cones satisfy \cite[Assumption 5.23]{Seo202411} (see also \cite[Proposition 5.24]{Seo202411}).
Therefore \cite[Theorem 5.25, Remark 5.26]{Seo202411} implies that for any $\lambda\in(0,\lambda_0)$, there exists $N=N(\cD,\lambda)>0$ such that if $v\in C^{\infty}(Q_1^{\cD}(0,{\mathbf{0}}))\cap C (\overline{Q_1^{\cD}(0,{\mathbf{0}})})$ satisfies
$$
v_t-\Delta u=0\quad\text{in}\quad Q_1^\cD(0,{\mathbf{0}})\quad;\quad v|_{(\bR\times \partial\cD)\cap Q_1(0,{\mathbf{0}})}=0,
$$
then
\begin{align}\label{250102425}
|v(t,x)|\leq N\Big(\sup_{Q_1^{\cD}(0,{\mathbf{0}})}|v|\Big)|x|^{\lambda}\quad\text{for all}\quad (t,x)\in Q_{1/2}^\cD(0,{\mathbf{0}})\,.
\end{align}
Consider $u\in\cV_2^{1,0}(Q_r^{\cD}(t_0,{\mathbf{0}}))$ satisfying 
$$
\cL u=0\quad\text{in}\quad Q_r^\cD(t_0,0)\quad;\quad u|_{(\bR\times \partial\cD)\cap Q_r(t_0,{\mathbf{0}})}=0\,.
$$
By the classical theory for parabolic equations, we have $u\in C_c^{\infty}(Q_r^{\cD}(t_0,{\mathbf{0}}))$.
Moreover, due to \eqref{250102420} and \cite[Theorem III.10.1]{LSU_1968}, $u$ is continuous in $(\bR\times \overline{\cD})\cap Q_r(t_0,{\mathbf{0}})$ with $u|_{(\bR\times \partial\cD)\cap Q_r(t_0,{\mathbf{0}})}=0$.
By setting $v(t,x):=u(t_0+r^2t,rx)$ and applying \eqref{250102425}, we obtain \eqref{23081054111} for all $\lambda \in(0,\lambda_0)$.
This implies $\hat{\lambda}^{\pm}_0(\cD,\cL)\geq \lambda_0$.

To prove  $\hat{\lambda}^{\pm}_0(\cD,\cL) \leq \lambda_0$, we consider the function
$$
u_0(t,x):=|x|^{\lambda_0}w_0(x/|x|)\,,
$$
where $w_0$ is the first eigenvalue of the spherical Laplacian on $\cM$. 
Let $w_{0,n}\in C_c^{\infty}(\cM)$ be functions that converge to $w_0$ in $W_2^1(\cM)$ (see \cite[Subsection 5.5]{Seo202411} for the definition of $C_c^{\infty}(\cM)$ and $W_2^1(\cM)$).
Let $\eta$ be a smooth function on $\bR$ such that $\eta(r)=0$ if $r\leq \frac{1}{2}$ and $\eta(r)=1$ if $r\geq 1$.
Then $\eta(n|x|)|x|^{\lambda_0}w_{0,n}(x/|x|)$ converges to $u_0$ in $\cV_2^{1,0}(Q_r^{\cD}(t_0,{\mathbf{0}}))$ (see \cite[(5,58), (5,59)]{Seo202411}).
Hence, we have $u_0\in \cV_2^{1,0}(Q_r^\cD(t_0,{\mathbf{0}}))$ with the boundary condition $u|_{(\bR\times \partial \cD)\cap Q_r(t_0,{\mathbf{0}})}=0$.
Therefore, from the behavior of $u_0$ near $\mathbf{0}$, we obtain that $\hat{\lambda}^{\pm}_0(\cD,\cL)\leq \lambda_0$.
\end{proof}

\begin{remark}\label{250220324}
	In this remark we discuss  lower bound of $\mathcal{E}_0$.
	Consider a sequence of  subdomains   $\cM_n \subset \bS^2$, $n=1,2,\cdots$, such that 
	\begin{align}\label{250121346}
	\overline{\cM_n}\subset \cM_{n+1}\subset \cM\quad\text{with}\quad \bigcup_{n=1}^\infty \cM_n=\cM
	\end{align}
	and $\cM_n$ has $C^{\infty}$-boundary
	(see \textit{e.g.} the proof of \cite[Proposition 8.2.1]{Daners}).
	We define the set  
	$$
	C(\theta):=\left\{(\sigma_1,\sigma_2,\sigma_3)\in\bS^2\,:\,\sigma_1>\cos\theta\right\},
	$$ 
	which is called the spherical cap with angle $\theta$.	 Let $\theta_\infty$ and $\theta_n$ be the constants  such that
	$$
	|C(\theta_\infty)|_{\bS^2}=|\cM|_{\bS^2}\quad\text{and}\quad |C(\theta_n)|_{\bS^2}=|\cM_n|_{\bS^2}\,,
	$$
	where $|\,\cdot\,|_{\bS^2}$ is the surface measure on $\bS^2$;
	one can check that
	\begin{align}\label{250306855}
	\cos\theta_\infty =1-\frac{|\cM|_{\bS^2}}{2\pi}\quad\text{and}\quad \cos\theta_n =1-\frac{|\cM_n|_{\bS^2}}{2\pi}\,.
	\end{align}
	Also, let  $\mathcal{E}_0(\cM_n)$ denote the first eigenvalue  of  the Laplace Beltrami operator on $\cM_n$ with Dirichlet boundary condition. By the Faber-Krahn inequality (see, \textit{e.g.}, \cite[Theorem 2, page 89]{Chavel}), we have
	$$
	\mathcal{E}_0(\cM_n)\geq \mathcal{E}_0\big(C({\theta_n})\big)\,.
	$$
	From \eqref{250121346} and \eqref{250306855}, we know that $|\cM_n|_{\bS^2}\nearrow |\cM|_{\bS^2}$, so  $\theta_n\nearrow \theta_\infty$.
	This implies that
	$$
	\overline{C(\theta_n)}\subset C(\theta_{n+1})\subset C(\theta_\infty)\quad\text{with}\quad \bigcup_{n=1}^\infty C(\theta_n)=C(\theta_\infty)\,.
	$$
	Moreover, from the definition of the first eigenvalue (see \eqref{250121350}), we have $\inf_{n}\mathcal{E}_0(\cM_n)=\mathcal{E}_0(\cM)$ and $\inf_{n}\mathcal{E}_0\big(C(\theta_n)\big)=\mathcal{E}_0\big(C(\theta_\infty)\big)$.
	Therefore,
	$$
	\mathcal{E}_0=:\mathcal{E}_0(\cM)\geq \mathcal{E}_0\big(C({\theta_\infty})\big)\,.
	$$
	It is shown in \cite[Theorem 2.1, Section 4]{BCG1983} that if $\theta_\infty\in(0,2\pi)$, \textit{i.e.}, $0<|\cM|_{\bS^2}<4\pi(=|\bS^2|_{\bS^2})$, then
	$$
	\mathcal{E}_0\big(C({\theta_\infty})\big)\geq \left(\log\frac{2}{\cos\theta_\infty+1}\right)^{-1}=\left(\log\frac{4\pi}{4\pi-|\cM|_{\bS^2}}\right)^{-1}.
	$$
	Moreover, if $\theta_\infty\in(0,\pi)$, \textit{i.e.}, $0<|\cM|_{\bS^2}<2\pi$, then
	$$
	\mathcal{E}_0\big(C({\theta_\infty})\big)\geq \frac{j_0^2}{4}\left(\frac{1}{\sin^2(\theta_\infty/2)}-\frac{1}{2}\right)-\frac{1}{4}=\frac{j_0^2}{4}\left(\frac{4\pi}{|\cM|_{\bS^2}}-\frac{1}{2}\right)-\frac{1}{4}\,.
	$$
	Here, $j_0$ denotes the first zero of the Bessel function of the first kind, $J_0$.
\end{remark}

For other critical exponents $\hat{\lambda}_{e,i}^{\pm}$, $i=1,2,\ldots, N_0$,  some concrete information is available.

\begin{prop}\label{250213305} 
With our operator $\cL$ in \eqref{our operator} conditioned by \eqref{uniform parabolicity} and polyhedral cone $\cD$, we have
the followings; below $\kappa_i\in(0,2\pi)\setminus\{\pi\}$, $i=1,2,\ldots,N_0$ are from Definition \ref{def.critical lambda.edges.}.

\begin{enumerate}[label=(\roman*)]
\item For each $i=1,2,\cdots,N_0$, $\hat{\lambda}^{\pm}_{e,i}(\cD,\cL)>0$ and
$$
\hat{\lambda}^{\pm}_{e,i}(\cD,\cL)\geq -1+\sqrt{\frac{\nu_1}{\nu_2}}\,\frac{\pi}{\kappa_i}.
$$

\item If $\cL=\frac{\partial}{\partial t}-\Delta_x$, then for each  $i=1,2,\cdots,N_0$,
$$
\hat{\lambda}^{\pm}_{e,i}(\cD,\cL)=\frac{\pi}{\kappa_i}.
$$
\end{enumerate}
\end{prop}
\begin{proof} 
	(i) For the fact $\hat{\lambda}^{\pm}_{e,i}(\cD,\cL)>0$ we place a detailed proof in Theorem \ref{A3} based on Remark \ref{edge.def.2.}.  

For the second claim of (i), as explained below, one can check that \cite[Theorems 2.4.6 and 2.4.7]{Kozlov_Nazarov_2014} yield
\begin{equation}\label{edge.2dim.1.}
\hat{\lambda}^{\pm}_{e,i}(\cD,\cL)\geq -\frac22+\sqrt{\frac{\nu_1}{\nu_2}}\,\sqrt{\hat{\mathcal{E}}_i+\frac{(2-2)^2}{4}},
\end{equation}
where $\hat{\mathcal{E}}_i$ is the first eigenvalue of the Laplace-Beltrami operator with zero Dirichlet boundary condition on $\big\{(\cos\theta,\sin\theta)\in\bS^1\,:\,\theta\in (0,\kappa_i)\big\}$. In fact we have $\hat{\mathcal{E}}_i=\big(\frac{\pi}{\kappa_i}\big)^2$; see \cite[Remark 3.1]{Green}. Hence, the second claim of (i) holds.

The authors of \cite{Kozlov_Nazarov_2014} consider the domains of the type $\mathcal{C}\times\mathbb{R}^l$, where $\mathcal{C}\subset \mathbb{R}^m$ with $m\ge 2$ is an infinite cone and $l\ge 0$. 
The case $m=2$ and $l=1$ gives $W_\kappa$  in \eqref{W}, which becomes wedged domains under rotation maps; for that matter $W_{\kappa}$ is also a wedged domain. For the domains $W_{\kappa}\subset \mathbb{R}^2\times \bR$,  \cite[Section 2]{Kozlov_Nazarov_2014} defines the critical exponents using the parabolic cylinders of the type
$$
\widetilde{Q}_r^{\cD}(t_0,{\bf{0}})
:=(t_0-r^2,t_0]\,\times (\widetilde{B}_r({\bf{0}}) \cap \cD), 
$$
$$
 \widetilde{B}_r({\bf{0}})=\{ (x_1,x_2,x_3)\in \bR^3:
 \sqrt{x_1^2+x_2^2}<r\,,\,\,|x_3|<r\}
$$
for $t_0\in\bR$ and $r>0$ and by \cite[Theorems 2.4.6 and 2.4.7]{Kozlov_Nazarov_2014} the critical exponents are bounded below by
$-1+\sqrt{\frac{\nu_1}{\nu_2}}\,\frac{\pi}{\kappa}$. From this we get
\begin{equation}\label{edge.2dim.2.}
\hat{\lambda}^{\pm}_{e}(W_{\kappa},\cL)\geq -1+\sqrt{\frac{\nu_1}{\nu_2}}\,\frac{\pi}{\kappa}
\end{equation}
as we note that the definition of the critical exponents for $W_{\kappa}$ in \cite{Kozlov_Nazarov_2014} is  equivalent to our definition of critical exponents for the edges, $\hat{\lambda}^{\pm}_{e}(W_{\kappa},\cL)$, in Definition~\ref{def.critical lambda.edges.} (ii) since
$$
Q_r^{\cD}(t_0,{\bf{0}}) \subset \widetilde{Q}_r^{\cD}(t_0,{\bf{0}})\subset Q_{\sqrt{2}r}^{\cD}(t_0,{\bf{0}})
$$
and $r/2$ in \eqref{23081054111} can be replaced by $\kappa r$ for any fixed constant $\kappa\in (0,1)$ (see the proof of \cite[Lemma 2.2]{Kozlov_Nazarov_2014}).

Since the right hand side of \eqref{edge.2dim.2.} depends only on $\nu_1,\nu_2$, and $\kappa$,   inequality \eqref{edge.2dim.2.} holds for any parabolic operator $L=\frac{\partial}{\partial t}-\sum_{i,j=1}^3 \alpha_{ij}(t)$ in place of $\cL$, provided that the coefficients $\alpha_{ij}$ satisfy the uniform parabolicity condition \eqref{uniform parabolicity} with the same constants $\nu_1$ and $\nu_2$.
Moreover, since the uniform parabolicity constants $\nu_1$ and $\nu_2$ in \eqref{uniform parabolicity} are invariant under rotations, for any rotation $\mathcal{W}$ of $W_{\kappa}$ we have
$$
\hat{\lambda}^{\pm}_{e}(\mathcal{W},\cL)\geq -1+\sqrt{\frac{\nu_1}{\nu_2}}\,\frac{\pi}{\kappa}
$$
and \eqref{edge.2dim.1.} holds as we defined $\hat{\lambda}^\pm_{e,i}(\cD,\cL)=\hat{\lambda}^\pm_{e}(\mathcal{W}_i,\cL)$.

(ii) The claim holds by \cite[Theorems 2.4.3 and 2.4.6]{Kozlov_Nazarov_2014}.
\end{proof}

Here is our main result  on the polyhedral cone $\cD$.

\begin{thm}\label{theorem polyhedral cone}
 For any 
$\lambda^+_o\in \big(0,\hat{\lambda}_o^+(\cD,\cL)\big)$, $\lambda^-_o\in \big(0,\hat{\lambda}_o^-(\cD,\cL)\big)$ and $\lambda^+_{e,i}\in\big(0,\hat{\lambda}_{e,i}^+(\cD,\cL)\big)$, $\lambda^-_{e,i}\in\big(0,\hat{\lambda}_{e,i}^-(\cD,\cL)\big)$, $i=1,2,\ldots,N_0$,
there exists constants $N=N\big(\cD,\cL, \Lambda^+,\Lambda^-\big)>0$ and $\sigma=\sigma(\nu_1,\nu_2)>0$ such that
$$
G(t,s,x,y)\leq N\cdot \mathbf{I}\big(x,\sqrt{t-s};\Lambda^+\big)\mathbf{I}\big(y,\sqrt{t-s};\Lambda^-\big)\,\frac{1}{(t-s)^{3/2}}e^{-\sigma\frac{|x-y|^2}{t-s}}
$$
for any $(t,s,x,y)\in \bR\times \bR\times \cD\times \cD$ with $t>s$, where
$
\Lambda^{\pm}=(\lambda_o^{\pm},\lambda_{e,1}^{\pm},\ldots,\lambda_{e,N_0}^{\pm})$, 
and
$$
\mathbf{I}(x,r;\Lambda^{\pm})=\left(\frac{d(x,V)}{r}\wedge 1\right)^{\lambda^{\pm}_0}\left(\prod_{i=1}^{N_0}\left(\frac{d\big(x,E_i\big)\wedge r}{d\big(x,V\big)\wedge r}\right)^{\lambda^{\pm}_{e,i}}\right) \frac{d(x,\partial\cD)\wedge r}{d(x,\hat{E})\wedge r}.
$$
\end{thm}

\begin{remark}
\label{remark decay}
The function $I(x,r;\Lambda^{\pm})$ gives information on the behavior of the Green function near the vertex, edges, and $\partial \cD$.

\begin{enumerate}[label=(\roman*)]
\item  Obviously ${\bf{I}}(x,r;\Lambda^{\pm})\leq 1$ since
 $$d(x, \partial \cD)\leq d(x,E_i)\leq d(x,V), \quad \forall \,i.
$$

\item  If $E_k$ is the closest edge to $x$ so that 
 $d(x,\hat{E})=d(x,E_k)$, then there exists $\delta>0$ such that $\frac{d(x,E_i)}{d(x,V)}=d\left(\frac{x}{|x|},E_i\right)\simeq d\left(\frac{x}{|x|},p_i\right)\geq \delta$ for all $i\neq k$ and hence
\begin{eqnarray*}
{\bf{I}}(x,r;\Lambda^{\pm}) \simeq \left(\frac{d(x,V)}{r}\wedge 1\right)^{\lambda^{\pm}_0}
\left(\frac{d\big(x,E_k\big)\wedge r}{d\big(x,V\big)\wedge r}\right)^{\lambda^{\pm}_{e,k}} \frac{d(x,\partial\cD)\wedge r}{d(x,E_k)\wedge r}
\end{eqnarray*}
holds; also see Lemma \ref{21.11.17.10}.(ii).

\item  If $x$ is away from all edges, that is, if 
there exists a constant $\delta>0$ such that $d\big(\frac{x}{|x|},\hat{E}\big)=\frac{d(x,\hat{E})}{d(x,V)} \geq \delta$, then $d(x,E_i)\leq d(x,V)\leq \delta^{-1}d(x,E_i)$ for all $i$. Hence, we have
$$
{\bf{I}}(x,r;\Lambda^{\pm}) \simeq \left(\frac{d(x,V)}{r}\wedge 1\right)^{\lambda^{\pm}_o-1}
\left(\frac{d(x,\partial\cD)}{r}\wedge 1\right).
$$
Of course, the relation $\simeq$ depends on $\delta$.

\item  If $x$ is away from the boundary, that is, if 
there exists $\delta>0$ such that $d\big(\frac{x}{|x|},\partial \cD\big)=\frac{d(x,\partial\cD)}{d(x,V)} \geq \delta $, then $d(x,\partial\cD)\leq d(x,E_i)\leq d(x,V)\leq \delta^{-1}d(x,\partial\cD)$ for all $i$
and hence 
$$
{\bf{I}}(x,r;\Lambda^{\pm}) \simeq  \left(\frac{d(x,V)}{r}\wedge 1\right)^{\lambda^{\pm}_o}
$$
holds.
\end{enumerate}

\end{remark}

\subsection{Green's function estimate on  polyhedrons}

\indent
A polyhedron  $\cP$ is a bounded three-dimensional object with finite number of  polygonal faces, straight edges and sharp corners or vertices.

We pose the following natural assumptions on $\cP$; we use notation $B(x,r):=B_r(x)$.

\begin{assumption}\label{250102724}
\label{ass 2} Let  $\hat{V}:=\{V_1,V_2,\ldots,V_{M_0}\}$ denote the set of all vertices of $\cP$, and  $\widehat{E}:=\bigcup_{j=1}^{N_0}E_j$ denote the union of all edges of $\cP$. 

\begin{enumerate}[label=(\roman*)]

\item $V_i$ are separated points in $\bR^3$, and each $E_i$ is a closed interval connecting two vertices.
    
\item  Each $E_i$ contains no other vertices except its two end points, and  $E_i\cap E_j\subset \hat{V}$.

\item There exist $r_0,\,r_1,\,r_2\in(0,1)$ such that

		\begin{itemize}
\item[(a)] for each $V_i$, there exists an infinite polyhedral cone $\cD_{v,i}$ with vertex $\bf{0}$ such that
\begin{align}\label{241014218}
\cP\cap B(V_i,r_0)=V_i+(\cD_{v,i} \cap B({\bf{0}},r_0));
\end{align}

\item[(b)] for each $E_j$, there exists a wedged domain $\mathcal{W}_{e,j}$ with angle $\kappa_i\in(0,2\pi)\setminus\{\pi\}$ such that for any $p\in E_j\setminus \hat{V}$,
\begin{align}\label{241125633}
\cP\cap B\big(p,r_1d(p,\hat{V})\big)= p+\big(\mathcal{W}_{e,j}\cap B\big({\bf{0}},r_1d(p,\hat{V})\big);
\end{align}

\item[(c)] for each $\xi\in \partial\cP\setminus \hat{E}$,\,
$\cP\cap B\big(\xi,r_2d(\xi,\hat{E})\big)$ is a   translation and rotation of $\bR_+^3\cap B\big({\bf{0}},r_2d(\xi,\hat{E})\big)$.

		\end{itemize}    
\end{enumerate} 
	
\end{assumption}

\begin{remark}\label{250306852}
One can observe, by applying Lebesgue's number lemma, that Assumptions \ref{250102724}.(iii) is equivalent to the following:

\begin{itemize}
	\item[(a)] For each $V_i$, there exists an infinite polyhedral cone $\cD_{v,i}$ with vertex $\bf{0}$, and a constant $r_{v,p}>0$ such that
	\begin{align*}
		\cP\cap B(V_i,r_{v,i})=(V_i+\cD_{v,i}) \cap B({\bf{0}},r_{v,i})\,.
	\end{align*}

	\item[(b)] For each $E_j$, there exists a wedged domain $\mathcal{W}_{e,j}$ with angle $\kappa_i\in(0,2\pi)\setminus\{\pi\}$ such that, for any $p\in E_j\setminus \hat{V}$, there exists a constant $r_{e,p}>0$ satisfying
	\begin{align*}
		\cP\cap B\big(p,r_{e,p}\big)= p+\big(\mathcal{W}_{e,j}\cap B\big({\bf{0}},r_{e,p})\big)\,.
	\end{align*}

	\item[(c)] For each $\xi\in \partial\cP\setminus \hat{E}$, there exists a constant $r_{f,\xi}>0$ such that
	$\cP\cap B\big(\xi,r_{f,\xi}\big)$ is a translation and rotation of $\bR_+^3\cap B\big({\bf{0}},r_{f,\xi}\big)$.
	
\end{itemize} 
 
\end{remark}

 We  define the critical exponents for edges and vertices of polyhedron $\cP$ based on Definition \ref{def.critical lambda.edges.}.

\begin{defn}\label{241227911}\,
	
	\begin{enumerate}[label=(\roman*)]
	\item (\emph{critical exponent for the vertices})  Let $\cD_{v,i}$ ($i=1,\,\ldots,\,M_0$) be the infinite polygonal cones from \eqref{241014218}.
	For each $i$, we define $\hat{\lambda}^+_{v,i}(\cP,\cL):=\hat{\lambda}_o^+(\cD_{v,i},\cL)$. In other words, $\hat{\lambda}^+_{v,i}(\cP,\cL)$ is 
the supremum of the positive constants $\lambda$ each of which allows a constant $N=N(\cD_{v,i}, \cL,\lambda)>0$ fulfilling the following:
	for any $t_0\in\bR$, $r>0$ and  $u\in\cV_2^{1,0}(Q_r^{\cD_{v,i}}(t_0,\bf{0}))$ satisfying 
$$
\cL u=0\quad\text{in}\quad Q_r^{\cD_{v,i}}(t_0,{\bf{0}})\quad;\quad u|_{(\bR\times\partial\cD_{v,i})\cap Q_r(t_0,{\bf{0}})}=0
$$
(see Definition \ref{250219615}),
	the estimate
	\begin{align*}
		|u(t,x)|\leq N \Big(\frac{d(x,V)}{r}\Big)^{\lambda}\sup_{Q^{\cD}_r(t_0,\bf{0})}|u|\quad\quad\forall \,\,\,
		(t,x)\in Q^{\cD}_{r/2}(t_0,\bf{0})
	\end{align*}
	holds. 
	Another constant $\hat{\lambda}^-_{v,i}(\cP,\cL)$ is similarly defined  using the operator $\tilde{\cL}$ in \eqref{another operator} instead of $\cL$.

	\item (\emph{critical exponents for edges})  For each $j=1,2,\ldots,N_0$ let $\mathcal{W}_{e,j}$ be the wedged domain from \eqref{241125633}. 
	We define $\hat{\lambda}^+_{e,j}(\cP,\cL):=\hat{\lambda}^+(\mathcal{W}_{e,j},\cL)$. In other words, $\hat{\lambda}^+_{e,j}(\cP,\cL)$
 	is the supremum of the positive constants $\lambda$ each of which allows us to have a constant $N=N(\mathcal{W}_{e,j}, \cL,\lambda)>0$ fulfilling the following: 
	for any $t_0\in\bR$, $r>0$ and  $u\in\cV_2^{1,0}(Q_r^{\mathcal{W}_{e,j}}(t_0,\bf{0}))$ satisfying 
$$
\cL u=0\quad\text{in}\quad Q_r^{\mathcal{W}_{e,j}}(t_0,{\bf{0}})\quad;\quad u|_{(\bR\times\partial\mathcal{W}_{e,j})\cap Q_r(t_0,{\bf{0}})}=0,
$$
	the estimate
	$$
	|u(t,x)|\leq N \Big(\frac{d(x,E_i)}{r}\Big)^{\lambda}\sup_{Q^{\mathcal{W}_{e,j}}_r(t_0,\xi_0)}|u|\quad\;\;\forall \,\,\,
	(t,x)\in Q^{\mathcal{W}_{e,j}}_{r/2}(t_0,\xi_0)
	$$
	holds. 
	Another constant $\hat{\lambda}^-_{e,j}(\cP,\cL)$ is similarly defined  using the operator $\tilde{\cL}$ in \eqref{another operator} instead of $\cL$.
	
	\end{enumerate}
\end{defn}

\begin{example} 
	Let $\cP$ be a platonic solid which is a convex regular polyhedron in $\bR^3$. Then $\cP$ is one of the following: tetrahedron (four equilateral triangle faces), cube (six square faces), octahedron (eight equilateral triangle faces), dodecahedron (twelve regular pentagon faces), or icosahedron (twenty equilateral triangle faces).
All edges of $\cP$ has the same inner angles.
Moreover, all polyhedral cones $\cD_v$ corresponding vertices $v$ of $\cP$ are the same under rotations ($\cD_v$ is a polyhedral cone such that $v+\big(\cD_v\cap B_r(0)\big)=\cP\cap B_r(v)$ for small enough $r>0$).
Hence solid angles of at the vertices of $\cP$ are exactly the same,  where the solid angle at $v\in\cP$ is the surface area of $\cM_v\subset\bS^2$ which generates the polyhedral cone $\cD_v$. 
 By using the common inner angle at edges, the common solid angle at vertices,  and  Remark \ref{250220324}, we can calculate lower bounds of $\hat{\lambda}_{v,i}^\pm(\cP,\cL)$ in Proposition \ref{2501181139} and lower bounds of $\hat{\lambda}_{e,j}^\pm(\cP,\cL)$ in Proposition \ref{250213305}; see the table below.

	\begin{table}[h]
\resizebox{\textwidth}{!}{
		\begin{tabular}{|c|c|c|c|c|c|}
\hline
			\begin{tabular}[c]{@{}c@{}}Platonic\\solid\end{tabular}
			& \begin{tabular}[c]{@{}c@{}}Tetrahedron\end{tabular} & Cube & \begin{tabular}[c]{@{}c@{}}Octahedron\end{tabular} & \begin{tabular}[c]{@{}c@{}}Dodecahedron\end{tabular} & \begin{tabular}[c]{@{}c@{}}Icosahedron\end{tabular} \\ \hline
			\begin{tabular}[c]{@{}c@{}}Inner angle\\ at edges\end{tabular} 
			&                                           $\arctan (2\sqrt{2})$                                         &  $\frac{\pi}{2}$                         &   $\pi-\arctan (2\sqrt{2})$                                                                                &  $\pi-\arctan (2)$                                &  $\pi-\arctan \big(\frac{2\sqrt{5}}{5}\big)$                                         \\ \hline
			\begin{tabular}[c]{@{}c@{}}Solid angle\\ at vertices\end{tabular} &               $\arccos\big(\frac{23}{27}\big)$                                                                     &       $\frac{\pi}{2}$                    &         $4\arcsin\big(\frac{1}{3}\big)$                                                                          &       $\pi-\arctan\big(\frac{2}{11}\big)$                           &          $2\pi-5\arcsin\big(\frac{2}{3}\big)$                                 \\ \hline
		\end{tabular}
}
\begin{flushright}
(for the calculation of solid angles, see \cite{Mazonka2012})
\end{flushright}
	\end{table}
	
\end{example}

Here is our Green's function estimate on the polyhedron $\cP$.

\begin{thm}\label{theorem polyhedron}
	Let $\lambda_{v,i}^\pm\in\big(0,\hat{\lambda}_{v,i}^\pm(\cP,\cL)\big)$ and $\lambda_{e,j}^\pm\in \big(0,\hat{\lambda}_{e,j}^\pm(\cP,\cL)\big)$.
	Put $\Lambda^{\pm}:=\big((\lambda_{v,i}^{\pm})_{1\leq i\leq M_0},\,(\lambda_{e,j}^{\pm})_{1\leq j\leq N_0}\big)$.
	
	\begin{enumerate}[label=(\roman*)]
	\item There exist constants $N=N\big(\cP,\cL, \Lambda^\pm\big)>0$ and $\sigma=\sigma(\nu_1,\nu_2)>0$ such that the inequality
	\begin{align*}
		G_{\cP}(t,s,x,y)\leq N\, {\bf{I}}\left(x,\sqrt{t-s};\Lambda^+\right){\bf{I}}\left(y,\sqrt{t-s};\Lambda^-\right)\cdot (t-s)^{-3/2}e^{-\sigma\frac{|x-y|^2}{t-s}}
	\end{align*}
	holds for all $t\geq s$ and $x,\,y\in \cP$, where
	\begin{align*}
	&{\bf{I}}(x,r;\Lambda^{\pm})\\
	&:=\left(\,\prod_{i=1}^{M_0}\left(\frac{d(x,V_i)}{r}\wedge 1\right)^{\lambda^{\pm}_{v,i}}\right)\cdot \left(\,\prod_{j=1}^{N_0}\left(\frac{d(x,E_j)\wedge r}{d\big(x,\hat{V}\cap E_j\big)\wedge r}\right)^{\lambda^{\pm}_{e,j}}\right)\cdot \frac{d(x,\partial\cP)\wedge r}{d(x,\hat{E})\wedge r}. \nonumber
	\end{align*}

	\item   For any $T_0>0$, there exists a constant $N=N\big(\cP,\cL,\Lambda^{\pm},T_0\big)>0$ such that  if $x,y\in \cP$ and  $t-s\geq T_0$, then
	\begin{align*}
		G_{\cP}(t,s,x,y)\leq N\cdot {\bf{I}}_{\infty}(x;\Lambda^+){\bf{I}}_{\infty}(y;\Lambda^-)e^{-\nu_1\lambda (t-s)}
	\end{align*}
	where 
	$$
	{\bf{I}}_{\infty}(x;\Lambda^\pm):=\left(\,\prod_{i=1}^{M_0}d(x,V_i)^{\lambda_{v,i}^\pm}\right)\cdot\left(\,\prod_{j=1}^{N_0}\left(\frac{d(x,E_j)}{d\big(x,\hat{V}\cap E_j\big)}\right)^{\lambda_{e,j}^\pm}\right)\cdot \frac{d(x,\partial\cP)}{d(x,\hat{E})}\,.
	$$
	\end{enumerate}
\end{thm}

\begin{remark}
	\label{remark decay 2}
	The functions ${\bf{I}}(x,r;\Lambda^{\pm})$  and ${\bf{I}}_{\infty}(x;\Lambda^\pm)$ characterizes the behavior of the Green function near the vertex, edges and $\partial \cD$.
	
	\begin{enumerate}[label=(\roman*)]
	\item Obviously ${\bf{I}}(x,r;\Lambda^{\pm})\leq 1$ and ${\bf{I}}_{\infty}(x;\Lambda^{\pm})\leq \mathrm{diam}(\cP)^{\sum_{i=1}^{M_0}\lambda_{v,i}^\pm}$ since
	$$d(x, \partial \cD)\leq d(x,E_j)\leq d(x,V_i), \quad \forall \,i,\,j\quad\text{with}\quad V_i\subset E_j.
	$$

	\item If $V_k$ is the closest vertex to $x$ and $E_l$ is the closest edge to $x$ so that $d(x,\hat{V})=d(x,V_k)$ and $d(x,\hat{E})=d(x,E_l)$, then
	\begin{align*}
		{\bf{I}}(x,r;\Lambda^{\pm}) \,&\simeq \left(\frac{d(x,V_k)}{r}\wedge 1\right)^{\lambda^{\pm}_{v,k}}
		\left(\prod_{j:E_j\ni V_k}^{N_0}\left(\frac{d(x,E_j)\wedge r}{d\big(x,V_k\big)\wedge r}\right)^{\lambda_{e,j}^\pm}\right) \frac{d(x,\partial\cP)\wedge r}{d(x,\widehat{E})\wedge r}\\
		&\simeq \left(\frac{d(x,V_k)}{r}\wedge 1\right)^{\lambda^{\pm}_{v,k}}
	\left(\frac{d\big(x,E_l\big)\wedge r}{d\big(x,V_k\big)\wedge r}\right)^{\lambda^{\pm}_{e,k}} \frac{d(x,\partial\cP)\wedge r}{d(x,E_l)\wedge r}
	\end{align*}
	holds; see Lemma \ref{22.03.16.7}.
	Since $${\bf{I}}_{\infty}(x;\Lambda^\pm)=(\mathrm{diam}(\cP))^{\sum_{i=1}^{M_0} \lambda_{v,i}^{\pm}}  \cdot {\bf{I}}(x,\mathrm{diam}(\cP);\Lambda^\pm) \simeq  {\bf{I}}(x,\mathrm{diam}(\cP);\Lambda^\pm) ,
$$  similar result also holds for ${\bf{I}}_{\infty}(x;\Lambda^{\pm})$.
	\end{enumerate}
	\end{remark}

\begin{remark}
When $t-s>0$ is large, Theorem \ref{theorem polyhedron}.(ii) provides a sharper result compared to Theorem \ref{theorem polyhedron}.(i).
More precisely, for any $T_0>0$, there exists a constant $N>0$ such that if $r\geq T_0$, then
$$
{\bf{I}}(x,\sqrt{r};\Lambda^+){\bf{I}}(y,\sqrt{r};\Lambda^-)r^{-3}e^{-\sigma |x-y|^2/r}\leq N {\bf{I}}_\infty(x;\Lambda^+){\bf{I}}_\infty(y;\Lambda^-)e^{-\nu_1 \lambda r}
$$
holds; see Step 2 and Step 3 in the proof of Theorem \ref{theorem polyhedron}.
\end{remark}

\mysection{Proof of Theorem \ref{theorem polyhedral cone}}

The  lemma below will simplify some calculations of what follow in this section.  Recall the notation $\hat{E}:=\bigcup^{N_0}_{i=1} E_i$.
\begin{lemma}\label{21.11.17.10}\,
	
	\begin{enumerate}[label=(\roman*)]
\item For any $i\in\{1,2,\ldots,N_0\}$, 
$$d(x,E_i)\leq \big|x-|x|p_i\big|\leq 2d(x,E_i)$$
holds for all $x\in\cD$.

\item Let $\alpha_i$ be any real numbers for $i=1,\,\ldots,\,N_0$. Then there exists a constant $N=N(\cD, \alpha_1,\ldots,\alpha_{N_0})>0$ such that, for any $x\in\cD$, $r>0$, and $k$ satisfying $d(x,\hat{E})=d(x,E_k)$, the inequalities
\begin{align*}
N^{-1}\Big(\frac{d(x,E_k)\wedge r}{|x|\wedge r}\Big)^{\alpha_k}\leq \prod_{i=1}^{N_0}\Big(\frac{d(x,E_i)\wedge r}{|x|\wedge r}\Big)^{\alpha_i}\leq N\Big(\frac{d(x,E_k)\wedge r}{|x|\wedge r}\Big)^{\alpha_k}
\end{align*}
hold.
\end{enumerate}
\end{lemma}
\begin{proof}
(i) The first inequality is obvious, so we only prove the second inequality.
Considering the conic structure of our domain, we notice that 
$\frac{d(x,E_i)}{|x|}=d\left(\frac{x}{|x|},E_i\right)$, and
the claimed statement is the same as
$$
\Big|\frac{x}{|x|}-p_i\Big|\leq 2\, d\left(\frac{x}{|x|},E_i\right)\,.
$$
One can easily observe this  from the picture below.
\begin{figure}[h]
	\begin{tikzpicture}
		
		\begin{scope}[scale=0.8]
		\clip (-4,-3.2) -- (8,-3.2) -- (8,3.2) -- (-4,3.2);
		\draw[fill=black] (0,0) circle (0.05);
		\node at (0,-0.4) {${\bf{0}}$};
		\draw[fill=black] (3,0) circle (0.05);
		\node at (3.5,-0.3) {$p_i$};
		
		\draw[line width=0.3mm] (0,0) circle (3);
		\draw[line width=0.3mm] (0,0) -- (6,0);
		\node at (6.3,0) {$E_i$};

		\draw[fill=black] (1.8,2.4) circle (0.05);
		\node at (2.2,2.8) {$\frac{x}{|x|}$};
		\draw (1.8,2.4) -- (3,0);
		\draw[dashed] (1.8,2.4) -- (1.8,0);
		
		\draw[fill=black] (0.8,2.891) circle (0.05);
		\draw (0.8,2.891) -- (3,0);
		\draw[dashed] (0.8,2.891) -- (0.8,0);
		
		\draw[fill=black] (-2.4,1.8) circle (0.05);
		\draw (-2.4,1.8) -- (3,0);
		\draw[dashed] (-2.4,1.8) -- (0,0);
		
		\end{scope}
	\end{tikzpicture}
\end{figure}
A rigorous proof is the following:
As $E_i=\{rp_i\,:\,r\geq 0\}$, there exists unique $c\geq 0$ such that $d(x,E_i)=|x-cp_i|$. 
If $c=0$, it follows that
$$
|x-|x|p_i|\leq |x|+|x|\cdot 1 = 2d(x,E_i)\,.
$$
When $c>0$,  the quadratic equation $f(t):=|x-tp_i|^2=t^2-2 (x\cdot p_i)t+|x|^2$ defined on $t>0$ attains the minimum at $t=c$. Hence, we have $x\cdot p_i=c$ and $|x-cp_i|^2=|x|^2-c^2$.
Consequently, we get
$$
|x-|x|p_i|^2=2|x|^2-2c|x|\leq 2|x|^2-2c^2= 2|x-cp_i|^2= 2d^2(x,E_i).
$$

(ii) Since $d(x,E_i)=r d(r^{-1}x,E_i)$ for all $i$, we only prove the case $r=1$.

We note $d(x,E_i)\leq d(x, 0)=|x|$ for any $i$ as $E_i$s share the vertex ${\mathbf{0}}(=V)$. 
Hence, we have
\begin{equation}\label{20220407-1}
\frac{d(x,E_i)\wedge 1}{|x|\wedge 1}\leq 1 \quad\text{for all}\,\,\,x\in\cD\;\text{and all }i.
\end{equation}
On the other hand, we choose a constant $\delta_0\in(0,1)$ (depending on the shape of $\cD$) such that
$$
\delta_0\leq \frac13 \min_{i,\,j\,:\,i\neq j}d(p_j,E_i).
$$
For $x\in\cD$ satisfying $d(x,\hat{E})=d(x,E_k)\geq \delta_0|x|$, we have
$$
\delta_0|x|\leq d(x,\hat{E})\leq d(x,E_i) \quad\text{for all}\;\;\; i\,.
$$
If $d(x,\hat{E})=d(x,E_k)\leq \delta_0|x|$, then by (i) and  the definition of $\delta_0$  we again have
\begin{eqnarray*}
d(x,E_i)&\geq& d\big(|x|p_k,E_i\big)-\big|x-|x|p_k\big|\\
&\geq& |x|\cdot d(p_k,E_i)-2d(x,E_k)\geq \delta_0|x|
\end{eqnarray*}
for all $i\neq k$.
Hence we have
\begin{equation}\label{20220407-2}
\delta_0\leq \frac{d(x,E_i)\wedge 1}{|x|\wedge 1}\quad\text{for all}\,\,\,x\in\cD\,,\,\,\,i\neq k.
\end{equation}
By \eqref{20220407-1} and \eqref{20220407-2}, the assertion of (ii) follows. The lemma is proved.
\end{proof}

The following lemma concerns our estimate near the flat faces of $\cD$.

\begin{lemma}\label{half space.estimate.}
For the operator $\cL$ in \eqref{our operator} conditioned by \eqref{uniform parabolicity}, there exists a constant $N=N(\nu_1,\nu_2)>0$ - depending only on $\nu_1$, $\nu_2$ - fulfilling the following: for any $t_0\in\bR$, $x_0\in\cD$, and small $r>0$ with which $B_r(x_0)$ either completely inside $\cD$ or the $\overline{B^{\cD}_r(x_0)}\cap\partial\cD$ is a subset of only one face without any intersections with the vertex, edges and other faces, if the function $u$  belongs to $\cV_2^{1,0}\big(Q^{\cD}_{r}(t_0,x_0)\big)$, and satisfies
$$
\cL u = 0\quad \text{in}\,\,\, Q^{\cD}_{r}(t_0,x_0)\quad;\quad u|_{(\bR\times\partial\cD)\cap Q_r(t_0,x_0)}=0,
$$
then 
\begin{equation}\label{half space.estimate.inequality.}
|u(t,x)|\leq N\cdot \left(\frac{d(x,\partial\cD)}{r}\wedge 1\right)\cdot\sup_{Q^{\cD}_{r}(t_0,x_0)}|u|
\end{equation}
holds for all $(t,x)\in Q^{\cD}_{r/8}(t_0,x_0)$.
\end{lemma}
\begin{proof}
If $d(x_0,\partial\cD)\geq r/4$, then $d(x,\partial\cD)\ge r/8$ for any $x\in B^{\cD}_{r/8}(x_0)$ as 
$$
|x-y|\ge |x_0-y|-|x-x_0|\ge \frac r4-\frac r8
$$
for any $y\in\partial\cD$. Hence, \eqref{half space.estimate.inequality.} holds with $N=8$ for instance.

If $d(x_0,\partial\cD)< r/4$, then $B_r(x_0)$ intersects $\partial D$. We choose $\xi_0\in\partial\cD$ such that $|\xi_0-x_0|=d(x_0,\partial\cD)$. Then we have
$$
B^{\cD}_{r/8}(x_0)\subset
B^{\cD}_{3r/8}(\xi_0)\subset
B^{\cD}_{r}(x_0)
$$
and $B^{\cD}_{3r/8}(\xi_0)$ is an half open ball inside $\cD$  due to our assumption on $B_r(x_0)$. Now, for any $x\in B^{\cD}_{r/8}(x_0)$, the point $\xi(x)\in \partial \cD$ satisfying $|x-\xi(x)|=d(x,\partial \cD)$ is on $B_{3r/8}(\xi_0)\cap\partial \cD$, and the straight line from $x$ to $\xi(x)$ (except $\xi(x)$) is in $B^{\cD}_{3r/8}(\xi_0)$. Hence, by mean value theorem, for any $(t,x)\in Q^{\cD}_{r/8}(t_0,x_0)\subset Q^{\cD}_{3r/8}(t_0,\xi_0)$ we have
$$
\frac{|u(t,x)|}{d(x,\partial\cD)}=\frac{|u(t,x)-u(t,\xi(x))|}{|x-\xi(x)|}
\leq \sup_{Q_{3r/8}^{\cD}(t_0,\xi_0)}|\nabla u|.
$$
Meanwhile, as $u|_{(\bR\times \partial\cD)\cap Q_r(t_0,x_0)}=0$, by the result like \cite[Proposition 3.1.(ii)]{Kozlov_Nazarov_2009}  there is a constant $N=N(\nu_1,\nu_2)>0$ such that
$$
\sup_{Q_{3r/8}^{\cD}(t_0,\xi_0)}|\nabla u|\leq N\cdot \frac{1}{r}\sup_{Q_{3r/4}^{\cD}(t_0,\xi_0)}|u|.
$$
Hence, 
\begin{equation*}
\frac{|u(t,x)|}{d(x,\partial\cD)}\leq N\cdot \frac{1}{r}\sup_{Q^{\cD}_{3r/4}(t_0,\xi_0)}|u|
\leq N\cdot \frac{1}{r}\sup_{Q^{\cD}_{r}(t_0,x_0)}|u|
\end{equation*}
holds for any $(t,x)\in Q^{\cD}_{r/8}(t_0,x_0)$. This $N$ makes  \eqref{half space.estimate.inequality.} hold.
The lemma is proved.
\end{proof}

The following lemma concerns our estimate near the edges $E_i$ of $\cD$. Recall the critical exponents $\hat{\lambda}^{\pm}_{e,i}(\cD,\cL)$  defined  in  Definition \ref{def.critical lambda.edges.}.
 \begin{lemma}\label{estimate.edges.}\,
 	
\begin{enumerate}[label=(\roman*)]
\item 
For each  $i=1,2, \ldots, N_0$ and any $\lambda\in\big(0,\hat{\lambda}^{+}_{e,i}(\cD,\cL)\big)$ there exists a constant $N=N(\kappa_i,\cD,\cL,\lambda)>0$ fulfilling the following: for any $t_0\in\bR$, $x_0\in\cD$, and $r>0$ with which $E_i$ is the closest edge from $x_0$ and $B_r(x_0)$ does not intersect the vertex and 
all edges except the edge $E_i$,
if the function $u$ belongs to $\cV_2^{1,0}\big(Q^{\cD}_{r}(t_0,x_0)\big)$ and  satisfies
$$
\cL u = 0\quad \text{on}\quad Q^{\cD}_{r}(t_0,x_0)\quad;\quad u|_{(\bR\times \partial\cD)\cap Q_r(t_0,x_0)}=0,
$$
then 
the estimate
\begin{equation*}
|u(t,x)|\nonumber\\
\leq N\left(\frac{d(x,E_i)}{r}\wedge 1\right)^{\lambda}\cdot\frac{d(x,\partial\cD)\wedge r}{d(x,E_i)\wedge r}\cdot\sup_{Q^{\cD}_{r}(t_0,x_0)}|u|
\end{equation*}
holds for all $(t,x)\in Q^{\cD}_{r/32}(t_0,x_0)$.

\item The assertion of (i) also holds with $\tilde{\cL}$ and $ \lambda\in\big(0,\hat{\lambda}^{-}_{e,i}(\cD,\cL)\big)$ in place of $\cL$ and $\lambda\in\big(0,\hat{\lambda}^{+}_{e,i}(\cD,\cL)\big)$ respectively.
\end{enumerate}
\end{lemma}
\begin{proof}
Since $\hat{\lambda}_{e,i}^-(\cD,\cL)=\hat{\lambda}_{e,i}^+(\cD,\tilde{\cL})$, we only prove (i).

\textbf{Step 1.} We first prove that there exists a constant $N(\kappa_i,\cD,\cL,\lambda)>0$ such that the estimate
\begin{equation}
 \label{22.03.10.1}
|u(t,x)|\leq N\cdot \Big(\frac{d\big(x,E_i\big)}{r}\wedge 1\Big)^{\lambda}\cdot\sup_{Q^{\cD}_{r}(t_0,x_0)}|u|
\end{equation}
holds for all $(t,x)\in Q^{\cD}_{r/16}(t_0,x_0)$  provided that $(t_0,x_0)\in \bR\times \cD$, $r>0$, and $u$ satisfies the assumptions prescribed in (i).

For any such $(t_0,x_0)$, $r>0$, and $u$, 
if $d(x_0,E_i)\geq \frac{r}{8}$, then we have
$$
d(x,E_i)\ge 
d(x_0,E_i)-|x-x_0|\ge \frac{r}{16},\quad
\frac{d(x,E_i)}{r}\wedge 1\geq \frac{1}{16},
$$
$$
|u(t,x)|\le 16^{\lambda}\Big(\frac{d\big(x,E_i\big)}{r}\wedge 1\Big)^{\lambda}\cdot \sup_{Q^{\cD}_{r}(t_0,x_0)}|u|
$$
for any $(t,x)\in Q^{\cD}_{r/16}(t_0,x_0)$.  

In case $d(x_0,E_i)< \frac{r}{8}$,
choose $\xi_0\in E_i$ such that $|x_0-\xi_0|=d(x_0,E_i)$. Then
$$
Q^{\cD}_{r/16}(t_0,x_0)\subset Q^{\cD}_{r/4}(t_0,\xi_0),\quad Q^{\cD}_{r/2}(t_0,\xi_0)\subset Q^{\cD}_{r}(t_0,x_0),
$$
and $B_{r/2}(\xi_0)$ does not intersect  the vertex and all edges except the edge $E_i$ as $B_r(x_0)$ does not.
Also, the function $u$ belongs to $\cV_2^{1,0}\big(Q^{\cD}_{r/2}(t_0,\xi_0)\big)$ and  satisfies
$$
\cL u = 0\quad \text{on}\quad Q^{\cD}_{r/2}(t_0,\xi_0)\quad;\quad u|_{(\bR\times\partial\cD)\cap Q_{r/2}(t_0,\xi_0)}=0.
$$
Hence, by the definition of $\hat{\lambda}^{+}_{e,i}(\cD,\cL)$, we have
$$
|u(t,x)|\leq N\left(\frac{d\big(x,E_i\big)}{r}\right)^{\lambda}\cdot \sup_{Q^{\cD}_{r/2}(t_0,\xi_0)}|u|
$$ for any $(t,x)\in Q^{\cD}_{r/16}(t_0,x_0)$.  This certainly yields \eqref{22.03.10.1}.
Note that $d(x,E_i)\le d(x_0,E_i)+|x-x_0|<r/8+r/16<r$ in this case.

\textbf{Step 2.} Recall the number  $r_1\in(0,1)$  in Assumption \ref{ass 1}. 
For $(t,x)\in Q^{\cD}_{r/32}(t_0,x_0)$, consider the following two cases.

\textbf{Case 1)} $d(x,\partial\cD)\geq \frac{r_1}{32}(d(x,E_i)\wedge r)$.

In this case, we have
$$
\frac{d(x,\partial\cD)\wedge r}{d(x,E_i)\wedge r}\geq \frac{r_1}{32},
$$
and by \eqref{22.03.10.1}, our claim of the lemma follows.

\textbf{Case 2)} $d(x,\partial\cD)<\frac{r_1}{32}(d(x,E_i)\wedge r)=: R'$.

In this case, we utilize $\xi(x)\in\partial \cD$ satisfying $|x-\xi(x)|=d(x,\partial\cD)$.  
This $\xi(x)$ satisfies $\xi(x)\not\in \hat{E}$. 
Indeed, if $d(x,\partial\cD) =d(x,E_i)$, then this contradicts $d(x,\partial\cD)< \frac{r_1}{32}(d(x,E_i)\wedge r)$.
If $\xi(x)$ is on an edge $E_j$, then we have
$$
d(x_0,E_j)\le |x_0-\xi(x)|\le |x-x_0|+|x-\xi(x)|<\frac{r}{32}+\frac{r}{32}<r,
$$
meaning that $B_r(x_0)$ also meet $E_j$  and this also contradicts our assumption.

Hence, $\xi(x)$ should be on a face of $\cD$. Then, by Remark \ref{distance to edges},  $B^{\cD}_{r_1d_{\xi(x)}}(\xi(x))$ is an open half ball, where $d_{\xi(x)}=d(\xi(x),\hat{E})$. We claim that
\begin{equation}\label{half ball inside}
8R'<r_1d_{\xi(x)}.
\end{equation}

It is possible that $d_{\xi(x)}$ is not the same as $d(\xi(x),E_i)$. However, we have
\begin{equation}\label{closest edge}
d(\xi(x),E_i)\le 2d_{\xi(x)}.
 \end{equation}
 Otherwise, 
$$
\frac12 d(\xi(x),E_i)>d(\xi(x),\hat{E})
$$
and $d(\xi(x),\hat{E})=d(\xi(x),E_j)$ for some $j\ne i$ and $d(\xi(x),E_j)=|x-t_jp_j|$ for some $t_j>0$. Then we have
\begin{eqnarray*}
|x_0-t_jp_j|&\le& |x-x_0|+|x-\xi(x)|+|\xi(x)-t_jp_j|\\
&<& |x-x_0|+|x-\xi(x)|+\frac12 d(\xi(x),E_i)\\
&\le&  |x-x_0|+|x-\xi(x)|+\frac12\big(|\xi(x)-x|+|x-x_0|+d(x_0,E_i)\big)\\
&<&\frac32\Big(\frac{r}{16}+\frac{r}{16}\Big)+\frac12 r<r,
\end{eqnarray*}
meaning that $B_r(x_0)$ meets $E_j$ and this contradicts our assumption. 

On the other hand, as
$$
d(x,E_i)\leq d(\xi(x),E_i)+|x-\xi(x)|\leq d(\xi(x),E_i)+\frac{1}{32}d(x,E_i),
$$
we have $d(x,E_i) \leq \frac{32}{31}d(\xi(x),E_i)$.

Hence, we have $d(x,E_i)\le \frac{64}{31}d_{\xi(x)}$ by \eqref{closest edge} and
$$
8R'\le \frac{r_1}{4} d(x,E_i)\le \frac{16}{31}r_1d_{\xi(x)}<r_1d_{\xi(x)}.
$$
 This gives \eqref{half ball inside}, and as a consequence $B^{\cD}_{\rho}(\xi(x))$ is also an open half ball for any $0<\rho<8R'$. As $B^{\cD}_{R'}(x)\subset B^{\cD}_{2R'}(\xi(x))$ due to
$$
|y-x|<R'  \quad  \Rightarrow\quad |y-\xi(x)|\le |y-x|+|x-\xi(x)|<R'+R',
$$
we note that $B_{R'}(x)$ is either  completely inside $\cD$ or $\overline{B^{\cD}_{R'}(x)}\cap\partial\cD$ is a subset of one face which also contains $\overline{B^{\cD}_{2R'}(\xi(x))}\cap\partial\cD$, without any intersections with the vertex, edges and other faces.  Now, we note $B_{R'}(x)\subset B_{r/16}(x_0)$ as $x\in B_{r/32}(x_0)$ and
$$
|y-x|<R'\quad \Rightarrow \quad |y-x_0|\le |y-x|+|x-x_0|<R'+\frac{r}{32}<\frac{r}{16}.
$$
Also, we have  $(t-(R')^2,t]\subset (t_0-(r/32)^2,t_0]$ as $t\in (t_0-(r/32)^2,t_0]$ and
$$
t-(R')^2>t-(r/32)^2>t_0-2(r/32)^2>t_0-(r/16)^2.
$$
Hence, $Q^{\cD}_{R'}(t,x)\subset Q^{\cD}_{r/16}(t_0,x_0)$ and the function $u$  belongs to $\cV_2^{1,0}\big(Q^{\cD}_{R'}(t,x)\big)$ and satisfies
$$
\cL u = 0\quad \text{in}\,\,\, Q^{\cD}_{R'}(t,x)\quad;\quad u|_{(\bR\times\partial\cD)\cap Q_{R'}(t,x)}=0.
$$
Therefore we can apply Lemma \ref{half space.estimate.} and have the estimate
\begin{align}\label{22.03.16.1}
|u(t,x)|\leq N\,\frac{d(x,\partial\cD)}{R'}\sup_{Q^{\cD}_{R'}(t,x)}|u|.
\end{align}
Finally, by \eqref{22.03.10.1}  we have
$$
\sup_{Q_{R'}(t,x)}|u| \le N \Big(\frac{d(x,E_i)}{r}\wedge 1\Big)^{\lambda}\sup_{Q_{r}(t_0,x_0)}|u|
$$
and by \eqref{22.03.16.1}
$$
|u(t,x)|\leq N\Big(\frac{d(x,E_i)}{r}\wedge 1\Big)^{\lambda}\cdot\frac{d(x,\partial\cD)\wedge r}{d(x,E_i)\wedge r}\sup_{Q^{\cD}_{r}(t_0,x_0)}|u|
$$
as $d(x,\partial\cD)<R'<r$. 

Combining the results of two cases, we get (i).  The lemma is proved.
\end{proof}

\begin{lemma}\label{21.11.02.2}
Let  $\cL$  and $\tilde{\cL}$ be operators defined  in \eqref{our operator} and  \eqref{another operator} respectively.

\begin{enumerate}[label=(\roman*)]

\item  For any $\lambda_o\in \big(0,\hat{\lambda}_o^+(\cD,\cL)\big)$ and $\lambda_{e,i}\in\big(0,\hat{\lambda}_{e,i}^+(\cD,\cL)\big)$, $i=1,2,\ldots,N_0$, there exists $N=N(\cD,\cL,\lambda_o,\lambda_{e,1},\ldots,\lambda_{e,N_0})>0$ fulfilling the following: for any $(t_0,x_0)\in \bR\times \cD$ and $r>0$, if $u$ belongs to $\cV_2^{1,0}\big(Q^{\cD}_{r}(t_0,x_0)\big)$ and satisfies
$$
\cL u = 0\quad \text{on}\quad Q^{\cD}_{r}(t_0,x_0)\quad;\quad u|_{(\bR\times \partial\cD)\cap Q_r(t_0,x_0)}=0,
$$
then the estimate
\begin{equation}\label{eqn1214123}
|u(t,x)|\leq N\cdot {\bf{I}}(x,r;\Lambda)\cdot \sup_{Q^{\cD}_{r}(t_0,x_0)}|u|
\end{equation}
holds for all $(t,x)\in Q^{\cD}_{r/64}(t_0,x_0)$,
where $\Lambda=(\lambda_o,\lambda_{e,1},\ldots,\lambda_{e,N_0})$ and
\begin{equation*}
{\bf{I}}(x,r;\Lambda)=\left(\frac{d(x,V)\wedge r}{r}\right)^{\lambda_0}\left(\prod_{i=1}^{N_0}\left(\frac{d\big(x,E_i\big)\wedge r}{d\big(x,V\big)\wedge r}\right)^{\lambda_{e,i}}\right)\cdot \frac{d(x,\partial\cD)\wedge r}{d(x,\hat{E})\wedge r}.
\end{equation*}

\item
 If  $\lambda_o\in \big(0,\hat{\lambda}_o^-(\cD,\cL)\big)$ and $\lambda_{e,i}\in\big(0,\hat{\lambda}_c^-(\cD,\cL)\big)$ for $i=1,2,\ldots,N_0$, then
 same statement of (i) holds with $\tilde{\cL}$ in place of $\cL$.
\end{enumerate}
\end{lemma}
\begin{proof} Again, we only prove (i).

\textbf{Step 1.} We claim that there exists a constant $N(\cD,\cL,\lambda_o)>0$ such that the estimate
$$
|u(t,x)|\leq N\cdot \Big(\frac{d\big(x,V\big)}{r}\wedge 1\Big)^{\lambda_o}\cdot\sup_{Q^{\cD}_{r}(t_0,x_0)}|u|
$$
holds for all $(t,x)\in Q^{\cD}_{r/16}(t_0,x_0)$ provided that  $(t_0,x_0)\in \bR\times \cD$, $r>0$, and $u$ satisfies the assumptions prescribed in (i). The poof is similar to Step 1 of the proof of Lemma \ref{estimate.edges.}.

For any such $(t_0,x_0)$, $r>0$, and $u$, 
if $|x|\geq \frac{r}{8}$, we have
$$
|x|\ge 
|x_0|-|x-x_0|\ge \frac{r}{16},\quad
\frac{|x|}{r}\wedge 1\geq \frac{1}{16},
$$
$$
|u(t,x)|\le 16^{\lambda_o}\Big(\frac{d\big(x,E_i\big)}{r}\wedge 1\Big)^{\lambda_o}\cdot \sup_{Q^{\cD}_{r}(t_0,x_0)}|u|
$$
for any $(t,x)\in Q^{\cD}_{r/16}(t_0,x_0)$.  
If $|x|< \frac{r}{8}$,  then
$Q^{\cD}_{r/16}(t_0,x_0)\subset Q^{\cD}_{r/4}(t_0,{\bf{0}})$, $Q^{\cD}_{r/2}(t_0,{\bf{0}})\subset Q^{\cD}_{r}(t_0,x_0)$. Also, $u$ belongs to $\cV_2^{1,0}\big(Q^{\cD}_{r/2}(t_0,{\bf{0}})\big)$ and satisfies
$$
\cL u = 0\quad \text{on}\quad Q^{\cD}_{r/2}(t_0,{\bf{0}})\quad;\quad u|_{(\bR\times\partial\cD)\cap Q_{r/2}(t_0,{\bf{0}})}=0.
$$
Hence, by the definition of $\hat{\lambda}^{+}_{o}(\cD,\cL)$ we have
\begin{equation}\label{estimate.vertex.}
|u(t,x)|\leq N\left(\frac{d\big(x,V\big)}{r}\right)^{\lambda_o}\cdot \sup_{Q^{\cD}_{r/2}(t_0,{\bf{0}})}|u|
\leq N\left(\frac{d\big(x_,V\big)}{r}\right)^{\lambda_o} \cdot \sup_{Q^{\cD}_{r}(t_0,x_0)}|u|
\end{equation}
for any $(t,x)\in Q^{\cD}_{r/16}(t_0,x_0)$. Note that $d(x,V)=|x|<r/8<r$.

\textbf{Step 2.} Now we prove \eqref{eqn1214123}. Recall the numbers  $r_0,r_1\in(0,1)$  we defined in  Assumption \ref{250102724}.  For $(t,x)\in Q^{\cD}_{r/64}(t_0,x_0)$, we consider the following three cases.

{\bf{Case 1)}} $d(x,\hat{E})\geq \frac{r_0}{64}(d(x,V)\wedge r)$ and $d(x,\partial\cD)\geq \frac{r_1}{32}(d(x,\hat{E})\wedge r)$.

In this case, we have
$$
\frac{r_0}{64}\leq \frac{d(x,E_i)\wedge r}{d(x,V)\wedge r}\leq 1\,\,\,\text{for all}\,\,\,i\quad\text{and}\quad\frac{r_1}{32}\leq \frac{d(x,\partial\cD)\wedge r}{d(x,\hat{E})\wedge r}\leq 1.
$$
Therefore \eqref{eqn1214123} follows from \eqref{estimate.vertex.}.

{\bf{Case 2)}} $d(x,\hat{E})\geq \frac{r_0}{64}(d(x,V)\wedge r)$ and $d(x,\partial\cD)< \frac{r_1}{32}(d(x,\hat{E})\wedge r)$.

In this case, we first have
\begin{equation}\label{to edges.}
\frac{r_0}{64}\le\frac{d(x,E_i)\wedge r}{d(x,V)\wedge r}\le 1
\end{equation}
for all $i=1,2,\ldots,N_0$.

We define $R':=\frac{r_1}{32}(d(x,\partial\cD)\wedge r)$ and choose $\xi(x)\in \partial\cD$ satisfying $|x-\xi(x)|=d(x,\partial\cD)$. Since
$$
0<d(x,\partial\cD)\leq \frac{r_1}{32}\big(d((x,\hat{E})\wedge r)\big)<d(x,\hat{E})\,,
$$
$\xi(x)$ can not be the vertex or on any edge. 
Hence, $\xi(x)$ is on a face of $\cD$. 

Note that
$$
d(x,\hat{E})\le d(\xi(x),\hat{E})+|x-\xi(x)|<d(\xi(x),\hat{E})+\frac{1}{32}d(x,\hat{E}),
$$
which yields $d(x,\hat{E})\le \frac{32}{31}d(\xi(x),\hat{E})$ and
$$
16R'\le \frac{r_1}{2}d(x,\hat{E})\le\frac{16}{31}r_1d_{\xi(x)}<r_1d_{\xi(x)}.
$$
Then by Remark \ref{distance to edges}  $B^{\cD}_{\rho}(\xi(x))$ is an open half ball for any $0<\rho<16R'$.

Since $B^{\cD}_{R'}(x)\subset B^{\cD}_{2R'}(\xi(x))$, $B_{R'}(x)$ is either  completely inside $\cD$ or $\overline{B^{\cD}_{R'}(x)}\cap\partial\cD$ is a subset of one face which also contains $\overline{B^{\cD}_{2R'}(\xi(x))}\cap\partial\cD$, without any intersections with the vertex, edges and other faces. 
Now, we note $B_{R'}(x)\subset B_{r/16}(x_0)$ as $x\in B_{r/64}(x_0)$ and
$$
|y-x|<R'\quad \Rightarrow \quad |y-x_0|\le |y-x|+|x-x_0|<R'+\frac{r}{64}<\frac{r}{16}.
$$
Also, we have  $(t-(R')^2,t]\subset (t_0-(r/16)^2,t_0]$ as $t\in (t_0-(r/64)^2,t_0]$ and
$$
t-(R')^2>t-(r/32)^2>t_0-(r/64)^2-(r/32)^2>t_0-(r/16)^2.
$$
Hence, $Q^{\cD}_{R'}(t,x)\subset Q^{\cD}_{r/16}(t_0,x_0)$ and the function $u$  belongs to $\cV_2^{1,0}\big(Q^{\cD}_{R'}(t,x)\big)$ and satisfies
$$
\cL u = 0\quad \text{in}\,\,\, Q^{\cD}_{R'}(t,x)\quad;\quad u|_{(\bR\times\partial\cD)\cap Q_{R'}(t,x)}=0.
$$
We can apply Lemma \ref{half space.estimate.} and have the estimate
\begin{align}\label{22.03.16.100}
|u(t,x)|\leq N\,\frac{d(x,\partial\cD)}{R'}\sup_{Q^{\cD}_{R'}(t,x)}|u|.
\end{align}
In return, by \eqref{estimate.vertex.} we have
$$
\sup_{Q_{R'}(t,x)}|u|\le N\Big(\frac{|x|}{r}\wedge 1\Big)^{\lambda_0}\sup_{Q_{r}(t_0,x_0)}|u|
$$
and by \eqref{22.03.16.100} 
$$
|u(t,x)|\leq N\Big(\frac{|x|}{r}\wedge 1\Big)^{\lambda_0}\cdot\frac{d(x,\partial\cD)\wedge r}{d(x,\hat{E})\wedge r}\sup_{Q_{r}(t_0,x_0)}|u|.
$$
This and  \eqref{to edges.}  prove \eqref{eqn1214123} in this case, too.

{\bf{Case 3)}} $d(x,\hat{E})<\frac{r_0}{64}(d(x,V)\wedge r)=\frac{r_0}{64}(|x|\wedge r)$.

Let us set $R''=\frac{r_0}{64}(|x|\wedge r)$ and choose $\hat{\xi}(x)\in \hat{E}$ satisfying $|x-\hat{\xi}(x)|=d(x,\hat{E})$.
Let $\hat{\xi}(x)\in E_k$; $d(x,\hat{E})=d(x,E_k)$.
As
$$
|x|\leq |\hat{\xi}(x)|+|x-\hat{\xi}(x)|\leq |\hat{\xi}(x)|+\frac{1}{64}|x|,\quad |x|\le \frac{64}{63} |\hat{\xi}(x)|,
$$
we have
$$
32R''\le \frac{r_0}{2}|x|\le r_0|\hat{\xi}(x)|\,.
$$
By Remark \ref{distance to edges}, $B^{\cD}_{r_0|\hat{\xi}(x)|}(\hat{\xi}(x))$ is a spherical wedge with inner angle $\kappa_k$, and therefore $B^{\cD}_{\rho}(\hat{\xi}(x))$ is so for any $0<\rho<32R''$.
Hence, as $B_{R''}(x)\subset B_{2R''}(\hat{\xi}(x))$, $B_{R''}(x_0)$ does not intersect the vertex and 
any edges except the edge $E_k$; note $\hat{\xi}(x)\in B_{R''}(x)$.

Moreover, we have $B_{R''}(x)\subset B_{r/32}(x_0)$ and  $(t-(R'')^2,t]\subset (t_0-(r/32)^2,t_0]$ as $t\in (t_0-(r/64)^2,t_0]$ and
$$
t-(R'')^2>t-(r/64)^2>t_0-(r/64)^2-(r/64)^2>t_0-(r/32)^2.
$$
Hence, $Q^{\cD}_{R''}(t,x)\subset Q^{\cD}_{r/32}(t_0,x_0)$ and the function $u$ belongs to $\cV_2^{1,0}\big(Q^{\cD}_{R''}(t,x)\big)$ and  satisfies
$$
\cL u = 0\quad \text{on}\quad Q^{\cD}_{R''}(t,x)\quad;\quad u|_{(\bR\times\partial\cD)\cap Q_{R''}(t,x)}=0.
$$

We can apply Lemma \ref{estimate.edges.} and have
$$
|u(t,x)|\leq N\cdot\left(\frac{d(x,E_k)}{R''}\wedge 1\right)^{\lambda_{e,k}} \frac{d(x,\partial\cD)\wedge R''}{d(x,E_k)\wedge R''}\sup_{Q^{\cD}_{R''}(t,x)}|u|.
$$
By Step 1,
$$
\sup_{Q_{R''}(t,x)}|u|\leq N \Big(\frac{|x|}{r}\wedge 1\Big)^{\lambda_o}\sup_{Q_{r}(t_0,x_0)}|u|,
$$
and therefore we have
$$
|u(t,x)|\leq N \Big(\frac{|x|}{r}\wedge 1\Big)^{\lambda_o}\cdot \Big(\frac{d(x,E_k)}{r}\wedge 1\Big)^{\lambda_{e,k}}\cdot\frac{d(x,\partial\cD)\wedge r}{d(x,\hat{E})\wedge r}\cdot\sup_{Q^{\cD}_{r}(t_0,x_0)}|u|.
$$
Now, thanks to Lemma~\ref{21.11.17.10},  \eqref{eqn1214123} also holds  in this case.

Combining the results of three cases, we get (i). The lemma is proved.
\end{proof}

\begin{remark}
\label{remark 12.14}
Let $\cD=x_0+\cD_0$ be a translation of $\cD_0$ in spatial direction. Here, $\cD_0$ is a polyhedral cone with vertex $V_0=\bf{0}$ and edges $E_{0,k}$, $k=1,2,\cdots, N_0$.  Put $V=x_0+V_0$, and $E_k=x_0+E_{0,k}$. For $x\in \cD$ we put $\tilde{x}=x-x_0$, then we have
$d(x,E_{k})=d(\tilde{x},E_{0,k})$ and $d(x,V)=d(\tilde{x},V_0)$.   Now suppose 
 $(t_0,x_0)\in \bR\times \cD$, $r>0$, and  $u$ belongs to $\cV_2^{1,0}\big(Q^{\cD}_{r}(t_0,x_0)\big)$ and satisfies
$$
\cL u = 0\quad \text{on}\quad Q^{\cD}_{r}(t_0,x_0)\quad;\quad u|_{(\bR\times \partial\cD)\cap Q_r(t_0,x_0)}=0,
$$
then, since $\cL$  is independent of $x$,  we can apply  Lemma  \ref{21.11.02.2} with $u_0(t,x):=u(t,\tilde{x})$ and $\cD_0$, and get \eqref{eqn1214123} for $\cD:=x_0+\cD_0$, $V:=x_0+V_0$, and $E_k:=x_0+E_{0,k}$. 
The similar result also holds for $\tilde{\cL}$.
\end{remark}

\textbf{Proof of Theorem \ref{theorem polyhedral cone}}

Recall that the the claim of the theorem is the following: if $\lambda_o^{\pm}\in (0, \hat{\lambda}_o^{\pm})$ and $ \lambda_i^{\pm}\in (0,\hat{\lambda}_{e,i}^{\pm})$ for $i=1,2,\cdots,N_0$, then there exist constants  $N=N\big(\cD,\cL, \Lambda^+,\Lambda^-\big)>0$ and $\sigma=\sigma(\nu_1,\nu_2)>0$ such that
 $$
G(t,s,x,y)\leq N\cdot {\bf{I}}\big(x,\sqrt{t-s};\Lambda^+\big){\bf{I}}\big(y,\sqrt{t-s};\Lambda^-\big)\,\frac{1}{(t-s)^{3/2}}e^{-\sigma\frac{|x-y|^2}{t-s}}
$$
for any $(t,s,x,y)\in \bR\times \bR\times \cD\times \cD$ with $t>s$
, where  
$
\Lambda^{\pm}=(\lambda_o^{\pm},\lambda_{e,1}^{\pm},\ldots,\lambda_{e,N_0}^{\pm})$
and
$$
{\bf{I}}(x,r;\Lambda^{\pm})=\left(\frac{d(x,V)}{r}\wedge 1\right)^{\lambda^{\pm}_0}\left(\prod_{i=1}^{N_0}\left(\frac{d\big(x,E_i\big)\wedge r}{d\big(x,V\big)\wedge r}\right)^{\lambda^{\pm}_{e,i}}\right)\cdot \frac{d(x,\partial\cD)\wedge r}{d(x,\hat{E})\wedge r}.
$$

Let us fix $(t,x), (s,y)\in  \bR \times \cD$ with $t>s$. Put $\rho=\frac{1}{2}\sqrt{t-s}$. 

\textbf{Step 1.} We note that by Lemma \ref{250219524} the function $u(t',z):=G(t',s,z,y)$ satisfies
$\cL u=0$ on $Q_{\rho}^{\cD}(t,x)$ and $u|_{(\bR\times\partial\cD)\cap Q_{\rho}(t,x)}= 0$ as $s<t-\rho^2=t-(t-s)/4$.
Hence, by Theorem ~\ref{21.11.02.2} (i) we have
$$
G(t,s,x,y)\le N {\bf{I}}(x,\rho;\Lambda^+)\sup_{(t',z)\in Q^{\cD}_{\rho}(t,x)}G(t',s,z,y).
$$

\textbf{Step 2.} Next, for each fixed $(t',z)\in Q_{\rho}(t,x)$, we define $v(s',w)=G(t',-s',z,w)$.
Then by Lemma \ref{250102534}, $v$ satisfies 
$\tilde{\cL} v=0$ on $Q_{\rho}^{\cD}(-s,y)$ and $u|_{(\bR\times\partial\cD)\cap Q_{\rho}(-s,y)}= 0$ as 
$$
s\le -s'<s+\rho^2<t-\rho^2<t'\le t,
\quad
-t'<-s-\rho^2<s'\le -s.
$$
Hence,  by Theorem ~\ref{21.11.02.2} (ii)   we have
\begin{equation*}
G(t',s,z,y)
\le N {\bf{I}}(y,\rho;\Lambda^-)\sup_{(s',w)\in Q_{\rho}(-s,y)}G(t',-s',z,w).
\end{equation*}

\textbf{Step 3.} Now, we note that for any $(t',z)\in Q_{\rho}(t,x)$ and $(s',w)\in Q_{\rho}(-s,y)$ we have
$$
|z-w|+2r\ge |x-y|,\;\; 2|z-w|^2+2(2\rho)^2\ge |x-y|^2,\;\;
|z-w|^2\ge \frac{|x-y|^2}{2}-(t-s),
$$
$$
\frac12(t-s)=t-s-2\left(\frac{\sqrt{t-s}}{2}\right)^2<t'+s'<t-s.
$$
Hence, using \eqref{250102531}, we obtain 
\begin{equation*}
G(t',-s',z,w)\le N (t'+s')^{-3/2}e^{-\sigma \frac{|z-w|^2}{t'+s'}} 
\le N (t-s)^{-3/2}e^{-\sigma' \frac{|x-y|^2}{t-s}}.
\end{equation*}

\textbf{Step 4.} Combining the results of Steps 1, 2, and 3, we have
\begin{align*}
&\quad G(t,s,x,y)\\
&\le N {\bf{I}}\big(x,\sqrt{t-s}/2;\Lambda^+\big){\bf{I}}\big(y,\sqrt{t-s}/2;\Lambda^-\big)\cdot (t-s)^{-3/2}e^{-\sigma'\frac{|x-y|^2}{t-s}}\\
&\le N {\bf{I}}\big(x,\sqrt{t-s};\Lambda^+\big){\bf{I}}\big(y,\sqrt{t-s};\Lambda^-\big)\cdot (t-s)^{-3/2}e^{-\sigma'\frac{|x-y|^2}{t-s}}
\end{align*}
with the generic constants $N$ having the claimed dependencies. The theorem is proved.
\hfill\qedsymbol{}

\mysection{Proof of Theorem \ref{theorem polyhedron}}

Recall that $V_1,\cdots, V_{M_0}$ are vertices of polyhedron $\cP$ and $\hat{V}=\{V_1,\cdots, V_{M_0}\}$. 
Also, $E_1, \cdots, E_{N_0}$ are egdes of  $\cP$ each of which is a closed interval connecting two vertices, and $\hat{E}=\bigcup_{j=1}^{N_0} E_j$.

The following lemma will be used in subsequent results to simplify some calculations. 
 
\begin{lemma}\label{22.03.16.7}
	Let $\alpha_{i},\,\beta_j\in\bR$ for $i=1,\,\ldots,\,M_0$ and $j=1,\,\ldots,\,N_0$.
	For any $R>0$, there exists a constant $N>0$, depending only on $R$, $\cP$, $(\alpha_i)_{1\leq i\leq M_0}$, and $(\beta_i)_{1\leq i\leq N_0}$  such that if  $x\in \cP$, 
	 $d(x,\hat{V})=d(x,V_k)$ and $d(x,\hat{E})=d(x,E_l)$ for some $k$ and $\ell$, then for any $r\in (0,R]$
	 	\begin{align}\label{240913451}
		\prod_{i=1}^{M_0}\left(\frac{d(x,V_i)}{r}\wedge 1\right)^{\alpha_{i}}\simeq_N \,\prod_{i:V_i\in E_l}\left(\frac{d(x,V_i)}{r}\wedge 1\right)^{\alpha_{i}}\simeq_N \left(\frac{d(x,V_k)}{r}\wedge 1\right)^{\alpha_{k}}
	\end{align}
	and
	\begin{align}\label{240913452}
		\prod_{j=1}^{N_0}\left(\frac{d(x,E_j)\wedge r}{d\big(x,\hat{V}\cap E_j\big)\wedge r}\right)^{\beta_j}\simeq_N \prod_{j:E_j\ni V_k}^{N_0}\left(\frac{d(x,E_j)\wedge r}{d\big(x,V_k\big)\wedge r}\right)^{\beta_j}\simeq_N \left(\frac{d(x,E_l)\wedge r}{d(x,V_k)\wedge r}\right)^{\beta_l}\,.
	\end{align}
	Consequently, 
	\begin{align*}
	\begin{split}
	&\left(\,\prod_{i=1}^{M_0}\left(\frac{d(x,V_i)}{r}\wedge 1\right)^{\alpha_{i}}\right) \left(\,\prod_{j=1}^{N_0}\left(\frac{d(x,E_j)\wedge r}{d\big(x,\hat{V}\cap E_j\big)\wedge r}\right)^{\beta_j}\right)\cdot \frac{d(x,\partial\cP)\wedge r}{d(x,\hat{E})\wedge r}\\
	\simeq_N\,&\left(\frac{d(x,V_k)}{r}\wedge 1\right)^{\alpha_{k}}\left(\frac{d(x,E_l)\wedge r}{d(x,V_k)\wedge r}\right)^{\beta_l}\cdot \frac{d(x,\partial\cP)\wedge r}{d(x,E_l)\wedge r}.
	\end{split}
	\end{align*}
\end{lemma}

\begin{proof}
	Note that for any $a,\,b,\,b'\in(0,\infty)$,
	\begin{align}\label{22.03.16.5}
		1\wedge \frac{b'}{b}\leq \frac{a\wedge b'}{a\wedge b}\leq 1\vee \frac{b'}{b}.
	\end{align}
	
\textbf{Step 1.} We first prove \eqref{240913451}.
	Take $\delta_1\in(0,1]$ such that
	\begin{align}\label{22.03.16.11}
		2\delta_1\leq r_0\wedge \min_{i,j:V_i\notin E_j}d(V_i,E_j)\wedge \min_{i,j:i\neq j}d(V_i,V_j)\,,
	\end{align}
	where $r_0$ is the constant in \eqref{241014218}.
Note that to prove  \eqref{240913451} we only need to consider the case of $d(x,\hat{V})\leq \delta_1$.
	Indeed, if $d(x,\hat{V})>\delta_1$, then for any $i=1,\,\ldots,\,M_0$,
	$$
	\frac{\delta_1}{R}\wedge 1\leq \frac{d(x,V_i)}{r}\wedge 1\leq 1\,,
	$$
	so that \eqref{240913451} holds. Now suppose  $d(x,\hat{V})=d(x,V_k)\leq \delta_1$.
	If $i\neq k$ then 
	$$
	d(x,V_i)\geq d(V_k,V_i)-d(x,V_k)\geq \delta_1\,,
	$$
	which implies that 
	$$
	\frac{\delta_1}{R}\wedge 1\leq \frac{d(x,V_i)}{r}\wedge 1\leq 1, 
	$$
	and this proves  \eqref{240913451}.

\textbf{Step 2.} Next we prove \eqref{240913452}.
	
	\textbf{Case 1)} $d(x,\hat{V})=d(x,V_k)\leq \delta_1$, where $\delta_1$ is the constant in \eqref{22.03.16.11}.
	
	If $E_j\not\ni V_k$, then
	$$
	\delta_1\leq  d(V_k,E_j)-d(x,V_k)\leq  d(x,E_{j})\leq \mathrm{diam}(\cP)
	$$
	and
	$$
	\delta_1\leq \min_{i:i\neq k}d(V_i,V_k)-d(x,V_k)\leq d(x,\hat{V}\cap E_j)\leq \mathrm{diam}(\cP).
$$
	Therefore,  due to \eqref{22.03.16.5}, we have
	$$
	\frac{\delta_1}{\mathrm{diam}(\cP)}\wedge 1\leq \frac{d(x,E_j)\wedge r}{d\big(x,\hat{V}\cap E_j\big)\wedge r}\leq 1 \vee \frac{\mathrm{diam}(\cP)}{\delta_1}.
	$$
	This implies 
	\begin{align*}
		&\prod_{j=1}^{N_0}\left(\frac{d(x,E_j)\wedge r}{d\big(x,\hat{V}\cap E_j\big)\wedge r}\right)^{\beta_j}\nonumber\\
		=&\prod_{j:E_j\ni V_k}\left(\frac{d(x,E_j)\wedge r}{d\big(x,\hat{V}\cap E_j\big)\wedge r}\right)^{\beta_j}\cdot\prod_{j:E_j\not\ni V_k}\left(\frac{d(x,E_j)\wedge r}{d\big(x,\hat{V}\cap E_j\big)\wedge r}\right)^{\beta_j}\\
		 \simeq &\prod_{j:E_j\ni V_k}\left(\frac{d(x,E_j)\wedge r}{d\big(x,\hat{V}\cap E_j\big)\wedge r}\right)^{\beta_j}=\prod_{j: E_j\ni V_k}\left(\frac{d(x,E_j)\wedge r}{d(x,V_k)\wedge r}\right)^{\beta_j}\,.\nonumber
	\end{align*}
	Since $d(x,V_k)\leq \frac{r_0}{2}$ and  $ \cP\cap B(V_k,r_0)=V_k+(\cD_{v,k} \cap B({\bf{0}},r_0))$,  by Lemma~\ref{21.11.17.10} (ii) we have
	$$
	\prod_{j:V_k\in E_j}\left(\frac{d(x,E_j)\wedge r}{d(x,V_k)\wedge r}\right)^{\beta_j}\simeq \left(\frac{d(x,E_l)\wedge r}{d(x,V_k)\wedge r}\right)^{\beta_l}\,.
	$$
	Therefore,  \eqref{240913452} is proved.
	
	\textbf{Case 2)} $d(x,\hat{V})\geq \delta_1$
	
	Since $d(x,\hat{V})\geq \delta_1$ and $r\in (0,R]$, we have
	\begin{align*}
	\frac{\delta_1}{R}\wedge 1\leq \frac{d(x,\hat{V}\cap E_j)}{r}\wedge 1\leq 1
	\end{align*}
	for all $j=1,\,\ldots,\,N_0$.	
	Therefore for any $j=1,\,\ldots,\,N_0$,
	\begin{align}\label{22.03.16.6}
		\left(\frac{d(x,E_j)\wedge r}{d(x,\hat{V}\cap E_j)\wedge r}\right)^{\beta_j}\simeq \left(\frac{d(x,E_j)}{r}\wedge 1\right)^{\beta_j}\,.
	\end{align}

For each $j=1,\,\ldots,\,N_0$, put 
	$$
	\widetilde{E}_j:=E_j\setminus \bigcup_{i=1}^{M_0}B\Big(V_i,\frac{\delta_1}{2}\Big)\,,
	$$
	so that $\widetilde{E}_j\cap E_{j'}=\emptyset$ for all $j,\,j'$ with $j\neq j'$.
	Take $\delta_2>0$ satisfying
	$$
	2\delta_2\leq \delta_1\wedge\min_{j,j':j\neq j'}d(\widetilde{E}_j,E_{j'})\,.
	$$ 
	
	\textbf{Case 2.1)} $d(x,\hat{V})\geq \delta_1$ and $d(x,\hat{E})=d(x,E_l)\leq \delta_2$.
	
	We take $\xi(x)\in E_l$ such that $|x-\xi(x)|=d(x,E_l)$.
	Then
	$$
	d\big(\xi(x),\hat{V}\big)\geq d(x,\hat{V})-\big|x-\xi(x)\big|=d(x,\hat{V})-d(x,E_l)\geq \delta_1-\delta_2\geq \frac{\delta_1}{2}.
	$$
	Therefore $\xi(x)\in \widetilde{E}_l$, which implies $d(x,E_l)=d(x,\widetilde{E}_l)$.
	From this, we get  for all $j\neq l$, 
	$$
	d(x,E_j)\geq d\big(\xi(x),E_j\big)-\big|x-\xi(x)\big|\geq d(\widetilde{E}_l,E_j)-d\big(x,E_l\big)\geq \delta_2\,.
	$$
	Therefore  for any $j\neq l$, 
	\begin{align}\label{240913453}
	\frac{\delta_2}{R}\wedge 1\leq\frac{d(x,E_j)}{r}\wedge 1\leq 1.
	\end{align}
	Combining \eqref{240913453}, \eqref{22.03.16.6} and  the fact $d(x,\hat{E})=d(x,E_l)$,  we get \eqref{240913452}.	
	
	\textbf{Case 2.2)} $d(x,\hat{V})\geq \delta_1$ and $d(x,\hat{E})\geq \delta_2$.

Due to $r\leq R$, we have
\begin{align}\label{241014304}
	\frac{\delta_2}{R}\wedge 1\leq \frac{d(x,E_j)}{r}\wedge 1\leq 1\quad\text{for all}\quad j=1,\,\ldots,\,M_0.
\end{align}
Therefore, \eqref{240913452} follows from \eqref{22.03.16.6} and \eqref{241014304}.
\end{proof}

\begin{lemma}\label{21.11.22.3}
	There exists  a constant $R_0=R_0(\cP)>0$ such that
	
	\begin{enumerate}[label=(\roman*)]
	\item 	for any  $\lambda_{v,i}\in \big( 0,\hat{\lambda}_{v,i}^+(\cP,\cL)\big)$ and $\lambda_{e,j}\in\big(0,\hat{\lambda}_{e,j}^+(\cP,\cL)\big)$, $i=1,\,\ldots,\,M_0$ and $j=1,\,\ldots,\,N_0$,
		there exists $N=N\big(\cP,\cL,(\lambda_{v,i})_{1\leq i\leq M_0},\,(\lambda_{e,j})_{1\leq j\leq N_0}\big)$ fulfilling the following: for any $(t_0,x_0)\in \bR\times \cP$ and $0<r\leq R_0$, if $u$ belongs to $\cV_2^{1,0}\big(Q^{\cP}_{r}(t_0,x_0)\big)$ and satisfies
	\begin{align*}
		& \cL u=0 \quad\text{on}\,\,\,\,Q^{\cP}_{r}(t_0,x_0)\quad\text{and}\quad u|_{(\bR\times \partial\cP)\cap Q_{r}(t_0,x_0)}=0
	\end{align*}
	then, the estimate
	\begin{align}\label{22.03.15.2}
		|u(t,x)|\leq N\cdot {\bf{I}}(x,r;\Lambda)\sup_{Q^{\cP}_{r}(t_0,x_0)}|u|
	\end{align}
	holds for all $(t,x)\in Q_{r/64}(t_0,x_0)$, where $\Lambda:=\big((\lambda_{v,i})_{1\leq i\leq M_0},\,(\lambda_{e,j})_{1\leq j\leq N_0}\big)$  and 
	\begin{align*}
	&{\bf{I}}(x,r;\Lambda)\\
	&:=\left(\,\prod_{i=1}^{M_0}\left(\frac{d(x,V_i)}{r}\wedge 1\right)^{\lambda_{v,i}}\right) \left(\,\prod_{j=1}^{N_0}\left(\frac{d(x,E_j)\wedge r}{d\big(x,\hat{V}\cap E_j\big)\wedge r}\right)^{\lambda_{e,j}}\right) \frac{d(x,\partial\cP)\wedge r}{d(x,\hat{E})\wedge r} ; \nonumber
	\end{align*}

	\item	if  $\lambda_{v,i}\in \big( 0,\hat{\lambda}_{v,i}^-(\cP,\cL)\big)$ and $\lambda_{e,j}\in\big(0,\hat{\lambda}_{e,j}^-(\cP,\cL)\big)$ for $i=1,\,\ldots,\,M_0$ and $j=1,\,\ldots,\,N_0$, then the same statement of (i) holds with $\tilde{\cL}$ in place of $\tilde{\cL}$.
\end{enumerate}
\end{lemma}

\begin{proof}
	Due to the similarity we only prove (i).
	
	We take positive constants $r_0,\,r_1,\,r_2\in(0,1]$ from  Assumption \ref{ass 2}. Let $\delta_1\in(0,r_0/2]$ be a constant satisfying \eqref{22.03.16.11}.
	
	Let $d(x,\hat{V})=d(x,V_k)$ and $d(x,\hat{E})=d(x,E_l)$.
	
	\textbf{Case 1)} $d(x,\hat{V})=d(x,V_k)\leq \frac{\delta_1}{2}$.
	
	Let $r\leq \frac{\delta_1}{2}$. 
	By Lemma \ref{22.03.16.7}, 
	\begin{align}\label{22.03.16.13}
		{\bf{I}}(x,r;\Lambda)\simeq \left(\frac{d(x,V_k)}{r}\wedge 1\right)^{\lambda_{v,k}}\cdot\left(\,\prod_{j:E_j\ni V_k}\left(\frac{d(x,E_j)\wedge r}{d(x,V_k)\wedge r}\right)^{\lambda_{e,j}}\right)\cdot \frac{d(x,\partial\cP)\wedge r}{d(x,\hat{E})\wedge r}\,.
	\end{align}
	Since $d(x,\hat{V})=d(x,V_k)\leq \frac{\delta_1}{2}\leq \frac{r_0}{4}$, we have $B(x,r)\subset B\big(V_k,\frac{r_0}{2}\big)$. Also,
	 by \eqref{241014218}	
	$$
	    \cP\cap B(V_k,r_0)=(V_k+\cD_{v,k}) \cap B(V_k,r_0).
	    $$
	    Therefore \eqref{22.03.15.2} follows from \eqref{22.03.16.13} and Lemma~\ref{21.11.02.2}.(i) (also see Remark \ref{remark 12.14}).

	\textbf{Case 2)} $d(x,\hat{V})\geq \frac{\delta_1}{2}$ and $d(x,\hat{E})=d(x,E_l)<\frac{\delta_1r_1}{8}$.
	
	Let $r\leq \delta_1r_1/16$.
	Since
	$$
	\frac{8}{r_1}\wedge1\leq \frac{d(x,V_i)}{r}\wedge 1\leq 1\quad\text{for all}\,\,\,i=1,\,\ldots,\,N_0,
	$$
	Lemma~\ref{22.03.16.7} implies 
	\begin{align}\label{22.03.16.14}
		{\bf{I}}(x,r;\Lambda)\simeq \left(\frac{d(x,E_l)}{r}\wedge 1\right)^{\lambda_{e,l}}\cdot \frac{d(x,\partial\cD)\wedge r}{d(x,E_l)\wedge r}\,.
	\end{align}
	Take $\xi(x)\in E_l$ satisfying $|x-\xi(x)|=d(x,E_l)=d(x,E)$.
	Then we have
	$$
	d(\xi(x),\hat{V})\geq d(x,\hat{V})-|x-\xi(x)|\geq \frac{3}{8}\delta_1
	$$
	and
	$$
	B(x,r)\subset B\Big(\xi(x),\frac{3\delta_1r_1}{16}\Big)\subset B\Big(\xi(x),\frac{r_1d(\xi(x),\hat{V})}{2}\Big).
	$$
	Considering
	\begin{align*}
		\cP\cap B\big(\xi(x),r_1d(\xi(x),\hat{V})\big) = (\xi(x)+\cD_{v,l})\cap B\big(\xi(x),r_1d(\xi(x),\hat{V})\big),
	\end{align*}
	 we  apply Lemma~\ref{21.11.02.2}.(i) (also see Remark \ref{remark 12.14}) and use \eqref{22.03.16.14} to get \eqref{22.03.15.2}.
	
	\textbf{Case 3)} $d(x,\hat{E})\geq \frac{\delta_1r_1}{8}$ and $d(x,\partial\cP)< \frac{\delta_1r_1r_2}{32}$.
	
	Let $r\leq \frac{\delta_1r_1r_2}{64}$.
	By Lemma~\ref{22.03.16.7} and 
	$$
	\frac{8}{r_2}\wedge 1\leq \frac{d(x,E_l)}{r}\wedge 1\Big(=\frac{d(x,\hat{E})}{r}\wedge 1\Big)\leq \frac{d(x,V_k)}{r}\wedge 1\Big(=\frac{d(x,\hat{V})}{r}\wedge 1\Big)\leq  1\,,
	$$
	we have
	\begin{align}\label{22.03.16.15}
		{\bf{I}}(x,r;\Lambda)\simeq \frac{d(x,\partial\cD)}{r}\wedge 1\,.
	\end{align}
	Choose $\xi'(x)\in\partial\cD$ satisfying $|x-\xi'(x)|=d(x,\partial\cP)$.
	Then
	$$
	d(\xi'(x),\hat{E})\geq d(x,\hat{E})-|x-\xi'(x)|\geq \frac{3\delta_1r_1}{32},
	$$
	and therefore
	$$
	B(x,r)\subset B\Big(\xi'(x),\frac{3\delta_0r_1r_2}{64}\Big)\subset B\Big(\xi'(x),\frac{r_2d(\xi'(x),V)}{2}\Big)\,.
	$$
	Note that
$$
	\cP\cap B\big(\xi'(x),r_2d(\xi'(x),\hat{V})\big)=\big(V_i+\cD_i\big)\cap B\big(\xi'(x),r_2d(\xi'(x),\hat{V})\big).
	$$
Due to \eqref{22.03.16.15} and Lemma \ref{21.11.02.2} (also see Remark \ref{remark 12.14}) for $\cD:=V_i+\cD_i$ and $r:=\frac{3\delta_0r_1r_2}{64}$, we get \eqref{22.03.15.2}.
	
	\textbf{Case 4)} $d(x,\partial\cP)\geq \frac{\delta_1r_1r_2}{32}$.
	
	In this case, 
	$$
	\frac{\delta_1r_1r_2}{32}\leq 	d(x,\partial\cP)\leq d(x,\hat{E})\leq d(x,\hat{V})\,,
	$$
	so we have
	$I(x,r;\Lambda^+)\simeq 1$ for all $0<r\leq \frac{\delta_1r_1r_2}{64}$.
	Therefore  \eqref{22.03.15.2}  trivially holds.

	Taking $R_0=\delta_1r_1r_2/64$, we complete the proof.
\end{proof}

\begin{lemma}\label{22.03.15.4}
	Let $u_0\in L_2(\cP)$ and $u\in\mathring{\cV}_2^{1,0}((0,T]\times \cP)$ satisfy (see Definition \ref{250219615})
	$$
	\cL u=0 \quad \text{on}\quad (0,T]\times \cP; \quad\quad u(0)=u_0\,.
	$$
	Then for any $t\in(0,T]$, 
	$$
	\|u(t)\|_{L_2(\cP)}\leq e^{-\nu_1\lambda t}\|u_0\|_{L_2(\cP)},
	$$
	where $\lambda$ is the first eigenvalue of Laplace operator on $\cP$ with the zero Dirichlet boundary condition.
\end{lemma}

\begin{proof}
	Put $v(t,x)=e^{\nu_1\lambda t}u(t,x)$, so that
	\begin{align*}
		v_t=\cL v+\nu_1\lambda v\quad;\quad v(0)=u_0\,.
	\end{align*}
	By the energy inequality (see, \textit{e.g.}, \cite[Lemma III.2.1]{LSU_1968}), we have
	$$
	\frac{1}{2}\|v(t)\|_{L_2(\cP)}^2\leq \frac{1}{2}\|u_0\|_{L_2(\cP)}^2+\nu_1\lambda\int_0^t\int_{\cP}v^2dx\,ds-\int_0^t\int_{\cP} a_{ij}(s) D_ivD_jvdx\,ds\,.
	$$
	By the uniform parabolicity of $a_{ij}(s)$ and the Poincare inequality, we have
	$$
	\int_0^t\int_{\cP} a_{ij}(s) D_ivD_jvdx\,ds\geq \nu_1\int_0^t\int_{\cP}|\nabla v|^2dx\,ds\geq \nu_1\lambda \int_0^t\int_{\cP}v^2dx\,ds
	$$
	Consequently, we get
	$$
	e^{\nu_1\lambda t}\|u(t)\|_{L_2(\cP)}=\|v(t)\|_{L_2(\cP)}\leq \|u_0\|_{L_2(\cP)}
	$$
	and proof is completed.
\end{proof}

\begin{lemma}\label{22.03.15.6}
	For any $T_0>0$, there exists $N'=N'(\nu_1,\,\nu_2,\,T_0)>0$ so that
	\begin{align}\label{241021442}
	G_{\cP}(t,s,x,y)\leq N'e^{-\nu_1\lambda (t-s)}
	\end{align}
	for all $t>s+T_0$ and $x,y\in\cP$.
\end{lemma}
\begin{proof}
	It follows from \eqref{250102531} that for any fixed $(s,y)\in\bR\times \cD$,
	$$
	\Big\|G_{\cP}\Big(s+\frac{T_0}{2},s,\,\cdot\,,y\Big)\Big\|_{L_2(\cP)}\leq N(\nu_1, \nu_2,T_0)\,.
	$$
	By Lemma~\ref{22.03.15.4} applied for $u(t,\cdot)=G_{\cP}\left(s+\frac{T_0}{2}+t,s,\cdot,y\right)$, we have
	\begin{align}\label{241224118}
	\|G_{\cP}\big(r,s,\,\cdot\,,y\big)\|_{L_2(\cP)}\leq Ne^{-\nu_1\lambda (r-s-T_0/2)}\quad\text{for all}\quad r\geq s+\frac{T_0}{2}\,.
	\end{align}
	Similarly we have
	\begin{align}
	\|G_{\cP}(t,r,x,\,\cdot\,)\|_{L_2(\cP)}\leq N e^{-\nu_1\lambda (t-r-T_0/2)}\quad\text{for all}\quad t\geq r+\frac{T_0}{2}\,.
	\end{align}
	For $t>s+T_0$, using \eqref{250114346} with $r:=\frac{t+s}{2}$, we have
	\begin{align}\label{241224117}
		\begin{aligned}
		G_{\cP}(t,s,x,y)\,&=\int_{\cP}G_{\cP}(t,r,x,z)G_{\cP}(r,s,z,y)dz\\
		&\leq \|G_{\cP}(t,r,x,\,\cdot\,)\|_{L_2(\cP)}\|G_{\cP}(r,s,\,\cdot\,,y)\|_{L_2(\cP)}\,.
		\end{aligned}
	\end{align}
By combining \eqref{241224118} - \eqref{241224117}, we complete the proof.
	\end{proof}

\textbf{Proof of Theorem \ref{theorem polyhedron}}

 Let  $\hat{\lambda}_{v,i}^\pm(\cP,\cL)$ and $ \hat{\lambda}_{e,j}^\pm(\cP,\cL) $ be critical exponents defined in Definiition \ref{241227911}. We remind that the the claim of the theorem is the following:  if 
 $\lambda_{v,i}^\pm\in\big(0,\hat{\lambda}_{v,i}^\pm(\cP,\cL)\big)$ and $\lambda_{e,j}^\pm\in \big(0,\hat{\lambda}_{e,j}^\pm(\cP,\cL)\big)$, then
\vspace{2mm}

$\bullet$   there exist  constants $N=N\big(\cP,\cL, \Lambda^\pm\big)>0$ and $\sigma=\sigma(\nu_1,\nu_2)>0$ so that the inequality 
	\begin{align}\label{24102144011}
		G_{\cP}(t,s,x,y)\leq N\, {\bf{I}}\left(x,\sqrt{t-s};\Lambda^+\right){\bf{I}}\left(y,\sqrt{t-s};\Lambda^-\right)\cdot (t-s)^{-3/2}e^{-\sigma\frac{|x-y|^2}{t-s}}
	\end{align}
	holds for any $t> s$ and $x,\,y\in \cP$ , where  $\Lambda^{\pm}:=\big((\lambda_{v,i}^{\pm})_{1\leq i\leq M_0},\,(\lambda_{e,j}^{\pm})_{1\leq j\leq N_0}\big)$ and 
	\begin{align*}
	&{\bf{I}}(x,r;\Lambda^{\pm})\\
	&:=\left(\,\prod_{i=1}^{M_0}\left(\frac{d(x,V_i)}{r}\wedge 1\right)^{\lambda^{\pm}_{v,i}}\right)\cdot \left(\,\prod_{j=1}^{N_0}\left(\frac{d(x,E_j)\wedge r}{d\big(x,\hat{V}\cap E_j\big)\wedge r}\right)^{\lambda^{\pm}_{e,j}}\right)\cdot \frac{d(x,\partial\cP)\wedge r}{d(x,\hat{E})\wedge r}; \nonumber
	\end{align*}

	$\bullet$  for any $T_0>0$, there exists a constant $N=N\big(\cP,\cL,\Lambda^{\pm},T_0\big)>0$ such that  if $x,y\in \cP$ and  $t-s\geq T_0$, then
	\begin{align*}
		G_{\cP}(t,s,x,y)\leq N\cdot {\bf{I}}_{\infty}(x;\Lambda^+){\bf{I}}_{\infty}(y;\Lambda^-)e^{-\nu_1\lambda (t-s)}
	\end{align*}
	where
	$$
	{\bf{I}}_{\infty}(x;\Lambda^\pm):=\left(\,\prod_{i=1}^{N_0}d(x,V_i)^{\lambda_{v,i}^\pm}\right)\cdot\left(\,\prod_{i=1}^N\left(\frac{d(x,E_i)}{d\big(x,\hat{V}\cap E_i\big)}\right)^{\lambda_{e,j}^\pm}\right)\cdot \frac{d(x,\partial\cP)}{d(x,\hat{E})}\,.
	$$

Now we begin the proof of the theorem. 
	Note that
\begin{align}\label{241021434}
{\bf{I}}_{\infty}(x;\Lambda^\pm)=(\mathrm{diam}(\cP))^{\sum_{i=1}^{M_0} \lambda_{v,i}^{\pm}} \cdot {\bf{I}}(x,\mathrm{diam}(\cP);\Lambda^\pm)\,.
\end{align}
Moreover, \eqref{22.03.16.5} implise that for any $r,\,r'>0$ satisfying $0<\delta\leq \frac{r}{r'}\leq \delta^{-1}$, there exists a constant $N=N(\Lambda^{\pm},\delta)>0$ such that 
\begin{align}\label{241021435}
N^{-1}{\bf{I}}(x,r;\Lambda^\pm)\leq {\bf{I}}(x,r';\Lambda^\pm)\leq N\,{\bf{I}}(x,r;\Lambda^\pm)\,.
\end{align}

	\textbf{Step 1.} 
	Let $R_0$ be a constant in Lemma~\ref{21.11.22.3}.
	We first prove 
	\begin{align*}
		G_{\cP}(t,s,x,y)\leq N\, {\bf{I}}\left(x,r;\Lambda^+\right){\bf{I}}\left(y,r;\Lambda^-\right)\cdot (t-s)^{-3/2}e^{-\sigma\frac{|x-y|^2}{t-s}}\,,
	\end{align*}
	where $r:=\frac{\sqrt{t-s}}{2}\wedge R_0$.
	For a fixed $(s,y)\in \bR\times \cP$, put $u(t',z)=G_{\cP}(t',s,z,y)$ so that 
	$u$ belongs to $\cV_2^{1,0}\big((s+\epsilon,\infty)\times \cP\big)$ for any $\epsilon>0$ and satisfies
	$$
	\cL u=0 \quad\text{on}\,\,\,\, (s,\infty)\times \cP\quad\text{and}\quad u|_{(s,\infty)\times \partial\cP}=0.
	$$

	By Lemma~\ref{21.11.22.3}, 
	$$
	G_\cP(t,s,x,y)\leq N {\bf{I}}(x,r;\Lambda^\pm)\sup_{(t',z)\in Q_{r}(t,x)}G_\cP(t',s,z,y).
	$$
	For a fixed   $(t',z)\in Q_r(t,x)$,  if we put $v(s',w):=G_\cP(t',s-s',z,w)$, then by Lemma \ref{250102534}
	$$
	\frac{\partial v(s',w)}{\partial s'}-\sum_{i,j=1}^3a_{ij}(s-s') D_{ij}v(s',w)=0\quad\text{on}\quad (s-t',\infty)\times \cP.
	$$
	Using Lemma~\ref{21.11.22.3} again, we get
	\begin{align}\label{22.03.15.9}
		G_\cP(t,s,x,y)\leq N {\bf{I}}(x,r;\Lambda^+){\bf{I}}(y,r;\Lambda^-)\sup_{\substack{(t',z)\in Q_{r}(t,x)\\(s',w)\in Q_{r}(0,y)}}G_\cP(t',s-s',z,w).
	\end{align}
	
For $(t',z)\in Q_{r}(t,x)$ and $(s',w)\in Q_{r}(0,y)$, 
	we have
	\begin{align}\label{22.03.15.10}
		|z-x|^2\geq \frac{|x-y|^2}{2}-4(t-s)\quad\text{and}\quad \frac{t-s}{2}\leq t'-s'+s\leq 2(t-s)\,.
	\end{align}
	By combining \eqref{22.03.15.9} and \eqref{22.03.15.10} with \eqref{250102531}, we obtain that
	\begin{align}\label{241021439}
	G_\cP(t,s,x,y)\leq N{\bf{I}}(x,r;\Lambda^+){\bf{I}}(y,r;\Lambda^-)(t-s)^{-3/2}e^{-\sigma\frac{|x-y|^2}{t-s}}
	\end{align}
	for some $\sigma$ depending only on our dimension $d=3$, and $N$ depending only on $\cP$ and $\cL$.
	\eqref{241021439} and \eqref{241021435} imply \eqref{24102144011} for $s<t\leq s+\mathrm{diam}(\cP)^2$ and $x,\,y\in\cP$, indeed,
	$$
	\frac{R_0}{\mathrm{diam}(\cP)}\wedge \frac{1}{2}\leq \frac{r}{\sqrt{t-s}}\leq \frac{1}{2}\,.
	$$
	
	\textbf{Step 2.} Let us assume that $t\geq s+T_0$.
	For $r=\frac{\sqrt{t-s}}{2}\wedge R_0$, we have
	$$
	\frac{\sqrt{T_0}}{2}\wedge R_0\leq r\leq R_0.
	$$
	Therefore, by \eqref{241021434} and \eqref{241021435}, we have
	\begin{align}\label{22.03.15.7}
		{\bf{I}}(x,r;\Lambda^\pm)\simeq {\bf{I}}_{\infty}(x;\Lambda^\pm)\,.
	\end{align}
	For $t'\in(t-r^2,t]$ and $s'\in (-r^2,0]$, we have $t'-s+s'\geq t-s-2R_0^2$.
	Therefore, by \eqref{241021442}, we obtain 
	\begin{align*}
		\sup_{\substack{(t',z)\in Q_{r}(t,x)\\(s',w)\in Q_{r}(0,y)}}G(t',s-s',z,w)&\leq N \sup_{\substack{t'\in(t-r^2,t]\\s'\in(-r^2,0]}}e^{-\nu_1\lambda(t'-s+s')}\leq N e^{-\nu_1\lambda(t-s)}\,.
	\end{align*}
	By combining \eqref{22.03.15.9}, \eqref{22.03.15.7} and \eqref{22.03.15.6}, we conclude that there exists a constant $N>0$ such that
	$$
	G(t,s,x,y)\leq N\cdot  {\bf{I}}_{\infty}(x;\Lambda^+){\bf{I}}_{\infty}(y;\Lambda^-)e^{-\nu_1\lambda (t-s)}
	$$
	for all $t\geq s+T_0$ and $x,\,y\in\cP$.
	
	\textbf{Step 3.} It is only remained to prove \eqref{24102144011} for $t>s+\mathrm{diam}(\cP)^2$.
	By the result in Step 2, there exists a constant $N>0$ such that
	$$
	G(t,s,x,y)\leq N{\bf{I}}_{\infty}(x;\Lambda^+){\bf{I}}_{\infty}(y;\Lambda^-)e^{-\nu_1\lambda(t-s)}.
	$$
	Since $\sqrt{t-s}\geq \mathrm{diam}(\cP)$,
	by \eqref{241021434} and \eqref{241021435}, we have
	\begin{align*}
	{\bf{I}}_{\infty}(x;\Lambda^\pm)&= {\bf{I}}(x,\mathrm{diam}(\cP);\Lambda^\pm)\cdot \prod_{i=1}^{N_0}\mathrm{diam}(\cP)^{\lambda^\pm_{v,i}}\\
	&={\bf{I}}\big(x,\sqrt{t-s}\wedge \mathrm{diam}(\cP);\Lambda^\pm\big)\cdot \prod_{i=1}^{N_0}\mathrm{diam}(\cP)^{\lambda^\pm_{v,i}}\\
	&= {\bf{I}}(x,\sqrt{t-s};\Lambda^\pm)\cdot \prod_{i=1}^{N_0}\left(\frac{\sqrt{t-s}\cdot \mathrm{diam}(\cP)}{\sqrt{t-s}\wedge \mathrm{diam}(\cP)}\right)^{\lambda_{v,i}^\pm}\,.
	\end{align*}
	Moreover, one can observe that there exists $N=N(\mathrm{diam}(\cP),\Lambda^\pm)>0$ such that
	\begin{align*}
	&\left( \prod_{i=1}^{N_0}\big(\sqrt{t-s}\vee \mathrm{diam}(\cP)\big)^{\lambda_{v,i}^+}\right) \left(\prod_{i=1}^{N_0}\big(\sqrt{t-s}\vee \mathrm{diam}(\cP)\big)^{\lambda_{v,i}^-}\right)e^{-\nu_1\lambda (t-s)}\\
	\leq \,&N\,(t-s)^{-3/2}e^{-\sigma'\frac{\mathrm{diam}(\cP)^2}{t-s}}
	\end{align*}
	whenever $t-s\geq \mathrm{diam}(\cP)$.
	Therefore we have
	\begin{align*}
	&{\bf{I}}_{\infty}(x;\Lambda^+){\bf{I}}_{\infty}(y;\Lambda^-)e^{-\nu_1\lambda(t-s)}\\
	\leq N \,&{\bf{I}}(x,\sqrt{t-s};\Lambda^+){\bf{I}}(y,\sqrt{t-s};\Lambda^-)(t-s)^{-3/2}e^{-\sigma'\frac{|x-y|^2}{t-s}}\,.
	\end{align*}
The proof is completed.
\hfill\qedsymbol{}

\appendix

\mysection{Auxiliary results of Green's functions}\label{Green function}
In this section, we collect elementary properties of Green's function on general domains $\cO\subset \bR^d$, $d\geq 1$, for the operator $\partial_t -\sum_{i,j=1}^d\alpha_{ij}(t)D_{ij}$ under the zero Dirichlet boundary condition.
The coefficients $\alpha_{ij}=\alpha_{ij}(t)$ are time-measurable functions on $\bR$, $\alpha_{ij}=\alpha_{ji}$, and satisfy the uniform parabolicity condition: there exist constants $0<\nu_1\leq \nu_2<\infty$ such that 
\begin{align*}
\nu_1|\xi|^2\leq \sum_{i,j=1}^d\alpha_{ij}(t)\xi^i\xi^j\leq \nu_2|\xi|^2\quad\forall\,\,\xi\in\bR^d\,,\,\,t\in\bR\,.
\end{align*}
We denote
$$
L:=\frac{\partial}{\partial t}-\sum_{i,j=1}^d \alpha_{ij}(t)D_{ij}\quad\text{and}\quad \tilde{L}:=\frac{\partial}{\partial t}-\sum_{i,j=1}^d \alpha_{ij}(-t)D_{ij}
$$

\begin{defn}\label{250219615}
Let $I:=(a,b]\cap \bR$, $-\infty\leq a<b\leq \infty$.

\begin{enumerate}[label=(\roman*)]
\item For  $u\in\cV_2^{1,0}(I\times \cO)$, we say that $L u=f$ in the sense of distributions on $I\times \cO$, if for any $\phi\in C_c^\infty(I\times \cO)$, 
\begin{align*}
\int_{\cO} u(b,x)\phi(b,x)dx+\int_{I\times \cO}\left(-u\partial_t\phi+\sum_{i,j=1}^d\alpha_{ij}D_{i}uD_{j}\phi\right)dxdt=\int_{I\times \cO}f\phi\,dxdt\,,
\end{align*}
where $u(\infty,x)\phi(\infty,x):=0$

\item Let $E$ be a relatively open set of $I\times \partial \cO$. 
For $u\in \cV_2^{1,0}\big(I\times \cO \big)$, we write $u|_{E}=0$ if the following holds: for any $\zeta\in C_c^{\infty}(I\times \overline{\cO})$ with $\mathrm{supp}(\zeta)\cap \big(I\times \partial \cO \big)\subset E$, $\zeta u\in \mathring{\cV}_2^{1,0}\big(I\times \cO\big)$.
\end{enumerate}
\end{defn}

The following lemmas are either explicitly provided in or can be inferred from \cite{ChoDongKim2007,KimXu2022}:

\begin{lemma}[Theorem 2.7 of \cite{ChoDongKim2007}]\label{241226452}
For any  connected open set $\cO\subset \bR^d$, there exists a unique Green's function $G_\cO(t,s,x,y)$ on $\bR\times \bR\times \cO\times \cO$ which is continuous in $\{(t,s,x,y)\,:\,t\neq s \,\,\text{or}\,\, x\neq y\}$, satisfies $G_\cO(t,s,x,y)=0$ for $t<s$, and has the property that $G_\cO(t,\cdot,x,\cdot)$ is locally integrable in $\bR\times \cO$ and
$$
L G(\cdot,s,\cdot,y)=\delta(\cdot-s)\delta(\cdot-y) \quad \text{in}\quad \bR\times \cO\quad  ; \quad   G|_{\bR\times \partial \cO}=0
$$
in the sense of distributions.
More precisely, for any $f=f(s,y)\in C_c^\infty(\bR\times \cO)$ the function $u$ given by
$$
u(t,x):=\int_{\bR\times \cO}G_\cO(t,s,x,y)f(s,y)\;dyds
$$
belongs to $\mathring{\cV}^{1,0}_2(\bR\times \cO)$ and satisfies $L u=f$ in the sense of distributions on $\bR\times \cO$ (see Definition \ref{250219615}).
\end{lemma}

\begin{lemma}[Theorem 2.7 of \cite{ChoDongKim2007}]\label{250219524}
Green's function $G_{\cO}$ in Lemma \ref{241226452} satisfies the following properties:
\begin{enumerate}[label=(\roman*)]
	\item For any $\phi\in C_c^\infty(\bR\times \cO)$,  
	$$
	\int_{\bR\times \cO}\left(-G_\cO(t,s,x,y)\partial_t\phi(t,x)-\sum_{i,j=1}^d\alpha_{ij}(t)G_\cO(t,s,x,y)D_{ij}\phi(t,x) \right)dxdt=\phi(s,y)\,.
	$$
	In particular,  $L G_{\cO}(\cdot,s,\cdot,y)=0$  on $(s,\infty)\times \cO$ in the sense of distributions (see Definition \ref{250219615}).

	\item For any $\eta\in C_c^\infty(\bR\times \cO)$ satisfying $\eta\equiv 1$ on a neighborhood of $(s,y)$, we have $(1-\eta) G_\cO (\cdot,s,\cdot,y)\in \mathring{\cV}_2^{1,0}(\bR\times \cO)$.
		In particular, $G_\cO (\cdot,s,\cdot,y)\in \mathring{\cV}_2^{1,0}\big((s+\epsilon,\infty)\times \cO\big)$ for all $\epsilon>0$.
	
	\item For any $g\in L^2(\cO)$, the function $u(t,x)$ given by
	$$
	u(t,x):=\int_{\cO}G_\cO(t,s,x,y)g(y) dy\quad\forall\,\,x\in\cO\,,\,\,t>s
	$$
	is the unique weak solution in $\mathring{\cV}^{1,0}_2((s,\infty)\times \cO)$ of the Cauchy problem
	$$
	L u=0 \quad\text{on}\,\,\,\, (s,\infty)\times \cO\quad;\quad u(s,\cdot)=g
	$$
	in the sense that for any $\phi\in C_c^\infty([s,\infty)\times \cO)$ and $t>s$,
	\begin{align*}
		\begin{aligned}
	&\int_{\cO}u(t,x)\phi(t,x)\,dx+\int_{[s,t]\times \cO}\left(-u\partial_{\tau}\phi+\sum_{i,j=1}^d\alpha_{ij}D_{i}uD_{j}\phi\right)dx d\tau\\
	=\,&\int_{\cO}g(x)\phi(s,x)\,dx\,.
	\end{aligned}
	\end{align*}

\end{enumerate}
\end{lemma}

\begin{lemma}\label{250102534}
	Green's function $G_{\cO}$ in Lemma \ref{241226452} satisfies the following properties:
	
	\begin{enumerate}[label=(\roman*)]
		\item There exist constants $N,\,\sigma>0$ depending only on $d$, $\nu_1$, $\nu_2$ such that for any $s < t$ and $x, y \in \cO$,
		\begin{align}\label{250102531}
			0 \leq G_\cO(t, s, x, y) \leq N(t-s)^{-d/2}e^{-\sigma|x-y|^2/(t-s)}.
		\end{align}
		
		\item For any $s < t$ and $x, y \in \cO$, $G_\cO(t, s, x, y) = \widehat{G}_\cO(-s, -t, y, x)$, where $\widehat{G}_\cO$ is the Green's function of the operator $\tilde{L}$ on $\mathbb{R} \times \cO$.
	In particular,  $\tilde{L} G_\cO(t,-\,\cdot,x,\cdot)=0$  on $(-t,\infty)\times \cO$ in the sense of distributions (see Definition \ref{250219615}).

		\item For any $s < r < t$ and $x, y \in \cO$,
		\begin{align}\label{250114346}
			\int_\cO G_\cO(t, r, x, z)G_\cO(r, s, z, y)dz = G_\cO(t, s, x, y).
		\end{align}
	\end{enumerate}

\end{lemma}

The second inequality in \eqref{250102531} is given in \cite[Theorem 3.1]{KimXu2022}, and the claims in (ii) and (iii) are proved in \cite[(4.28) and Remark 2.12]{ChoDongKim2007}.
The first inequality in \eqref{250102531} can be derived from Lemma \ref{241227626} below.
To verify this, take $\cO_n$ as a connected component of $\cO \cap B_n(0)$ such that $\cO_n \subset \cO_{n+1}$ for all $n \in \mathbb{N}$. 
Then these domains $\cO_n$ satisfy the assumption in Lemma \ref{241227626}. 
Therefore, the first inequality in \eqref{250102531} follows directly from Lemma \ref{241227626}.

\begin{lemma}\label{241227626}
	Let $\cO_n$, $n\in N$, be bounded domains such that $\cO_1\subset \cO_2\subset \cdots$, and $\bigcup_{n\in \bN}\cO_n=\cO$.
	Then for any $s<t$, 
	$$
	0\leq G_{\cO_n}(t,s,x,y)\nearrow G_\cO(t,s,x,y)
	$$
	for almost all $(x,y)\in \cO\times \cO$.
\end{lemma}
\begin{proof}
	Recall that $\cO_n$, $n\in\bN$, are bounded.
	By the maximum principle (\cite[Theorem 7.1]{LSU_1968}),  for any $n<m$ and non-negative function $g\in C_c^{\infty}(\cO_n)$, we have
	$$
	0\leq \int_{\cO}G_{\cO_n}(t,s,x,y)g(y)dy\leq \int_{\cO}G_{\cO_{m}}(t,s,x,y)g(y)dy\,.
	$$
	This implies that $G_{\cO_n}(t,s,x,y)$ is non-negative and increasing as $n\rightarrow \infty$.
	Denote $\widetilde{G}_{\cO}(t,s,x,y)$ as the limit of $G_{\cO_n}(t,s,x,y)$ as $n\rightarrow \infty$.
	 We claim  that for any given $g\in C_c^\infty(\cO)$ with $g\geq 0$, and $t>s$,
	\begin{align}\label{241226514}
		\int_{\cO}\widetilde{G}_{\cO}(t,s,x,y)g(y)dy=:u_{\infty}(t,x)=\int_{ \cO}G_{\cO}(t,s,x,y)g(y)dy\,,
	\end{align}
	for almost all $x\in \cO$.
	\eqref{241226514}  obviously implies that $G_\cO(t,s,\cdot,\cdot)=\widetilde{G}_{\cO}(t,s,\cdot,\cdot)$ almost everywhere on $\cO\times \cO$, which completes the proof.
	
	To prove the claim, we take $n_0\in\bN$ such that $\mathrm{supp}(g)\subset \cO_{n_0}$.
	For $n\geq n_0$, put
	$$
	u_n(t,x):=\int_{\cO_n}G_{\cO_n}(t,s,x,y)g(y)dy\,.
	$$
	By the energy inequality (see, \textit{e.g.}, \cite[Theorem 2.1]{LSU_1968} or \cite[Lemma 2.7]{KimXu2022}), we have
	$$
	\|u_n\|_{\mathring{\cV}_2^{1,0}((s,\infty)\times \cO_n)}:=\sup_{t\in [s,\infty)}\|u_n(t)\|_{L_2(\cO_n)}+\|\nabla u_n\|_{L_2((s,\infty)\times \cO_n)}\leq N_0\|g\|_{L_2(\cO)}\,,
	$$
	where $N_0$ is a constant independent of $n$.
	Therefore, by the monotone convergence theorem, 
	\begin{align}\label{241227734}
		\sup_{t\in [s,\infty)}\|u_\infty(t,\cdot)\|_{L_2(\cO)}\leq N_0\|g\|_{L_2(\cO)}\,.
	\end{align}
	For each $i\in\bN$, the sequence $(u_n)_{n\in\bN}$ (each $u_n$ is extended as zero on $\cO\setminus \cO_n^c$) is bounded in the Hilbert space $\mathring{W}_2^{1,0}((s,s+i)\times \cO)$.
	By a diagonal argument, it follows that there exists a subsequence $n_k$ and a function $v\in \bigcap_{i\in\bN}\mathring{W}_2^{1,0}((s,s+i)\times \cO)$ such that, for any $i\in\bN$, $u_{n_k}$ weakly converges to $v$ in $\mathring{W}_2^{1,0}((s,s+i)\times \cO)$.
	This also implies that for any $i\in\bN$, $u_{n_k}$ weakly converges to $v$ in $L_2((s,s+i)\times \cO)$ and $\nabla u_{n_k}$ weakly converges to $\nabla v$ in $L_2((s,s+i)\times \cO;\bR^d)$.
	Hence, $v(t)=u_\infty(t)$ for almost every $t\in (s,\infty)$, so that $u_\infty\in \bigcap_{i\in\bN}\mathring{W}_2^{1,0}((s,s+i)\times \cO)$.
	Moreover,
	\begin{align*}
		\|\nabla u_\infty\|_{L_2((s,\infty)\times \cO_n)}\leq \lim_{i\rightarrow \infty}\liminf_{k\rightarrow \infty}\|\nabla u_{n_k}\|_{L_2((s,s+i)\times \cO_n)}\leq N_0\|g\|_{L_2(\cO)}\,.
	\end{align*}
	Due to \eqref{241227734} and the above, $u_\infty$ belongs to $\mathring{\cV}_2((s,\infty)\times \cO)$.
	
	Let us fix $\phi\in C_c^{\infty}([s,\infty)\times \cO)$, and take $n_1\geq {n_0}$ such that $\textrm{supp}(\phi)\subset \bR\times \cO_{n_1}$.
	Then for any $n\in\bN$ with $n\geq n_1$, and $t_1>s$,
	\begin{align}\label{241226444}
		\begin{aligned}
			&\int_{\cO_n}u_n(t_1,x)\phi(t_1,x)\,dx+\int_{[s,t_1]\times \cO_n}\left(-u_n\partial_t\phi+\alpha_{ij}D_{i}u_nD_{j}\phi\right)dxdt\\
			=\,&\int_{\cO_n}g(x)\phi(s,x)\,dx\,.
		\end{aligned}
	\end{align}
	The dominated convergence theorem ensures that \eqref{241226444} also holds with $u_\infty$ and $\cO$ in place of $u_n$ and $\cO_n$, respectively.
	Therefore, by \cite[Theorem 4.2]{LSU_1968} (whose proof is valid even if $\cO$ is unbounded), we have $u_\infty\in \mathring{\cV}_2^{1,0}([s,\infty)\times \cO)$.
	Finally, \eqref{241226514} follows from Lemma \ref{250219524}.(iii), and the proof is completed.
\end{proof}

\begin{remark}
	\label{probability}
	We now give  a probabilistic view on the Green's function on domains $\cO$ satisfying (A)-condition (see \eqref{250102420}).  We only consider the case
	$0< s<t$ with the  additional assumption that  $\alpha_{ij}$ are continuous in $t$.  For a given  $g \in C_c(\cO)$,   let $u(s,y)$ be the weak solution to 
	$$
	\frac{\partial u}{\partial s}+\sum_{i,j=1}^d\alpha_{ij}(s)u_{y^iy^j}=0, \quad s<t; \quad u(t,\cdot)=g(\cdot),\,\, u|_{(-\infty,t)\times \partial \cO}=0.
	$$
	Then $u(\cdot,\cdot)$ is continuous in $(-\infty, t]\times \bar{\cO}$,   $u(\cdot, y)$ is continuously differentiable in $ (-\infty,t)$, and  $u(s,\cdot)$ is twice continuously differentiable in $\cO$  (see. e.g. \cite[Chapter III]{LSU_1968}).
	Note that $v(s,y):=u(-s,y)$ satisfies 
	$$
	\tilde{L}v:=\frac{\partial u}{\partial s}-\sum_{i,j=1}^d\alpha_{ij}(-s)u_{y^iy^j}=0\quad;\quad \quad v(-t,\cdot)=g(\cdot),\,\, v|_{(-t,\infty)\times \partial \cO}=0.
	$$ 
	By Lemmas \ref{250219524}.(iii) and \ref{250102534}.(ii),
	\begin{align}\label{2502201146}
		u(s,y)=v(-s,y)=\int_\cO \hat{G}_\cO(-s,-t,y,x)g(x) dx=\int_\cO G_\cO(t,s,x,y)g(x) dx\,.
	\end{align}

	For   $s>0$ and $y\in \cO$,  we consider 
	$$\xi^{s,y}_{r}:=y+\int^{r}_s \sigma(u)dw_{u}, \quad r\geq s,
	$$
	where $\sigma$ is a $d\times d$ matrix satisfying  $\frac{1}{2} \sigma \sigma^{T}=(\alpha_{ij})_{1\leq i,j\leq d}$ and 
	$w_{r}=(w^1_{r}, w^2_{r}, \ldots,w^d_{r})$ is a $d$-dimensional Wiener process. We note that the $i$-th component ($i=1,\ldots,d$) of $\xi^{s,y}_r$ is $y_i+\sum_{k=1}^d \int^{r}_s \sigma_{ik}(u)dw^k_{u}$. 
	
	Define the stopping times
	$$\tilde{\tau}:=\inf\{r: (r,\xi^{s,y}_r )\not\in (0,t)\times \cO\}  
	$$
	and
	$$  \tau=\tau({s,y}):=\inf\{r \geq s: \xi^{s,y}_{r}\not\in \cO\}.
	$$
	Note $\tilde{\tau}=t\wedge \tau$. We also define  the killed process
	$$
	\xi^{\tau,s,y}_r:= \begin{cases}  \xi^{s,y}_r &: r<\tau \\
		\partial &:  r\geq \tau
	\end{cases}
	$$
	where $\partial \not\in \bR^d$ is a cemetery point and we define $g(\partial )=0$ for any function $g$. By It\^o's formula, for all $r\leq \tilde{\tau}$ (a.s.),
	$$
	u(r,\xi^{s,y}_r)=u(s,y)+\int^r_s L u (\kappa,\xi^{s,y}_{\kappa}) d\kappa + \sum_{i,k=1}^d
	\int^r_s \sigma^{ik}(\kappa)u_{y^i}(\kappa,\xi^{s,y}_{\kappa})dw^k_{\kappa}.
	$$
	Therefore, 
	\begin{eqnarray*}
		u(s,y)=\bE u(\tilde{\tau}, \xi^{s,y}_{\tilde{\tau}})
		&=&\bE u(t,\xi^{s,y}_t)1_{t< \tau} +\bE u(\tau,\xi^{s,y}_\tau)1_{t\geq  \tau} \\
		&=&\bE g(\xi^{s,y}_t)1_{t< \tau}=\bE g(\xi^{\tau,s,y}_t).
	\end{eqnarray*} 
	The third equality above is due to the fact that  $\xi^{s,y}_{\tau}\in \partial \cO$ and $u|_{\partial \cO}=0$.  By \eqref{2502201146},  for any  $g\in C_c(\cO)$, 
	$$
	u(s,y)=\bE g(\xi^{\tau,s,y}_t)=\int_{\cO} G_\cO(t,s,x,y)g(x)dx,  \quad t > s \geq 0,
	$$
	and by the dominating convergence theorem for any compact set $K\subset \cO$, we have 
	$$
	\bP ( \xi^{\tau,s,y}_t\in K)=\int_K G_{\cO} (t,s,x,y)dx.
	$$
	We finally conclude that  the probability density function of $\xi^{\tau,s,y}_t$  is equal to  $G_{\cO}(t,s,\cdot,y)$. 
\end{remark}

\mysection{Proofs of Propositions \ref{2501181139}.($i$) and \ref{250213305}.($i$)}\label{sec:appendix A}
The following definition generalizes \eqref{250102420}.
\begin{defn}\label{A}
For $\Gamma'\subset \bR\times \partial \cD$ we say that $\Gamma'$ satisfies (A)-condition if there exist positive constants $\theta_0<1$, $\rho_0$ such that
$$
|Q^{\cD}_{\rho}(\tau,\xi)|_{d+1}\le (1-\theta_0)
|Q_{\rho}(\tau,\xi)|_{d+1}
$$
for any $(\tau,\xi)\in \Gamma'$ and $\rho\le \rho_0$, where $|\cdot|_{d+1}$ denotes the Lebesgue measure on $\bR^{d+1}$.
\end{defn}

\begin{thm}\label{A2}
 Under the same conditions of  Definition \ref{def.critical lambda.edges.} the constants $\hat{\lambda}^{\pm}_{o}(\cD,\cL)$ defined in Definition \ref{def.critical lambda.edges.} (i) are strictly positive. 
\end{thm}
\begin{proof} We will utilize Theorem 10.1 in Chapter III of \cite{LSU_1968} and  we only show $\hat{\lambda}^{+}_{o}(\cD,\cL)>0$ since the other case similarly follows.

\textbf{Step 1.} First, we consider $Q_1^{\cD}(0,\bf{0})$, and $\Gamma ':=(\bR\times\partial\cD)\cap \{Q_1(0,{\bf{0}})\setminus Q_{3/4}(0,{\bf{0}})\}$. We note that $\Gamma'$ satisfies (A)-condition 
by the cone structure of our domain $\cD$. Moreover, we have $d(\Gamma', Q^{\cD}_{1/2}(0,{\bf{0}}))=3/4-1/2=1/4$.

Let $\hat{\cL}$ be {\emph{any}} operator of the form
$$
\hat{\cL} v(\tau,\zeta)=\frac{\partial v(\tau,\zeta)}{\partial \tau}-\sum^{3}_{i,j=1}\hat{a}_{ij}(t)D_{\zeta_i\zeta_j} v(\tau,\zeta)
$$
with the condition
$$
\nu_1 |\xi|^2\le \sum_{i,j=1}^3\hat{a}_{ij}(\tau)\xi_i\xi_j\le \nu_2|\xi|^2,\quad \xi\in \bR^3,
$$
where $\nu_1,\nu_2$ are the constants in the uniform parabolicity condition \eqref{uniform parabolicity}  for our fixed operator $\cL$.  We then consider functions
$v\in\mathcal{V}(Q_1^{\cD}(0,{\bf{0}}))$ satisfying $\hat{\cL} v=0$ in $Q_1^{\cD}(0,{\bf{0}})$ with $v|_{(\bR\times\partial\cD)\cap Q_1(0,{\bf{0}})}=0$ and $\sup_{Q_1^{\cD}(0,\bf{0})}|v|=1$. Note that any such $v$ belongs to $H^{\alpha/2,\alpha}(\Gamma')$ for any $\alpha\in(0,1)$ since $v=0$ on $\Gamma '$. 

By Theorem 10.1 in Chapter III of \cite{LSU_1968},   there is a constant $\lambda_1>0$ depending only on $\nu_1,\nu_2, \rho_0$ and $\theta_0$  such that 
$$|v|_{H^{\lambda_1/2,\lambda_1}(Q^{\cD}_{1/2}(0,{\bf{0}}))}\leq N(\nu_1,\nu_2, \rho_0,\theta_0),
$$
 where $\rho_0,\theta_0$ are from  Definition \ref{A} and chosen for our $\Gamma'$.
In particular, there exists a constant $N=N(\nu_1,\nu_2,\rho_0,\theta_0)$ such that
$$
\sup_{(\tau,\zeta)\in Q^{\cD}_{1/2}(0,{\bf{0}})}\frac{|v(\tau,\zeta)-v(\tau,{\bf{0}})|}{|\zeta-{\bf{0}}|^{\lambda_1}}\le N
$$
or as $v(\tau,{\bf{0}})=0$
$$
|v(\tau,\zeta)|\le N|\zeta|^{\lambda_1},\quad (\tau,\zeta)\in Q^{\cD}_{1/2}(0,{\bf{0}}).
$$
 
\textbf{Step 2.} If we have the same situation of Step 1 except that the condition $\sup_{Q_1^{\cD}(0,\bf{0})}|v|=1$ is replaced by $M:=\sup_{Q_1^{\cD}(0,\bf{0})}|v|>0$, we consider  the function $\frac{1}{M} v$ to apply Step 1 and have
$$
|v(\tau,\zeta)|\le N|\zeta|^{\lambda_1}\sup_{Q_1^{\cD}(0,\bf{0})}|v|,\quad (\tau,\zeta)\in Q^{\cD}_{1/2}(0,{\bf{0}})
$$
; note that this estimate trivially holds when $M=0$.

\textbf{Step 3.} Now, we consider any  $t_0\in\bR$,  $r>0$, $Q_r^{\cD}(t_0,\bf{0})$,  and  functions
$u\in\mathcal{V}(Q_r^{\cD}(t_0,\bf{0}))$ satisfying $\cL u=0$ in $Q_r^{\cD}(t_0,\bf{0})$ with $u|_{(\bR\times\partial\cD)\cap Q_r(t_0,\bf{0})}=0$.
Then we define  the function $v(\tau,\zeta)=u(t,x)$ with the relations $t=r^2\tau+t_0$, $x=r\zeta$ so that $v$ is defined on $Q^{\cD}_1(0,\bf{0})$ and satisfies $v\in\mathcal{V}(Q_1^{\cD}(0,\bf{0}))$ and $\hat{\cL} v=0$ in $Q_1^{\cD}(0,\bf{0})$ with $v|_{(\bR\times\partial\cD)\cap Q_1(0,\bf{0})}=0$, where $\hat{\cL}$ is an operator for which we can apply the result of Steps 1 and 2 since
$$
\hat{a}_{ij}(\tau)=a_{ij}(r^2\tau+t_0),\quad i,j=1,2,3,
$$
where $a_{ij}(t)s$ are the diffusion coefficients in our fixed operator $\cL$. Hence, we have
\begin{align*}
|u(t,x)|=\left|v\left(\frac{t-t_0}{r^2},\frac{x}{r}\right)\right| &\le N\left|\frac{x}{r}\right|^{\lambda_1}\sup_{Q_1^{\cD}(0,\bf{0})}|v|\\
&= N\left(\frac{d(x,V)}{r}\right)^{\lambda_1}\sup_{Q_r^{\cD}(t_0,\bf{0})}|u|,\quad (t,x)\in Q^{\cD}_{r/2}(t_0,{\bf{0}}),
\end{align*}
where $N$ is the constant from Steps 1 and 2.

\textbf{Step 4.} The positive constant $\lambda_1$ we found in the previous steps serves as one of $\lambda$s mentioned in Definition \ref{def.critical lambda.edges.} (i). Hence, as the supremum of them,
$\hat{\lambda}^{+}_{o}(\cD,\cL)$ is strictly positive.
\end{proof}

\begin{thm} \label{A3}
Under the same conditions of  Definition \ref{def.critical lambda.edges.} the constants $\hat{\lambda}^{\pm}_{e,i}(\cD,\cL)$, $i=1,2,\ldots,N_0$, defined in Definition \ref{def.critical lambda.edges.}  (ii) are strictly positive. 
\end{thm}
\begin{proof}
We again use Theorem 10.1 in chapter III of \cite{LSU_1968} and  we only show $\hat{\lambda}^{+}_{e,i}(\cD,\cL)>0$. 

\textbf{Step 1.} First, we consider $Q_{r_i}^{\cD}(0,p_i)$, where $p_i$ is the $i$th vertex of the manifold $\cM$; $|p_i|=1$ and $r_i:=\min_{j\ne i}d(p_i,E_j)$. Note that $r_i$ is the largest $r$ satisfying $B_{r}(\xi_0)$ does not intersect the vertex $V$ and all edges except $E_i$.

 Then we consider $\Gamma ':=(\bR\times\partial\cD)\cap \{Q_{r_i}(0,p_i)\setminus Q_{3r_i/4}(0,p_i)\}$ and note that  $\Gamma'$ satisfies (A)-condition 
by the cone structure of our domain $\cD$ with $d(\Gamma', Q^{\cD}_{r_i/2}(0,p_i))=r_i/4$. 

Next we consider functions
$v\in\mathcal{V}(Q_{r_i}^{\cD}(0,p_i))$ satisfying $\hat{\cL} v=0$ in $Q_{r_i}^{\cD}(0,p_i)$ with $v|_{(\bR\times\partial\cD)\cap Q_{r_i}(0,p_i)}=0$ and $\sup_{Q_{r_i}^{\cD}(0,p_i)}|v|=1$, where $\hat{\cL}$ is {\emph{any}} operator explained in Step 1 of the proof of Theorem \ref{A2}. Again, note that any such $v$ belongs to $H^{\alpha/2,\alpha}(\Gamma')$ for any $\alpha\in(0,1)$ since $v=0$ on $\Gamma '$. 

Hence, by Theorem 10.1 in chapter III of \cite{LSU_1968}  there is a constant $\lambda_1>0$ depending only on $\nu_1,\nu_2, r_0, \rho_0$ and $\theta_0$   such that 
$$
|v|_{H^{\lambda_1/2,\lambda_1}(Q^{\cD}_{r_i/2}(0,p_i))}\leq N(\nu_1,\nu_2, r_0, \rho_0,\theta_0),
$$
 where $\rho_0,\theta_0$ are from  Definition \ref{A} and chosen for $\Gamma'$; note that $\theta_0$ depends on $\kappa_i$ and we will emphasize it.
In particular, there exists a constant $N$$=$$N(\nu_1,\nu_2,r_0,\rho_0,\theta_0, \kappa_i)$ such that
$$
\sup_{(\tau,\zeta)\in Q^{\cD}_{r_i/2}(0,p_i)}\frac{|v(\tau,\zeta)-v(\tau,\xi_i(\zeta))|}{|\zeta-\xi_i(\zeta)|^{\lambda_1}}\le N
$$
holds, where $\xi_i(\zeta)$ is on $E_i$ satisfying $|\zeta-\xi_i(\zeta)|=d(\zeta, E_i)$; we note
$|\xi_i(\zeta)-p_i|^2+|\zeta-\xi_i(\zeta)|^2=|\zeta-p_i|^2$ and hence $\xi_i(\zeta)\in B^{\cD}_{r_i/2}(p_i)$. Now, as $v(\tau,\xi_i(\zeta))=0$ for any $(\tau,\zeta)\in Q^{\cD}_{r_i/2}(0,p_i)$, we have
$$
|v(\tau,\zeta)|\le N \;d(\zeta,E_i)^{\lambda_1},\quad (\tau,\zeta)\in Q^{\cD}_{r_i/2}(0,p_i).
$$

\textbf{Step 2.} As in Step 2 of the proof of Theorem \ref{A2},  we have 
$$
|v(\tau,\zeta)|\le N \;d(\zeta,E_i)^{\lambda_1}\sup_{Q_{r_i}^{\cD}(0,p_i)}|v|,\quad (\tau,\zeta)\in Q^{\cD}_{r_i/2}(0,p_i)
$$
with $N$ in Step 1.

\textbf{Step 3.} Now, we consider any $t'_0\in\bR$,  $r'\in(0,r_i)$, and $Q_{r'}^{\cD}(t'_0,p_i)$. 

If  $w\in\mathcal{V}(Q_{r'}^{\cD}(t'_0,p_i))$ satisfies $\overline{\cL} w=0$ in $Q_{r'}^{\cD}(t'_0,p_i)$ with $w|_{(\bR\times\partial\cD)\cap Q_{r'}(t'_0,p_i)}=0$, where $\bar{\cL}$ is one of any operator $\hat{\cL}$ described in Step 1, then we define  the function $v(\tau,\zeta)=w(\sigma,\eta)$ with the relations $\sigma=\frac{(r')^2}{r_i^2}\tau+t'_0$, $\eta=\frac{r'}{r_i}(\zeta-p_i)+p_i=\frac{r'}{r_i}\zeta+\left(1-\frac{r'}{r_i}\right)p_i$ so that $v$ is defined on $Q^{\cD}_{r_i}(0,p_i)$ and satisfies $v\in\mathcal{V}(Q_{r_i}^{\cD}(0,p_i))$ and $\hat{\cL} v=0$ in $Q_{r_i}^{\cD}(0,p_i)$ with $v|_{(\bR\times\partial\cD)\cap Q_{r_i}(0,p_i)}=0$, where $\hat{\cL}$ is an operator for which we can apply the result of Steps 1 and 2 since
$$
\hat{a}_{ij}(\tau)=\overline{a}_{ij}\left(\frac{(r')^2}{r_i^2}\tau+t'_0\right),\quad i,j=1,2,3,
$$
where $\overline{a}_{ij}(\sigma)s$ are the diffusion coefficients in the operator $\overline{\cL}$.

Hence, noting that the relation $\eta=\frac{r'}{r_i}(\zeta-p_i)+p_i$ between $\eta$ and $\zeta$ gives 
$$
\zeta-p_i=\frac{r_i}{r'}(\eta-p_i),\quad
d(\zeta,E_i)=\frac{r_i}{r'} d(\eta,E_i),
$$
we have
\begin{align*}
|w(t,\eta)|&=\left|v\left(\frac{r_i^2}{(r')^2}(t-t'_0),\frac{r_i}{r'}(\eta-p_i)+p_i\right)\right|
\\ &\le N\;  d\left(\frac{r_i}{r'}(\eta-p_i)+p_i,E_i\right)^{\lambda_1}\sup_{Q_{r_i}^{\cD}(0,p_i)}|v|\\
&= N\left(\frac{r_i}{r'}d(\eta,E_i)\right)^{\lambda_1}\sup_{Q_{r'}^{\cD}(t'_0,p_i)}|u|,\quad (t,\eta)\in Q^{\cD}_{r'/2}(t'_0,p_i),
\end{align*}
where $N$ is the constant from Steps 1 and 2.

\textbf{Step 4.} Now, we consider any $t_0\in\bR$, $\xi_0\in E_i$, $r>0$ with which $B_r(\xi_0)$ does not intersect the vertex $V$ and all edges except $E_i$, 
and consider $Q_r^{\cD}(t_0,\xi_0)$. We note that $r\in (0,|\xi_0|r_i)$ since $\xi_0=|\xi_0|p_i$.

If $u\in\mathcal{V}(Q_r^{\cD}(t_0,\xi_0))$ satisfies $\cL u=0$ in $Q_r^{\cD}(t_0,\xi_0)$ with $u|_{(\bR\times\partial\cD)\cap Q_r(t_0,\xi_0)}=0$, then we define  the function $w(\sigma,\eta)=u(t,x)$ with the relations, $t=|\xi_0|^2\sigma$, $x=|\xi_0|\eta$ so that $w$ is defined on $Q^{\cD}_{r'}(t'_0,p_i)$, where $r'=r/|\xi_0|$, $t'_0=t_0/|\xi_0|^2$ and satisfies $w\in\mathcal{V}(Q_{r'}^{\cD}(t'_0,p_i))$ and $\overline{\cL} v=0$ in $Q_{r'}^{\cD}(t'_0,p_i)$ with $v|_{(\bR\times\partial\cD)\cap Q_{r'}^{\cD}(t'_0,p_i))}=0$, where $\overline{\cL}$ is an operator for which we can apply the result of Step 3 since
$$
\overline{a}_{ij}(\sigma)=a_{ij}(|\xi_0|^2\sigma),\quad i,j=1,2,3,
$$
where $a_{ij}(t)s$ are the diffusion coefficients in our fixed operator $\cL$. Hence, we have
\begin{align*}
|u(t,x)|&=\left|w\left(\frac{t}{|\xi_0|^2},\frac{x}{|\xi_0|}\right)\right|
\\ &\le N\;  \left(\frac{r_i}{r'}d\left(\frac{x}{|\xi_0|},E_i\right)\right)^{\lambda_1}\sup_{Q_{r'}^{\cD}(t'_0,p_i)}|v|\\
&=N\;  \left(\frac{r_i}{r'|\xi_0|}d\left(x,E_i\right)\right)^{\lambda_1}\sup_{Q_{r'}^{\cD}(t'_0,p_i)}|v|\\
&= N\;r_i^{\lambda_1}\left(\frac{d(x,E_i)}{r}\right)^{\lambda_1}\sup_{Q_r^{\cD}(t_0,\xi_0)}|u|,\quad (t,x)\in Q^{\cD}_{r/2}(t_0,\xi_0),
\end{align*}
where $N$ is the constant from Steps 1 and 2.

\textbf{Step 5.} The positive constant $\lambda_1$ we found in the previous steps serves as one of $\lambda$s mentioned in Definition \ref{def.critical lambda.edges.} (ii) and hence the supremum of them
$\hat{\lambda}^{+}_{e,i}(\cD,\cL)$ is strictly positive.
\end{proof}

The following lemma is essential in the proof of  Theorem \ref{A6}.
\begin{lemma}\label{A4}
Let $\mu, \epsilon_1, \epsilon_2$ be  constants satisfying $\mu^2<\frac{\nu_1}{\nu_2}\left(\mathcal{E}_0+\frac{1}{4}\right)$ and  $0<\epsilon_1<\epsilon_2\le 1$.  Then 
there exists a constant $N$  depending only on $\mu,\epsilon_1,\epsilon_2, \nu_1, \nu_2, \cM$ such that the inequality
$$
\int_{Q^{\mathcal{D}}_{\epsilon_1 r}(t_0,{\bf{0}})} |x|^{2\mu} |\nabla u|^2 dxdt+\int_{Q^{\mathcal{D}}_{\epsilon_1 r}(t_0,{\bf{0}})} |x|^{2\mu-2} |u|^2 dxdt \le N  r^{2\mu-2} \int_{Q^{\mathcal{D}}_{\epsilon_2 r}(t_0,{\bf{0}})} |u|^2 dxdt
$$
 holds  for any  $r>0$ and any function $u$ belonging to $\mathcal{V}_{loc}(Q^{\mathcal{D}}_r(t_0,{\bf{0}}))$ and satisfying 
$\mathcal{L}u=0$  in  $Q^{\mathcal{D}}_r(t_0,{\bf{0}})$, $u=0$ on $\mathbb{R}\times \partial\mathcal{D}$.
\end{lemma}

\begin{proof}
The proof  is almost the  same as that of Lemma A.1 of \cite{Kozlov_Nazarov_2014} and so we omit it.  The only difference is that we use $\nu_1,\nu_2$ instead of $\nu,\nu^{-1}$. \end{proof}

\begin{remark}
 In the expression $\mu^2<\frac{\nu_1}{\nu_2}\left(\mathcal{E}_0+\frac{1}{4}\right)$   the number 1 means $(d-2)^2$; our dimension is 3. 
 It is worth mentioning the reason for the appearance of the value $\mathcal{E}_0+\frac{(d-2)^2}{4}$. It is due to the necessary computation of the type $\int_{\cD}|\nabla f(x)|^2dx$ in relation with $\int_{\cD}|x|^{-2}f^2dx$ for the functions $f$ satisfying $f|_{\partial \cD}=0$. As our domain $\cD$ is an infinite cone, the computation of $\int_{\cD}|\nabla f(x)|^2dx$ involves the surface integration on $\cM$ and the integration on the radial direction. Hence, $\mathcal{E}_0$, the first or the smallest eigenvalue of the Laplace-Beltrami operator on $\cM$ with zero Dirichlet boundary condition on $\partial \cM$, appears; see the definition of $\mathcal{E}_0$, \eqref{250121350}. On the other hand, on radial direction we use a version of the one dimensional Hardy inequality and this indeed gives $\frac14=\frac{(d-2)^2}{4}$. Combining two computations, we have
$$
\int_{\cD}|\nabla f(x)|^2dx \ge \left(\mathcal{E}_0+\frac14\right)\int_{\cD}|x|^{-2}f^2dx
$$
in our hands and utilize it when we prove Lemma \ref{A4}.
\end{remark}

\begin{thm}\label{A6}
 Under the same conditions of  Definition \ref{def.critical lambda.edges.} the constants $\hat{\lambda}^{\pm}_{o}(\cD,\cL)$ defined in Definition \ref{def.critical lambda.edges.} (i) satisfy
\begin{equation}\label{202502081324}
\hat{\lambda}^{\pm}_{o}(\cD,\cL)\geq -\frac{3}{2}+\sqrt{\frac{\nu_1}{\nu_2}}\sqrt{\mathcal{E}_0+\frac{1}{4}}\,.
\end{equation}
\end{thm}

\begin{proof}
In \eqref{202502081324}  the number 3 represents our dimension 3. Below we just show \eqref{202502081324} for $\hat{\lambda}^{+}_{o}$; the case of $\hat{\lambda}^{-}_{o}$ can be similarly done.

{\textbf{Step 1.}} We show that  for any $\mu\in\mathbb{R}$  satisfying
$\mu^2<\frac{\nu_1}{\nu_2}\left(\mathcal{E}_0+\frac{1}{4}\right)$, there exists a constant $N$ depending only on $\mathcal{M}, \nu_1, \nu_2, \mu$ - or only on $\cD,\cL,\mu$ - such that
\begin{equation*}
|u(t,x)|\le N \left(\frac{|x|}{r}\right)^{-\frac{3}{2}-\mu}\sup_{Q^{\mathcal{D}}_{r}(t_0,\bf{0})}\ |u|,
\quad
\forall \;(t,x)\in Q^{\mathcal{D}}_{r/2}(t_0,\bf{0})
\end{equation*}
  for any $t_0>0$, $r>0$, and $u$  belonging to $\mathcal{V}_{loc}(Q^{\mathcal{D}}_r(t_0,\bf{0}))$ and satisfying 
\begin{equation*}
\mathcal{L}u=0\quad \text{in}\; Q^{\mathcal{D}}_r(t_0,{\bf{0}})\quad ; \;\quad
u(t,x)=0\quad\text{for}\;\; x\in\partial\mathcal{D}.
\end{equation*}
Then considering $\mu$ close to the value $-\sqrt{\frac{\nu_1}{\nu_2}}\sqrt{\mathcal{E}_0+\frac{1}{4}}$ and noting $\hat{\lambda}^{+}_{o}\ge -\frac32-\mu$ for such $\mu$,  we see that \eqref{202502081324} holds for $\hat{\lambda}^{+}_{o}$.

Below we note that we may assume $t_0=0$.

{\textbf{Step 2.}} Take any function $u$ satisfying the conditions in Step 1 with $t_0=0$ and take any $(t,x)\in Q^{\mathcal{D}}_{r/2}(0,\bf{0})$. Let us denote  
$$\rho=|x| \,\Big(<\frac{r}{2}\,\Big), \quad \mathcal{A}^{\cD}_{\rho}=(t-\rho^2/4,t]\times ((B_{\frac{3}{2}\rho}(0)\setminus B_{\frac12 \rho}(0))\cap \cD)\subset Q^{\cD}_{\frac34 r}(0,{\bf{0}}).
$$ 
Then as in the proof of \cite[Theorem 2.4.7]{Kozlov_Nazarov_2014}, we have  
\begin{eqnarray*}
|u(t,x)|^2&\le& N \rho^{-3-2}\int_{\mathcal{A}^{\cD}_{\rho}} |u(\tau,y)|^2dyd\tau\nonumber\\
&\le&N \rho^{-3-2\mu}\int_{\mathcal{A}^{\cD}_{\rho}}|y|^{2\mu-2} |u(\tau,y)|^2dyd\tau
\end{eqnarray*}
and by Lemma \ref{A4} with $\epsilon_1=\frac34, \epsilon_2=1$ applied we have
\begin{eqnarray}
|u(t,x)|^2 &\le& N \rho^{-3-2\mu}r^{2\mu-2}\int_{Q^{\cD}_{r}(0,{\bf{0}})} |u(\tau,y)|^2dyd\tau\nonumber\\
&\le& N \rho^{-3-2\mu}r^{2\mu-2} \cdot |Q^{\cD}_{r}(0,{\bf{0}})|_{3+1}\cdot \sup_{Q^{\cD}_{r}(0,{\bf{0}})} |u|^2\nonumber\\
&=&N \rho^{-3-2\mu}r^{2\mu-2} \cdot r^{3+2}\cdot \sup_{Q^{\cD}_{r}(0,{\bf{0}})} |u|^2.\nonumber
\end{eqnarray}
All the constants $N$ in Step 2 are not identical, but depend only on $\mathcal{M}, \nu_1, \nu_2, \mu$. Hence, the claim in Step 1 holds.

\end{proof}

\end{document}